\documentclass[leqno,11pt,a4paper]{amsart}

\title{Legendrian Embedded Contact Homology}
\author{Julian Chaidez, Oliver Edtmair, Luya Wang, Yuan Yao, Ziwen Zhao}
\usepackage[full]{textcomp}
\usepackage{color,soul}
\usepackage{newpxtext}
\usepackage{cabin} 
\usepackage[varqu,varl]{inconsolata} 
\usepackage[bigdelims,vvarbb]{newpxmath} 
\usepackage[cal=cm,bb=esstix,bbscaled=1.05]{mathalfa} 
\usepackage[margin=2.4cm]{geometry}
\usepackage{savesym}

\usepackage{mathabx}
\usepackage{hyperref}
\hypersetup{
 colorlinks,
 linkcolor={red!50!black},
 citecolor={blue!50!black},
 urlcolor={blue!80!black}
}
\usepackage{tikz}
\usepackage{tikz-cd}
\usepackage{graphicx}
\usepackage{subcaption}
\usepackage{mathtools}
\usetikzlibrary{cd}
\usepackage{graphicx}
\usepackage[all,cmtip]{xy}
\usepackage{mleftright}
\usepackage{booktabs}
\usepackage{enumitem}
\usepackage{mathdots}
\usepackage{eucal}
\usepackage{enumitem}

\usepackage{biblatex} 
\setlist[enumerate]{labelsep=*, leftmargin=1.5pc}
\setlist[enumerate]{label=\normalfont(\roman*), ref=\roman*}
\newtheorem{theorem}{Theorem}[section]
\newtheorem{lemma}[theorem]{Lemma}
\newtheorem{corollary}[theorem]{Corollary}
\newtheorem{proposition}[theorem]{Proposition}
\newtheorem{claim}[theorem]{Claim}
\newtheorem{conjecture}[theorem]{Conjecture}

\newtheorem{theorem*}[intro]{Theorem}
\newtheorem{corollary*}[intro]{Corollary}

\theoremstyle{definition}
\newtheorem{example}[theorem]{Example}
\newtheorem{remark}[theorem]{Remark}
\newtheorem{definition}[theorem]{Definition}

\newtheorem*{acknowledgements*}{Acknowledgements}
\newtheorem{question}[theorem]{Question}

\newtheorem{setup}[theorem]{Setup}
\numberwithin{equation}{section}
\newtheorem{definition*}[intro]{Definition}
\newcommand{\la}{\langle}
\newcommand{\ra}{\rangle}

\definecolor{dark green}{rgb}{0.0, 0.6, 0.0}

\graphicspath{{images/}}

\newcommand{\C}{{\mathbb C}}

\newcommand{\R}{{\mathbb R}}
\newcommand{\D}{{\mathbb D}}
\newcommand{\N}{{\mathbb N}}
\newcommand{\Z}{{\mathbb Z}}
\newcommand{\Sp}{\op{Sp}}

\newcommand{\op}{\operatorname}

\newcommand{\U}{\op{U}}

\newcommand{\SW}{\op{SW}}

\newcommand{\CZ}{\op{CZ}}

\newcommand{\ind}{\op{ind}}

\newcommand{\bpm}{\begin{pmatrix}}
\newcommand{\epm}{\end{pmatrix}}

\newcommand{\lw}[1]{\textbf{\textcolor{pink}{#1}}}

\newcommand{\yy}[1]{\textbf{\textcolor{blue}{#1}}}

\setlist{itemsep=4pt}

\begin{document}

\maketitle
\begin{abstract} We give a construction of embedded contact homology (ECH) for a contact $3$-manifold $Y$ with convex sutured boundary and a pair of Legendrians $\Lambda_+$ and $\Lambda_-$ contained in $\partial Y$ satisfying an exactness condition. The chain complex is generated by certain configurations of closed Reeb orbits of $Y$ and Reeb chords of $\Lambda_+$ to $\Lambda_-$. The main ingredients include
\vspace{3pt}
\begin{itemize}
\item a general Legendrian adjunction formula for curves in $\R \times Y$ with boundary on $\R \times \Lambda$.
\item a relative writhe bound for curves in contact $3$-manifolds asymptotic to Reeb chords.
\item a Legendrian ECH index and an accompanying ECH index inequality.
\vspace{3pt}
\end{itemize}
The (action filtered) Legendrian ECH of any pair $(Y,\Lambda)$ of a closed contact $3$-manifold $Y$ and a Legendrian link $\Lambda$ can also be defined using this machinery after passing to a sutured link complement. This work builds on ideas present in Colin-Ghiggini-Honda's proof of the equivalence of Heegaard-Floer homology and ECH. 

\vspace{3pt}

The independence of our construction of choices of almost complex structure and contact form should require a new flavor of monopole Floer homology. It is beyond the scope of this paper. \end{abstract}

\tableofcontents 

\section{Introduction} The embedded contact homology of a closed contact $3$-manifold $(Y,\xi)$ is a Floer theoretic invariant originally introduced by Hutchings \cite{Hutchings_lectures_on_ECH}. These homology groups, denoted by
\[ECH(Y,\xi)\]
are, roughly speaking, computed as the homology of a chain complex freely generated by certain finite sets of closed simple Reeb orbits (with multiplicity). The differential counts certain embedded, possibly disconnected $J$-holomorphic curves in the symplectization of the contact manifold that can have arbitrary genus. 

\vspace{3pt}

Since its inception, embedded contact homology and its associated invariants have been applied to many of the fundamental problems in 3-dimensional Reeb dynamics, to dramatic effect. These applications include a simple formal proof of the 3-dimensional Weinstein conjecture by \cite{taubes_weinstein}, and the Arnold chord conjecture by Hutchings-Taubes \cite{ht2011,ht2013}, and the existence of broken book decompositions \cite{cdr2020}. The spectral invariants derived from the ECH groups (and related PFH groups), called the ECH capacities, have been used to prove the smooth closing lemma for Reeb flows \cite{i2015} and area preserving surface maps \cite{cpz2021,eh2021}, and many results on symplectic embeddings \cite{m2011,h2011,cg2019}. 

\vspace{3pt}

Embedded contact homology can morally be regarded as a flavor of symplectic field theory (SFT), as formulated by Eliashberg-Givental-Hofer \cite{egh2000}. SFT is, broadly speaking, a framework for constructing invariants of contact manifolds in any dimension, acquired from chain complexes generated by Reeb orbits. Variants of SFT include cylindrical contact homology (cf. \cite{hn2014,n2015}), linearized contact homology (cf. \cite{b2017}) and the contact homology algebra (cf. \cite{p2019}). 

\vspace{3pt}

Many flavors of SFT can be extended to invariants of a pair $(Y,\Lambda)$ of a closed contact manifold $Y$ and a closed Legendrian sub-manifold $\Lambda \subset Y$. The goal of this paper is to initiate the study of a corresponding Legendrian version of embedded contact homology.

\subsection{Standard ECH} We begin this introduction with a review of standard ECH, largely drawing on \cite{Hutchings_lectures_on_ECH}. We discuss many of the key ideas that allow one to define ECH and to show that the differential in ECH, which counts certain $J$-holomorphic curves, defines a chain complex. Besides serving as a review, this will also provide a road map of the results required to construct Legendrian ECH.

\subsubsection{Holomorphic Currents} \label{subsubsec:holomorphic_currents} Let $(Y,\lambda)$ be a closed contact 3-manifold with a non-degenerate contact form $\lambda$. We start by considering the symplectization of $Y$, i.e. the cylindrical manifold
\[
\R_s \times Y
\]
There is a natural class of translation invariant complex structures $J$ on $\R \times Y$, which sends $\partial_s$ to the Reeb vector-field $R$ of $\lambda$ and positively preserves the contact structure $\ker\lambda$. We call such almost complex structure \emph{$\lambda$-adapted}.

\vspace{3pt}

Recall that a \emph{$J$-holomorphic curve} $C$ from a punctured Riemann surface (without boundary) is an equivalence class of map
\[
u:(\Sigma,j) \to (\R \times Y,J) \qquad\text{satisfying}\qquad du \circ j = J \circ du
\]
modulo holomorphic reparametrization of the domain $(\Sigma,j)$. We say that $C$ is proper if the map $u$ is proper and finite energy if the integral of $u^*d\lambda$ is finite. If $u$ is proper and finite energy, then $u$ converges (in an appropriate sense) to the cylinder over a collection of Reeb orbits
\[\Gamma_+ = (\gamma_1^+,\dots,\gamma_k^+) \qquad\text{and}\qquad \Gamma_- = (\gamma_1^-,\dots,\gamma_k^-)\]
near $+\infty \times Y$ and $-\infty \times Y$, respectively. The \emph{Fredholm index} $\ind(u)$ of $u$ is the Fredholm index of a certain linearized Cauchy-Riemann operator associated to $u$. It is given by the formula
\[\ind(u) = -\chi(\Sigma) + 2c_\tau(u) + \CZ_\tau^{ind}(\Gamma_+) - \CZ_\tau^{ind}(\Gamma_-) \qquad\text{where}\qquad \CZ_\tau^{ind}(\Gamma_\pm) := \sum_i CZ_{\tau}(\gamma_i^\pm)\]
Here $c_\tau(u)$ is the relative first Chern number of $u^*\xi$ with respect to a trivialization $\tau$ of $\xi$ along $\Gamma_+$ and $\Gamma_-$, and  $CZ_{\tau}(\gamma)$ is the Conley-Zehnder index of the linearized Reeb flow around $\gamma$ in the trivialization $\tau$. Standard transversality result states that there is a generic class of \emph{regular} $J$ such that, if $u$ is somewhere injective, the moduli space of proper, finite energy, somewhere injective $J$-holomorphic maps $v$ near $u$ asymptotic to $\Gamma_\pm$ at $\pm\infty$ is a manifold of dimension $\ind(u)$. 

\vspace{3pt}

A \emph{$J$-holomorphic current} $\mathcal{C} = \{(C_i,m_i)\}$ in $Y$ is a finite collection of connected, proper, finite energy, somewhere injective $J$-holomorphic curves $C_i$ and positive integer multiplicities $m_i$.  

\begin{figure}[h!]
    \centering
    \includegraphics[width=.8\textwidth]{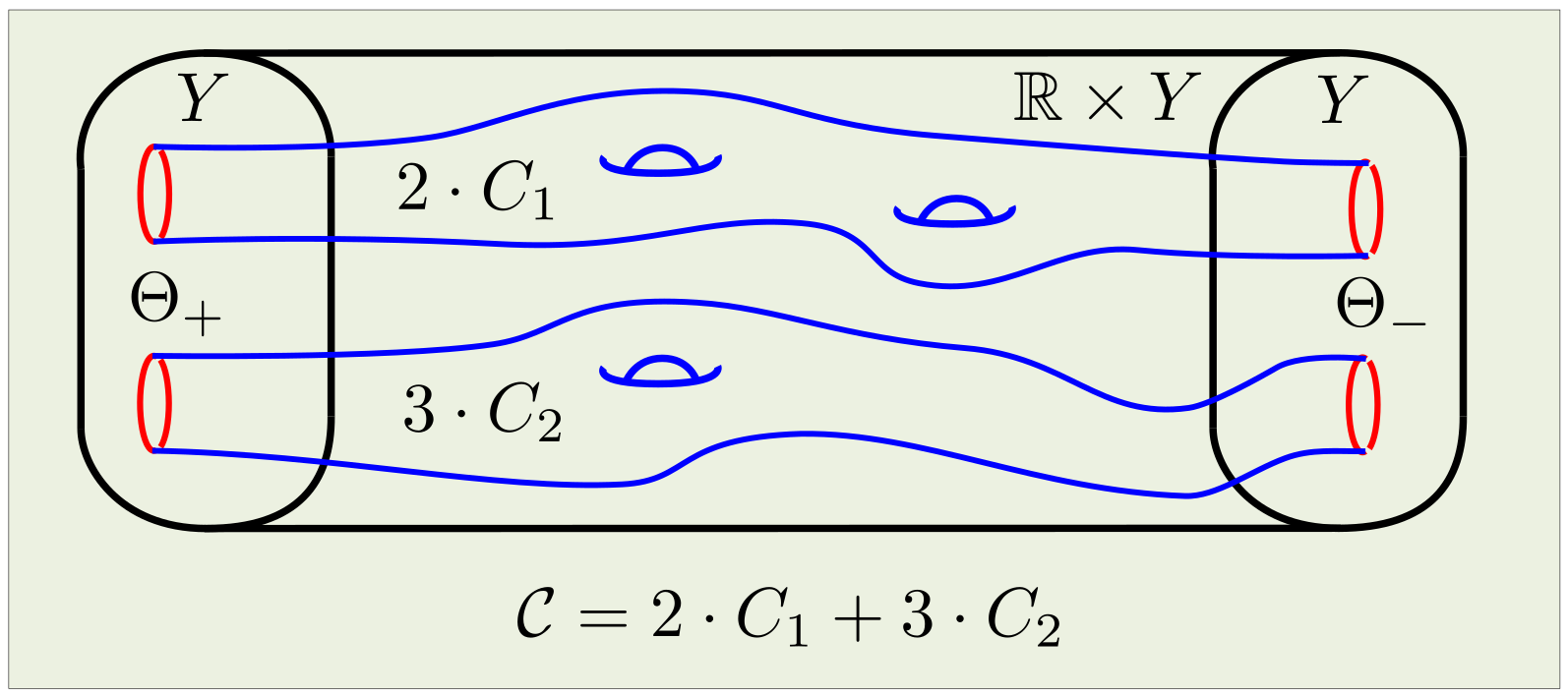}
    \caption{An depiction of a $J$-holomorphic current.}
    \label{fig:holomorphic_current}
\end{figure}
\noindent An \emph{orbit set} $\Theta = \{(\gamma_i,n_i)\}$ is, similarly, a collection of simple closed Reeb orbits $\gamma_i$ at $\pm \infty$. In analogy to curves, every $J$-holomorphic current is asymptotic orbit sets at $\pm \infty$. We note here we must count with multiplicities, e.g. if the $J$-holomorphic current $(C,1)$ is positively asymptotic to the simple Reeb orbit $\gamma$, then $(C,2)$ is positively asymptotic to $(\gamma,2)$. We denote the moduli space of $J$-holomorphic currents asymptotic to $\alpha$ at $+\infty$ and $\beta$ at $-\infty$ by
\[\mathcal{M}(\alpha,\beta)\]
Note that this moduli space admits an $\R$-action given by $\R$-translation in $\R \times Y$. We denote the quotient by $\mathcal{M}(\alpha,\beta)/\R$. 

\subsubsection{ECH Index} The key ingredient of ECH that differentiates it from other versions of symplectic field theory is the \emph{ECH index}. It may be viewed as a map
\[I:\mathcal{M}(\alpha,\beta) \to \Z\]
On a $J$-holomorphic current $\mathcal{C}$ asymptotic to $\alpha$ at $+\infty$ and $\beta$ at $-\infty$, the ECH index is given by
\begin{equation} \label{eqn:ECH_index}
I(\mathcal{C}) = c_\tau(\mathcal{C}) + Q_\tau(\mathcal{C}) + \CZ^{ECH}_{\tau}(\alpha)-\CZ^{ECH}_{\tau}(\beta).
\end{equation}
Here $Q_\tau$ is the relative self-intersection number \cite[\S 2.4]{Hutchings2002} that counts intersections between $\mathcal{C}$  and a push-off of $\mathcal{C}$ determined by $\tau$, and $\CZ^{ECH}_{\tau}$ is a Conley-Zehnder index term given by
\begin{equation} \label{eqn:ECH_CZ_intro}
\CZ^{ECH}_\tau(\Theta) = \sum_{i=1}^k \sum_{j = 1}^{m_i} \CZ(\gamma_i^j) \qquad\text{for}\qquad \Theta = \{(\gamma_i,m_i)\}
\end{equation}
The fundamental property that the ECH index satisfies is the following index inequality. 

\vspace{3pt}

\begin{theorem*}[Index Inequality] \cite{Hutchings2002,Hutchings_index_revisited} \label{thm:intro_standard_index_inequality} Let $C$ is a somewhere injective $J$-holomorphic curve in $\R \times Y$, for a compatible $J$ and let $\delta(C)$ denote the count of singularities (with multiplicity) of $C$. Then
\[
\ind(C) \le I(C) - 2 \delta(C)
\]
\end{theorem*}
\noindent This inequality places stringent constraints on curves and currents of low ECH index. Let us discuss the main ingredients of the proof of this inequality, as their Legendrian analogues will be the main topic of this paper.

\vspace{3pt}

The first ingredient of Theorem \ref{thm:intro_standard_index_inequality} is the writhe bound. Recall that the \emph{writhe} $w(\zeta)$ of a braid $\zeta$ in $S^1 \times D^2$ is an isotopy invariant that can be computed as a signed count of the self-intersections of the image $\pi(\zeta)$ under the projection $\pi:S^1 \times D^2 \to S^1 \times [-1,1]$. If $C$ is a somewhere injective J-holomorphic curve asymptotic to (covers of) a simple orbit $\gamma$ at some subset of its punctures, then $\mathcal{C}$ determines a braid in a tubular neighborhood of $\gamma$ (due to $J$-holomorphic curve asymptotics established by Siefring \cite{siefring2008relative}).

\begin{figure}[h!]
    \centering
    \includegraphics[width=.8\textwidth]{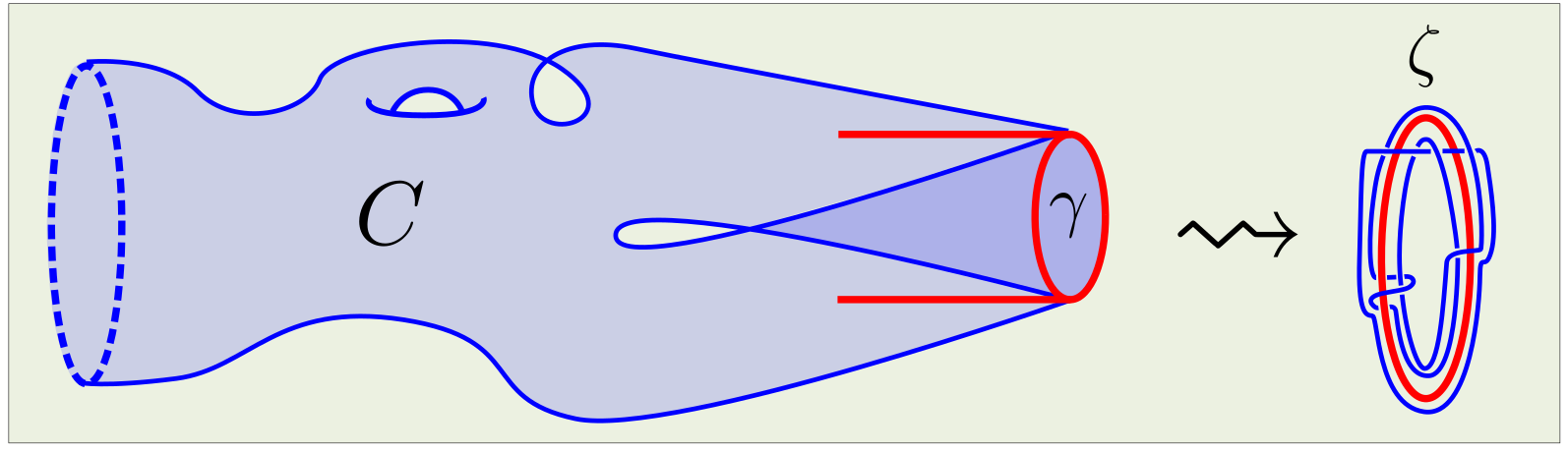}
    \caption{The braid near a simple orbit $\gamma$ determined by a curve (or current) $C$.}
    \label{fig:asymptotic_braid}
\end{figure}
\noindent Given a choice of trivialization $\tau$ of $\xi$, any braid near $\gamma$ is identified with a braid in $S^1 \times D^2$. The writhe of this braid is denoted by $w_\tau(\mathcal{C};\gamma)$, and we define the writhe of $\mathcal{C}$ as
\[w_\tau(\mathcal{C}) = \sum_{(\gamma,m) \in \alpha} w_\tau(\mathcal{C};\gamma) - \sum_{(\gamma,m) \in \beta} w_\tau(\mathcal{C};\gamma)\]
The writhe bound estimates the writhe of $\mathcal{C}$ in terms of the difference between the Conley-Zehnder terms in the ECH index and Fredhom index. 

\begin{theorem*}[Writhe Bound] \label{thm:intro_standard_writhe_bound} \cite{Hutchings_index_revisited} Let $C$ be a somewhere injective $J$-holomorphic curve in $\R \times Y$. Then
\[
w_\tau(C) \leq CZ_{\tau}^{ECH}(C)-CZ_{\tau}^{ind}(C).
\]
\end{theorem*}

\noindent The second ingredient to Theorem \ref{thm:intro_standard_index_inequality} is the adjunction inequality. This is a version of the classical adjunction inequality in complex geometry, tailored to the setting of ECH.

\begin{theorem*}[Adjunction] Let $C$ be a somewhere injective $J$-holomorphic curve in $\R \times Y$. Then
\[c_\tau(C) = \chi(C) + Q_\tau(C) + w_\tau(C) -  2 \delta(C)\]
\end{theorem*}

\noindent The index inequality is a short calculation using these two bounds. 

\subsubsection{ECH Complex} \label{subsubsec:ECH_complex} We are now ready to explain the construction of the ECH chain complex, using the properties of the ECH index discussed above.

\begin{remark}[Orientations] In this paper, we work without orientation for simplicity. In particular, we use $\Z/2$-coefficient to define all ECH groups.\end{remark}

An \emph{ECH generator} is an orbit set $\Theta$ such that any hyperbolic orbit $\gamma_i$ in $\Theta$ has multiplicity $1$. The ECH chain complex is simply the free vector-space over $\Z/2$ generated by ECH generators.
\[ECC(Y,\lambda) := \Z/2\big\langle \text{ECH generators $\Theta$ of $(Y,\lambda)$}\big\rangle\]
The differential on the ECH complex is defined by counting currents. To formalize this, we require the following classification of low ECH index currents deducible from the index bound.

\begin{theorem*}[Low-Index Currents] \label{theorem_regular_ECH_index} Let $J$ be a regular, compatible almost complex structure on $\R \times Y$ and let $\mathcal{C}$ be a $J$-holomorphic current of ECH index $I(\mathcal{C}) \le 2$. Then
\begin{itemize}
    \item $I(\mathcal{C}) \ge 0$ with equality only if $\mathcal{C}$ is a union of trivial cylinders.
    \item If $I(\mathcal{C}) = 1$ then $\mathcal{C} = C \sqcup \mathcal{T}$ where $C$ is Fredholm index $1$ and embedded, and $\mathcal{T}$ is a union of trivial cylinders with multiplicity.
    \item If $I(\mathcal{C}) = 2$ and is asymptotic to ECH generators at $\pm \infty$, then $\mathcal{C} = C \sqcup \mathcal{T}$ where $C$ is Fredholm index $2$ and embedded, and $\mathcal{T}$ is a union of trivial cylinders with multiplicity.
\end{itemize}
\end{theorem*}

\noindent This classification can, in turn, be used to deduce compactness properties of low ECH index moduli spaces.

\begin{theorem*} \label{thm:intro_standard_ECH_M1} \cite[ \S 5.3-5.4]{Hutchings_lectures_on_ECH} Let $J$ be a regular, compatible almost complex structure on $\R \times Y$ and let $\mathcal{M}_k(\Theta,\Xi)$ be the space of ECH index $k$ $J$-holomorphic currents from $\Theta$ to $\Xi$. Then
\begin{itemize}
    \item the space $\mathcal{M}_1(\Theta,\Xi)/\R$ is $0$-dimensional and compact
    \item the space $\mathcal{M}_2(\Theta,\Xi)/\R$ is a $1$-manifold with a compact truncation\footnote{It is a technical point that we cannot guarantee a priori that the moduli space of index $2$ currents is compact. The truncation may be viewed as a replacement for the compactification.} $\mathcal{M}_2'(\Theta,\Xi)/\R$ with a map
    \[\Pi:\partial \mathcal{M}_2'(\Theta,\Xi)/\R \to \bigsqcup_{\Theta'} \; \mathcal{M}_1(\Theta,\Theta')/\R \times \mathcal{M}_1(\Theta',\Xi)/\R\]
    \item The inverse image $\Pi^{-1}(\mathcal{C},\mathcal{C}')$ of a pair of currents in $\mathcal{M}_1(\Theta,\Theta')/\R \times \mathcal{M}_1(\Theta',\Xi)/\R$ has an odd number of points if and only if the orbit set $\Theta'$ is an ECH generator.
\end{itemize}
\end{theorem*}
Most of Theorem \ref{thm:intro_standard_ECH_M1} follows from Theorem \ref{theorem_regular_ECH_index}, a type of Gromov compactness due to Taubes and a bound on the topological complexity of low ECH index curves \cite{Hutchings_lectures_on_ECH}. However, the last point require a delicate obstruction bundle gluing argument that is well beyond the scope of this introduction. However, we will remark that this analysis requires certain \emph{partition conditions} obeyed by low ECH index currents, which restrict the braids that can appear at their ends. For a more detailed explanation of partition conditions, see \cite{Hutchings_lectures_on_ECH}.

\vspace{3pt}

By applying Theorem \ref{thm:intro_standard_ECH_M1}, one can simply define the ECH differential as the count of ECH index $1$ curves, modulo reparametrization, for a regular choice of $J$.
\[
\partial:ECC(Y,\lambda) \to ECC(Y,\lambda) \quad\text{be given by}\quad \partial\Theta = \sum_{\Xi} \#(\mathcal{M}_1(\Theta,\Xi)/\R) \cdot \Xi \mod 2
\]
It is a simple consequence of Theorem \ref{thm:intro_standard_ECH_M1} that $\partial$ is well-defined and that $\partial \circ \partial = 0$.

\subsection{Legendrian ECH And Main Results} We now move to the main topic of this paper, providing an overview of the construction of Legendrian ECH. We shall see that all of the constructions in ordinary ECH have generalizations to the Legendrian setting.

\subsubsection{Holomorphic Currents With Boundary} Let $Y$ be a contact 3-manifold with convex sutured boundary $\partial Y$ and a non-degenerate, adapted contact form $\lambda$. We refer the reader to \cite{CGHH} for a detailed treatment of sutured contact manifolds. The boundary of $Y$ divides as
\[\partial_-Y \qquad \partial_\sigma Y \quad\text{and}\quad \partial_+Y\]
where the Reeb vector-field is inward normal, tangent and outward normal respectively. Also fix a closed (possibly disconnected) Legendrian $\Lambda \subset \partial Y$ decomposing as
\[\Lambda_+ \subset \partial_+ Y \qquad\text{and}\qquad \Lambda_- \subset \partial_-Y \]
We assume that that these are exact Lagrangians in the Liouville domains $(\partial_\pm Y,\lambda|_{\partial_\pm Y})$, so that
\[\lambda|_{\partial_\pm Y} \text{ vanish along }\Lambda_\pm\]
We have included Figure \ref{fig:legendrian_ech_setup} below to better illustrate this setup for the reader.

\begin{figure}[h!]
    \centering
    \includegraphics[width=.8\textwidth]{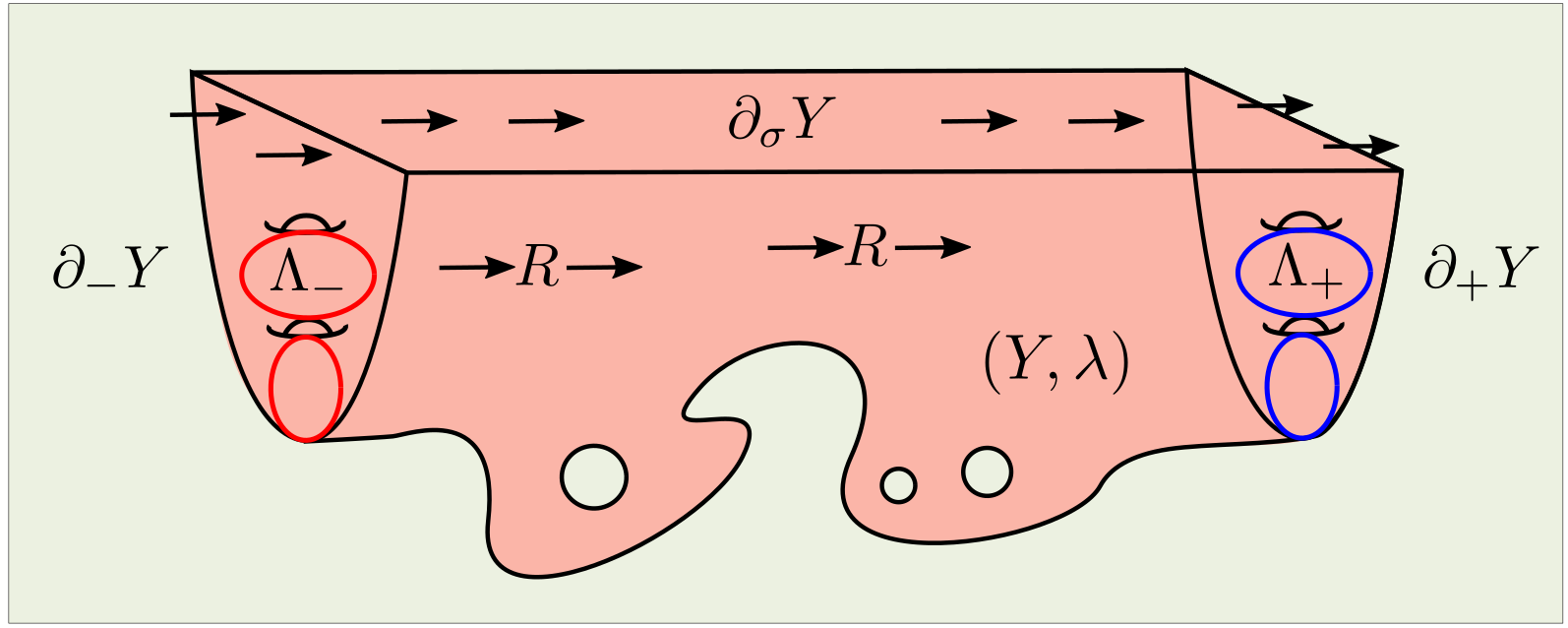}
    \caption{A contact manifold $Y$ with convex boundary and Legendrians $\Lambda_\pm \subset \partial_\pm Y$.}
    \label{fig:legendrian_ech_setup}
\end{figure}

\noindent There is a natural class of \emph{tailored} almost complex structures on $\R \times Y$, which we write as $\mathcal{J}_T(\mathbb{R}\times Y)$, that are compatible and also satisfy a set of assumptions near $\partial Y$ that gurantee that $J$-holomorphic curves do not cross the boundary. As defined in \cite{CGHH} tailoered almost complex structures exist in great abundance (see Section \ref{sec:J_holomorphic_currents} for a description).

\vspace{3pt}

We consider finite energy, proper $J$-holomorphic maps from a punctured Riemann surface $(\Sigma,j)$ with boundary to the symplectization of $Y$, with boundary on the symplectization of $\Lambda$.
\begin{equation} \label{eqn:J_curve_w_boundary}
u:(\Sigma,\partial\Sigma) \to (\R \times Y,\R \times \Lambda) \qquad\text{with}\qquad du \circ j = J \circ du
\end{equation}
A $J$-holomorphic current $\mathcal{C}$ in $(Y,\Lambda)$ is defined exactly as in the closed case, using $J$-holomorphic maps with boundary (\ref{eqn:J_curve_w_boundary}). These curves are now asymptotic at $\pm \infty$ to \emph{orbit-chord sets}
\[
\Xi = \{(\gamma_i,m_i)\}
\]
Here $\gamma_i$ is either a simple closed Reeb orbit or a Reeb chord of $\Lambda$, and $m_i$ is a multiplicity. We continue to denote set of $J$-holomorphic currents asymptotic to orbit-chord sets $\Theta$ at $+\infty$ and $\Xi$ at $-\infty$ by
\[\mathcal{M}(\Theta,\Xi)\]
We will discuss the foundations of $J$-holomorphic currents with boundary in our setting more fully in Section \ref{sec:J_holomorphic_currents}.

\subsubsection{Legendrian ECH Index} \label{subsubsec:LECH_index} We are now ready to introduce the Legendrian ECH index, generalizing the ECH index from the closed case. 

\begin{definition*}[Definition \ref{def:ech_index}] \label{def:intro_LECH_index} The \emph{(Legendrian) ECH index} of a $J$-holomorphic current $\mathcal{C}$ in the pair $(Y,\Lambda)$ from the orbit-chord set $\Theta$ to the orbit-chord set $\Xi$ is 
\[
I(\mathcal{C}) =\frac{1}{2} \cdot \mu_\tau (\mathcal{C}) + Q_\tau(\mathcal{C})  + CZ^{ECH}_{\tau}(\Theta) - CZ^{ECH}_\tau(\Xi)
\]
where the terms in the index $I$ are as follows.
\begin{itemize}
    \item $\tau$ is a trivialization of the bundle pair $(\xi,T\Lambda)$ over $\Theta$ and $\Xi$ (see Definition \ref{def:trivializations_over_orbit_chord_set})
    \item $\mu_\tau$ is the (relative) Maslov number (see Definition \ref{def:Maslov_number_of_surface_class})
    \item $Q_\tau$ is the \emph{relative self intersection number} with respect to $\tau$ (see Definition \ref{def:relative_intersection})
    \item $\CZ_\tau^{ECH}$ is a Conley-Zehnder term associated to the orbit-chord set (see Definition \ref{def:LECH_CZ_term}).
\end{itemize}
\end{definition*}

\noindent All of the terms in Definition \ref{def:intro_LECH_index} directly correspond to (and generalize) the terms in (\ref{eqn:ECH_index}). The Legendrian ECH index satisfies a direct generalization of the ECH index inequality.

\begin{theorem*}[Legendrian Index Inequality] \label{thm:legendrian_index_inequality} Let $C$ be a somewhere injective, $J$-holomorphic curve with boundary in $(\R \times Y,\R \times \Lambda)$ for tailored $J$. Then
\[\ind(C) \le I(C) - 2\delta(C) - \epsilon(C)\]
Here $\delta(C)$ and $\epsilon(C)$ are the counts of the singularities of $C$ in its interior and its boundary, respectively. \end{theorem*}

This result is a consequence of Legendrian generalizations of adjunction and the writhe bound.

\begin{theorem*}[Legendrian Adjunction, \S \ref{subsection:adjunction}] \label{intro:legendrian_adjunction} Let $C$ be a proper, finite energy, somewhere injective curve in $(\R \times Y,\R \times \Lambda)$. Then
\[\frac{1}{2} \cdot \mu_\tau(C) = \bar{\chi}(C) + Q_\tau(C) + w_\tau(C) - 2\delta(C) - \epsilon(C) \]
\end{theorem*}

\noindent Note that $\bar{\chi}$ denotes a corrected Euler characteristic that only counts boundary punctures with a factor of $\frac{1}{2}$. 

To study the writhe, we further adapt Siefring's \cite{siefring2008relative} asymptotic formula for $J$-holomorphic curves asymptotic to chords.

\begin{theorem*}
Let $u_i:[0,\infty) \times [0,1]\rightarrow \mathbb{R} \times Y$ be a collection of $J$-holomorphic strips asymptotic to a Reeb chord $\gamma(t)$. Here $\gamma$ connects between Legendrians $L_1$ and $L_2$, and $u_i$ maps the boundary of the strip to $\mathbb{R} \times L_i$. Then there exists a neighborhood $U$ of $\gamma$, a smooth embedding $\Phi:\mathbb{R} \times U \rightarrow \mathbb{R} \times Y$, proper reparametrizations $\phi_i: \mathbb{R} \times [0,1] \rightarrow \mathbb{R} \times [0,1]$ exponentially asymptotic to the identity, and positive integers $N_i$ so that for large values of $s$,
\[
\Phi \circ u_i \circ \psi_i(s,t) = (s,t, \sum_{k=1}^{N_i} e^{\lambda_{i,j}s}e_{i,j}(t))
\]
where $\lambda_{i,j}$ are negative eigenvalues of the asymptotic operator associated to $\gamma$, and the $e_{i,f}(t)$ are the corresponding eigenfunctions.
\end{theorem*}
This theorem is discussed in detail in the Appendix.
We will delay the precise statement of the writhe bound to Section \ref{sec:writhe and linking bounds}. 

\subsubsection{Legendrian ECH Complex} \label{subsubsec:legendrian_ECH_complex} Finally, we present the construction of the Legendrian ECH chain complex of $(Y,\Lambda)$, mirroring the construction in the closed setting. These claims will be revisited and proven in Section \ref{sec:lech}.

\begin{definition} \label{def:ECH_generator_intro} An \emph{ECH generator} of $(Y,\Lambda)$ is an orbit-chord set $\Theta = \{(\gamma_i,m_i)\} \cup \{(c_i,n_i)\}$ where
\begin{itemize}
    \item Every hyperbolic orbit $\gamma_i$ has multiplicity $1$.
    \item Every chord $c_i$ is multiplicity $1$.
    \item There is at most one Reeb chord incident to $L$ in $\Theta$ for each connected component $L$ of $\Lambda$
\end{itemize}
\end{definition}

As with the closed case, we define the differential by counting ECH index $1$ curves. We again need a classification of low ECH index curves and an accompanying compactness statement. 

\begin{theorem*}[Low-Index Currents With Boundary] \label{thm:regular_LECH_index} Let $J$ be a regular, tailored almost complex structure on $(\R \times Y,\R \times \Lambda)$ and let $\mathcal{C}$ be a $J$-holomorphic current of ECH index $I(\mathcal{C}) \le 2$. Then
\begin{itemize}
    \item $I(\mathcal{C}) \ge 0$ with equality only if $\mathcal{C}$ is a union of trivial cylinders and strips with multiplicity.
    \item If $I(\mathcal{C}) = 1$ then $\mathcal{C} = C \sqcup \mathcal{T}$ where $C$ is Fredholm index $1$ and embedded, and $\mathcal{T}$ is a union of trivial cylinders and strips with multiplicity.
    \item If $I(\mathcal{C}) = 2$ and is asymptotic to ECH generators at $\pm \infty$, then $\mathcal{C} = C \sqcup \mathcal{T}$ where $C$ is Fredholm index $2$ and embedded, and $\mathcal{T}$ is a union of trivial cylinders and strips with multiplicity.
\end{itemize}
\end{theorem*}

\begin{theorem*} \label{thm:intro_Legendrian_ECH_compactness} Let $J$ be a regular, tailored almost complex structure on $(\R \times Y,\R \times \Lambda)$ and let $\mathcal{M}_k(\Theta,\Xi)$ be the space of ECH index $k$ $J$-holomorphic currents in $(Y,\Lambda)$ from $\Theta$ to $\Xi$. Then
\begin{itemize}
    \item the space $\mathcal{M}_1(\Theta,\Xi)/\R$ is $0$-dimensional and compact
    \item the space $\mathcal{M}_2(\Theta,\Xi)/\R$ is a $1$-manifold with a compact truncation\footnote{It is a technical point that we cannot guarantee a priori that the moduli space of index $2$ currents is compact. The truncation may be viewed as a replacement for the compactification.} $\mathcal{M}_2'(\Theta,\Xi)/\R$ with a map
    \[\Pi:\partial \mathcal{M}_2'(\Theta,\Xi)/\R \to \bigsqcup_{\Theta'} \; \mathcal{M}_1(\Theta,\Theta')/\R \times \mathcal{M}_1(\Theta',\Xi)/\R\]
    \item The inverse image $\Pi^{-1}(\mathcal{C},\mathcal{C}')$ of a pair of currents in $\mathcal{M}_1(\Theta,\Theta')/\R \times \mathcal{M}_1(\Theta',\Xi)/\R$ has an odd number of points if and only if the orbit-chord set $\Theta'$ is an ECH generator.
\end{itemize}
\end{theorem*}

\noindent We will give detailed proofs of these claims in Section \ref{sec:lech}. They are analogous to the closed case. 

\begin{remark} Note that, due to the lack of multiply covered Reeb chords (in contrast to orbits), there is no need to carry out a new obstruction bundle gluing strategy. Instead, the work in
\cite{obs1} and \cite{obs2} will yield the last claim in Theorem \ref{thm:intro_Legendrian_ECH_compactness} after some minor modifications. We will discuss this in \ref{subsec:obg}.
\end{remark}

Finally, we can state the definition of Legendrian embedded contact homology.

\begin{definition}[Legendrian ECH] The \emph{ECH chain complex} $ECC(Y,\Lambda,\lambda)$ of $(Y,\Lambda,\lambda)$ is the free $\Z/2$-module generated by ECH generators.
\[ECC(Y,\Lambda,\lambda) = \Z/2\big\langle \text{ECH generators }\Theta\big\rangle\]
The differential $\partial_J$ with respect to a regular, tailored almost complex structure $J$ on $(\R \times Y,\R \times \Lambda)$ is given by the count of ECH index $1$ $J$-holomorphic currents. 
\[
\partial_J\Theta = \sum_{\Xi} \#_2\mathcal{M}_1(\Theta,\Xi)/\R \cdot \Xi
\]
The \emph{Legendrian embedded contact homology} $ECH(Y,\Lambda;\lambda,J)$ is the homology $ECC(Y,\Lambda,\lambda)$. \end{definition}

We prove that the differential defined above gives rise to a chain complex (i.e. $\partial ^2 =0$) in Section \ref{sec:lech}.

\begin{remark}[Invariance] \label{rmk:invariance} If $Y$ is closed, then $ECH(Y,\lambda,J)$ is independent of $\lambda$ and $J$ up to canonical isomorphism. Thus we may write
\[ECH(Y) := ECH(Y,\lambda,J) \qquad\text{for any choice of } \lambda,J\]
This is due to a chain level correspondence with the hat flavor $\widehat{HM}(Y)$ of Mrowka-Kronheimer's monopole Floer homology. An analogous proof of invariance in the Legendrian setting is beyond the scope of this paper and is an interesting topic for future work.
\end{remark}

\begin{remark}[Component Grading] \label{rmk:component_grading} Let $\Lambda$ and $K$ be Legendrians in $\partial Y$ satisfying the hypotheses of our construction, and suppose that $\Lambda \subset K$. Then this induces an injective map
\[\iota^K_\Lambda:ECH(Y,\Lambda,\lambda,J) \to ECH(Y,K,\lambda,J)\]
If $\Lambda$ is the empty set, then $ECH(Y,\emptyset,\lambda,J)$ is simply the sutured ECH $ECH(Y)$ of Colin-Ghiggini-Honda-Hutchings \cite{CGHH}. In particular, there is always an inclusion
\[ECH(Y) \to ECH(Y,\Lambda,\lambda,J)\]\end{remark}

\begin{remark} \label{rmk:homology_grading} Given a homology class $A \in H_1(Y,\Lambda;\Z)$, there is a sub-group
\[
ECH_A(Y,\Lambda,\lambda,J) \subset ECH(Y,\Lambda,\lambda,J)
\]
generated by orbit sets $\Theta$ in the homology class $A$. This induces a direct sum decomposition
\[ECH(Y,\Lambda,\lambda,J) = \bigoplus_{A \in H_1(Y,\Lambda)} ECH_A(Y,\Lambda,\lambda,J)\]
\end{remark}

\begin{remark}[Reeb Chord Filtration] \label{rmk:chord_filtration} Let $\mathfrak{z}$ be a Reeb chord connecting $\partial_-Y$ to $\partial_+Y$ that is disjoint from $\Lambda$. Then there is an associated holomorphic sub-manifold $Z \subset \R \times Y$ given by
\[Z = \R \times \mathfrak{z}\]
Given this choice, we can define an extended ECH complex
\[
\overline{ECC}(Y,\Lambda,\lambda,J) = \mathbb{F}_2[t] \otimes ECC(Y,\Lambda,\lambda,J)
\]
and a differential on $\overline{ECH}(Y,\Lambda,\lambda,J)$ (as a module over $\mathbb{F}_2[t]$) determined by $J$ and $\mathfrak{z}$.
\[
\partial_{J,\mathfrak{z}}(\Theta) = \sum_{\Xi} \Big( \sum_{\mathcal{C} \in \mathcal{M}_1(\Theta,\Xi)} t^{\mathcal{C} \cdot Z} \cdot \Xi \Big)
\]
Here $\mathcal{C} \cdot Z$ denotes the count (with multiplicity) of interior intersections between $\mathcal{C}$ and $Z$. By intersection positivity, this must always be positive. This defines an extended homology
\[\overline{ECH}(Y,\Lambda,\lambda,J) := H(\overline{ECC}(Y,\Lambda,\lambda,J))\]
We can extract further homology groups by studying the associated graded to the $t$-filtrartion. As we will discuss in \S \ref{subsec;HF_and_ECH}, our construction is related to a construction of Heegaard-Floer homology. \end{remark}

\begin{remark}[Previous Work] \label{rmk:previous_work} The Legendrian ECH index has appeared in the works of Colin-Ghiggini-Honda \cite{echhf1,echhf2} in a more limited context, in the process of establishing an isomorphism between ECH and Heegaard Floer homology. 

\vspace{3pt}

This work is a natural elaboration on \cite{echhf1}. In particular, we fully develop a general theory of holomorphic currents with boundary, adjunction, writhe bound and the Legendrian ECH index that goes beyond the specialized context of \cite{echhf1,echhf2}. In our version of the index inequality, our holomorphic currents, for instance, have no restrictions on their asymptotic braids and are permitted to have boundary singularities.\footnote{However we must impose more restrictions if we want the differential to square to zero, see Section \ref{sec:lech} for details.}
\end{remark}

\subsubsection{Legendrian ECH In The Closed Case} The Legendrian ECH setup is, at first glance, highly constrained. We now explain how to use our framework to associate ECH groups to \emph{any} pair
\[
(Y,\Lambda)
\]
of a closed contact $3$-manifold $Y$ with contact form $\alpha$ and a closed Legendrian $\Lambda \subset Y$. Crucially, these ECH groups essentially count Reeb chords of $\Lambda$ with respect to the initial contact form. 

\vspace{3pt}

To start, choose an arbitrary metric $g$ on $\Lambda$ and choose a small $\epsilon > 0$. Let $D^* \Lambda$ denote the unit codisk bundle. The Weinstein neighborhood theorem for Legendrians states that, for small $\epsilon$ and after scaling the metric to shrink the codisk bundle, there is an embedding
\[
\iota:B := [-\epsilon,\epsilon]_t \times D^* \Lambda \to Y \qquad\text{such that}\qquad \iota(0,\Lambda) = \Lambda \text{ and } \iota^*\alpha = dt + \lambda_{\op{std}}
\]
The complement $M = Y \setminus \iota(\op{int}(B))$ of the interior of $B$ in $Y$ is a \emph{concave} sutured contact manifold (see \cite[Def. 4.2]{CGHH}) with a decomposition of the boundary into three pieces
\[
\partial_- M = \epsilon \times D^* \Lambda \qquad\partial_+ M = -\epsilon \times D^* \Lambda \quad\text{and} \quad \partial_\circ M = [-\epsilon,\epsilon] \times S^* \Lambda
\]
Note that the Reeb vector-field of $\alpha$ points out of $\partial_+M$ and into $\partial_-M$. This is the reason for the sign reversal in the notation. There are two copies of $\Lambda$ on $\partial M$ given by
\[\Lambda_- := \epsilon \times \Lambda 
\subset \partial_-M \qquad\text{and}\qquad \Lambda_+ := \epsilon \times \Lambda \subset \partial_+M \]

The next step is to apply the \emph{concave-to-convex} operation described in \cite[\S 4.2]{CGHH}. In particular, there is a plug $U$ that one can attach to a neighborhood of $\partial_\circ M$ to acquire a convex sutured contact manifold
\[
\check{Y} := M \cup U \qquad\text{with contact form }\check{\alpha}
\]
The plug only modifies $\partial M$ near $\partial_\circ M$, so that $\Lambda_+ \subset \partial_+\check{Y}$ and $\Lambda_- \subset \partial_-\check{Y}$. Moreover, every Reeb chord $c$ from $\Lambda_-$ to $\Lambda_+$ arises as a sub-chord of a self Reeb chords of $\Lambda$, and the lengths differ by an error of $2\epsilon$. We can now make the following definition.

\begin{definition} \label{def:LECH_of_closed_Legendrian} The \emph{Legendrian embedded contact homology} $ECH^L(Y,\Lambda)$ of $(Y,\alpha)$ and $\Lambda$ is the Legendrian ECH of $(\R \times Y,\R \times \Lambda)$ of action filtration below $L$.
\[ECH^L(Y,\Lambda;\delta) := ECH^L(\check{Y},\check{\Lambda};\check{\alpha},J)\]
Here $\delta$ denotes the set of all choices made during the construction: the metric $g$, the parameter $\epsilon$, the embedding $\iota$ and the tailored almost complex structure $J$. \end{definition}

The groups in Definition \ref{def:LECH_of_closed_Legendrian} certainly depend on the specific choices, since Reeb chords can appear and disappear depending on the size of $\iota(B)$ in $Y$. Addressing this issue is far beyond the scope of this paper. However, let us state an optimistic conjecture in this direction. 

\begin{definition} A choice of data $\delta = (g,\epsilon,\iota,J)$ is \emph{$L$-admissible} if
\begin{itemize}
\item Every Reeb orbit of length less than or equal to $L$ is in $\check{Y}$.
\item The Reeb chords $\Lambda_+ \to \Lambda_-$ in $\check{Y}$ of length less than or equal to $L$ are in bijection with the Reeb chords of $\Lambda \to \Lambda$ in $Y$ of length less than or equal to $L$
\end{itemize}
These criteria are always achievable by shrinking $\iota$ and scaling the metric to reduce the size of $B$. \end{definition}

\begin{conjecture} \label{conj:ECH_of_closed_pair} Let $(Y,\Lambda)$ be a pair of a closed contact $3$-manifold $Y$, a closed Legendrian link $\Lambda \subset Y$ and a contact form $\alpha$ that is non-degenerate for $(Y,\Lambda)$. Then
\begin{itemize}
\item (Well-Defined) If $\delta$ and $\delta'$ are two choices of $L$-admissible data, then there is a natural isomorphism
\[ECH^L(Y,\Lambda;\delta) \simeq ECH^L(Y,\Lambda;\delta')\]
The resulting group $ECH^L(Y,\Lambda)$ is the filtered ECH of $(Y,\Lambda)$.
\item (Filtration Map) For any $K > L$, there is a map
\[\iota^K_L:ECH(Y,\Lambda) \to ECH^K(Y,\Lambda) \qquad\text{such that}\qquad \iota^M_L = \iota^M_K \circ \iota^K_L\]
\item (Colimit) The colimit of the filtered ECH groups of $(Y,\Lambda)$ over $L$
\[ECH(Y,\Lambda) := \underset{L}{\op{colim}} ECH^L(Y,\Lambda)\]
depends only on $(Y,\Lambda)$ up to contactomorphism of the pair.
\end{itemize}
\end{conjecture}

\begin{remark} It is likely that the maps in Conjecture \ref{conj:ECH_of_closed_pair} would be built from cobordism maps arising in a Seiberg-Witten based model for Legendrian ECH. In particular, Conjecture \ref{conj:ECH_of_closed_pair} is of a similar difficulty to the invariance proof discussed in Remark \ref{rmk:invariance}. \end{remark}

\begin{remark}
In \cite{CGHH} they use sutured ECH to define an (conjectured) invariant of Legendrian in closed 3-manifolds by considering sutured ECH of the same convex sutured manifold. However, we expect our invariant to be different from theirs because we allow Reeb chords in our chain complex.
\end{remark}

\subsection{Motivation And Future Directions} This paper lays the groundwork for several future projects on the structure of ECH, each of which provides ample motivation for the development of our theory. We conclude this introduction by giving an overview of these motivating projects.

\subsubsection{Circle-Valued Gradient Flows} \label{subsubsec:circle_valued_flows} Let $M$ be a closed $3$-manifold equipped with a circle-valued Morse function. That is, a smooth function
\[f:M \to S^1\]
with isolated critical points $p$ that each have non-degenerate Hessian. Assume also that $f$ has no index $0$ or index $3$ critical points. In \cite{hl1999a,hl1999b}, Hutchings and Lee defined a $3$-manifold invariant 
\[I_3:\op{Spin}^c(M) \to \Z\]
from the set of spin-c structures on $M$ to the integers, via counts of configurations of closed orbits and flow lines of the gradient vector field of $f$. This invariant was motivated by and related to a number of other previously known invariants. 

\vspace{3pt}

First, in a series of papers \cite{t1990,t1997}, Turaev introduced a form of Reidemeister torsion, later dubbed \emph{Turaev torsion}, which is also a map
\[\tau_M:\op{Spin}^c(M) \to \Z.\]
Hutchings-Lee proved in \cite{hl1999b} that $\tau_M = I_3$. On the other hand, rapid developments in low-dimensional topology and gauge theory contemporaneous to \cite{t1990,t1997} lead to the introduction of the \emph{Seiberg-Witten invariant}
\[\op{SW}_M:\op{Spin}^c(M) \to \Z\]
This invariant is defined using a signed and weighted count of solutions to the $3$-dimensional Seiberg-Witten equation (cf. \cite{l2000}). Turaev established in \cite{t1998} that $\op{SW}_M = \tau_M$ (up to sign), proving through indirect means that
\begin{equation} \label{eqn:I3_is_tau} I_3 = \SW_M\end{equation}
Through the equality (\ref{eqn:I3_is_tau}), one is lead to the following question.

\begin{question} \label{qu:I3_vs_SW} Is there a direct proof of the equality $I_3 = \SW_M$ that does not use Turaev torsion?
\end{question}

Moreover, the Seiberg-Witten invariant was categorified by Kronheimer-Mrowka's monopole Floer homology $HM_\bullet$ \cite{km2007} and other variants of Seiberg-Witten-Floer theory. This suggests the following (roughly formulated) question, which is related to Question \ref{qu:I3_vs_SW}. 

\begin{question} \label{qu:categorify_I3} Is there a Floer homology theory $FH_\bullet(Y,f)$ of $3$-manifolds $Y$ with a circle valued Morse function $f:Y \to S^1$ as above such that
\begin{itemize}
    \item[(a)] $FH_\bullet$ categorifes $I_3$, i.e. it is computed as the homology of a complex generated by counts of configurations of gradient flow lines and closed orbits as in $I_3$.
    \item[(b)] $FH_\bullet$ is isomorphic to (the appropriate flavor of) monopole Floer homology.
\end{itemize}
\end{question}

\noindent Embedded contact homology, and its sister theory periodic Floer homology (PFH), answer both of these questions when the Morse function $f$ has no critical points. In this case, $M$ may be viewed as a mapping torus of a map
\[\phi:\Sigma \to \Sigma \qquad\text{where}\qquad \Sigma = f^{-1}(0)\]
and the generators of the hypothetical Floer homology groups $FH_\bullet$ must be configurations of periodic points of $\phi$. The PFH groups $PFH_\bullet$ provides just such a theory and a result of Lee-Taubes \cite{lt2012} states that $PFH_\bullet$ and $HM_\bullet$ are isomorphic\footnote{Note that in the general case, the isomorphism between $HM_\bullet$ and $PFH_\bullet$ uses variants of both Floer groups with appropriate twisted Novikov coefficients.}. 

\vspace{3pt}

In the general case where $f$ can have critical points, Questions \ref{qu:I3_vs_SW} and \ref{qu:categorify_I3} are still open. However, there is a potential approach based on a slight generalization of the constructions in this paper. Choose a metric $g$ on $M$ such that $f$ is harmonic (this is possible if $f$ has no index $0$ or $3$ critical points). Assume that each index $2$ critical point $p$ has a Morse chart where 
\[
f(x,y,z)= 2x^2-y^2-z^2.
\]
In this chart, we can remove a standard neighborhood $U$ and introduce a boundary component to $M$. The new boundary has corners, and on the smooth components the gradient vector-field is either orthogonal to or tangent to the boundary. This neighborhood is depicted in Figure \ref{fig:critical_point_of_harmonic_function}.
\begin{figure}[h!]
    \centering
    \includegraphics[width=.8\textwidth]{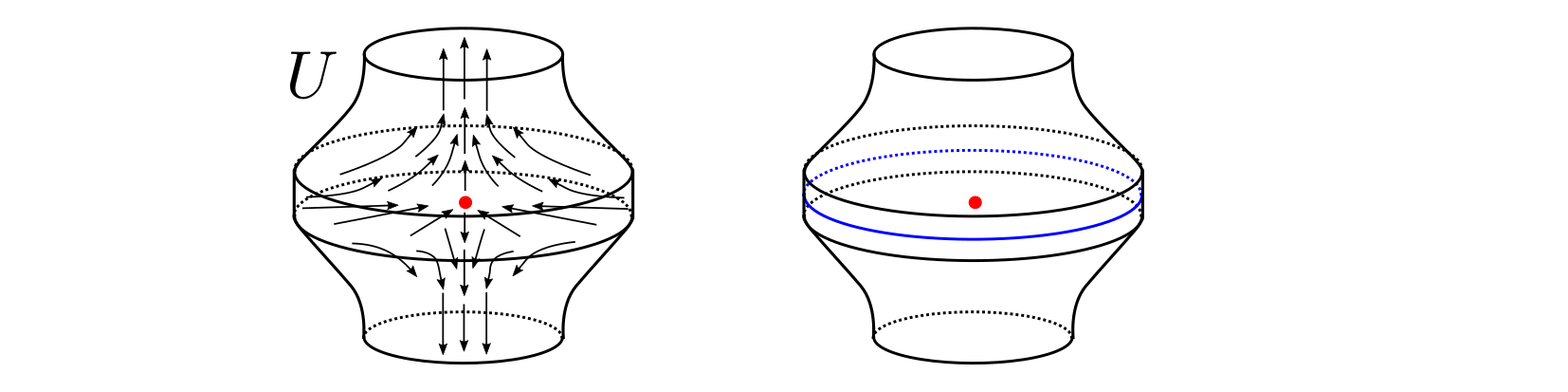}
    \caption{A standard neighborhood $U$ of an index $2$ critical point. The critical point is in red and the descending sphere is in blue. Note the regions where the gradient points in and out of the boundary are level surfaces of $f$.}
    \label{fig:critical_point_of_harmonic_function}
\end{figure}

We can perform this neighborhood removal around each critical point of $f$ (using an analogous local model near the index $1$ points) to acquire a new space $Y \subset M$. This space is equipped with a stable Hamiltonian structure with $2$-form $\omega_f$ given by the Hodge dual of $df$ and stabilizing $1$-form $\theta_f = df$. 
\begin{figure}[h!]
    \centering
    \includegraphics[width=.8\textwidth]{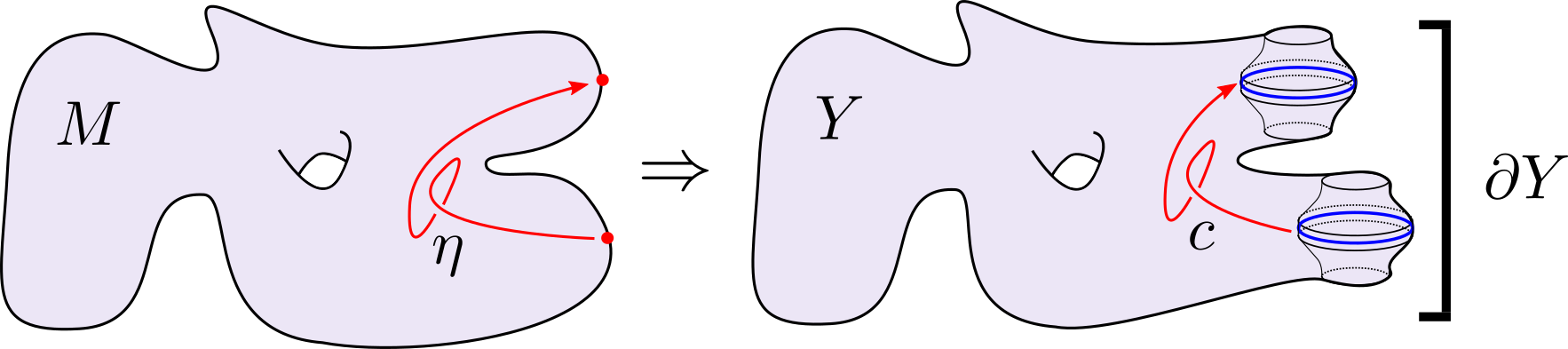}
    \caption{The space acquired by removing standard neighborhoods around each critical point from $Y$.}
    \label{fig:deleted_neighborhood_space}
\end{figure}
The boundary of $Y$ contains a natural $1$-dimensional sub-manifold $\Lambda$ given as the union of the ascending sphere of the index $1$ critical points and the descending spheres of the index $2$ critical points. This sub-manifold satisfies
\[
T\Lambda \subset \text{ker}(\theta_f) \qquad \text{and}\qquad \omega_f|_{T\Lambda} = 0
\]
In other words, $\Lambda$ is the analogue of a Legendrian in the stable Hamiltonian manifold $(Y,\omega_f,\theta_f)$. Moreover, any gradient flow line $\eta$ from an index $1$ critical point $p$ to an index $2$ critical point $q$ becomes a chord $c$ connecting the corresponding components of $\Lambda$.

\vspace{3pt}

The pair $(Y,\Lambda)$ very much resembles a stable Hamiltonian analogue of the setup used in this paper to define Legendrian ECH. Thus, one approach to addressing Question \ref{qu:categorify_I3} is to use the methods of this paper to develop a PFH version of our Legendrian ECH theory for this setting. There are some significant technical challenges to carrying out this program. For example, the boundary of $Y$ is naturally \emph{concave} sutured, rather than convex, and thus we do not have immediate access to an appropriate maximum principle. This could be addressed by adapting the \emph{concave-to-convex} operation in \cite{CGHH} to our stable Hamiltonian setting and establish a maximum principle for $J$-holomorphic curves near the boundary. We hope to pursue this in future work.

\subsubsection{Heegaard-Floer Homology And Embedded Contact Homology} \label{subsec;HF_and_ECH} It has been shown that embedded contact homology is isomorphic to (the appropriate flavors of) several other Floer homologies, notably Heegaard-Floer homology and monopole Floer homology. The isomorphisms are, however, highly non-trivial to construct. For instance, the construction of the isomorphism relating ECH to Heegaard-Floer theory (through open book decompositions) occupies four long papers due to Colin-Ghiggini-Honda \cite{echhf_survey,echhf1,echhf2,echhf3}.

\vspace{3pt}

The work of Colin-Ghiggini-Honda utilizes a cylindrical reformulation of Heegaard-Floer homology due to Lipshitz \cite{Lipshitz} that we now describe in broad terms. The construction begins with a \emph{pointed Heegaard diagram} that we write as follows.
\[(\Sigma,\alpha,\beta,\mathfrak{z})\]
This consists of a closed orientable surface $\Sigma$ of genus $g$, a distinguished point $\mathfrak{z} \in \Sigma$ and two collections of $g$ simple, non-separating, closed curves
\[\mathbf{\alpha} = \alpha_1\cup\dots \cup \alpha_g\qquad\text{and}\qquad\mathbf{\beta} = \beta_1\cup\dots \cup \beta_g\]
Recall that $(\Sigma,\alpha,\beta,\mathfrak{z})$ determines a $3$-manifold $M$ (uniquely, up to diffeomorphism). More precisely, we take the $3$-manifold
\[Y = [0,1] \times \Sigma \qquad \text{with curves}\qquad \Lambda_\alpha = (0 \times \alpha) \quad\text{and}\quad \Lambda_\beta = (1 \times \beta)\]
and attach $2$-handles to each curve in $\Lambda_\alpha$ and $\Lambda_\beta$. The resulting manifold has a boundary consisting of two $2$-spheres, and after attaching $3$-handles to these areas, we acquire $M$. 

\vspace{3pt}

The space $Y = [0,1]_t \times \Sigma$ may be viewed as a stable Hamiltonian manifold with the two-form $ \omega$ equal to an area form on $\Sigma$ and stabilizing $1$-form $dt$. The analogue of the Reeb vector-field is $R = \partial_t$. Moreover, $\Lambda_\alpha$ and $\Lambda_\beta$ are Legendrians in the sense that
\[\omega|_{\Lambda_\pm} = 0 \qquad\text{and}\qquad dt|_{\Lambda_\pm} = 0\]
In this picture, Reeb chords between $\Lambda_\alpha$ and $\Lambda_\beta$ are equivalent to intersection points in $\alpha \cap \beta$. The symplectization of $Y$ is given the symplectic manifold
\[W = \R_s \times Y \qquad\text{with symplectic form}\qquad \Omega = ds \wedge dt + \omega_\Sigma\]
The cylindrical sub-manifolds $L_\alpha = \R \times \Lambda_\alpha$ and $L_\beta = \R \times \Lambda_\beta$ are both Lagrangians. There is a natural class of compatible almost complex structures $J$ on $W$: those that are translation invariant and that satisfy $J(\partial_s) = \partial_t$ and $J(T\Sigma) = T\Sigma$. Finally, note that the marked point $\mathfrak{z}$ determines a $J$-holomorphic strip
\[Z = \R \times [0,1] \times \mathfrak{z}\]

\begin{figure}
    \centering
    \includegraphics[width=.8\textwidth]{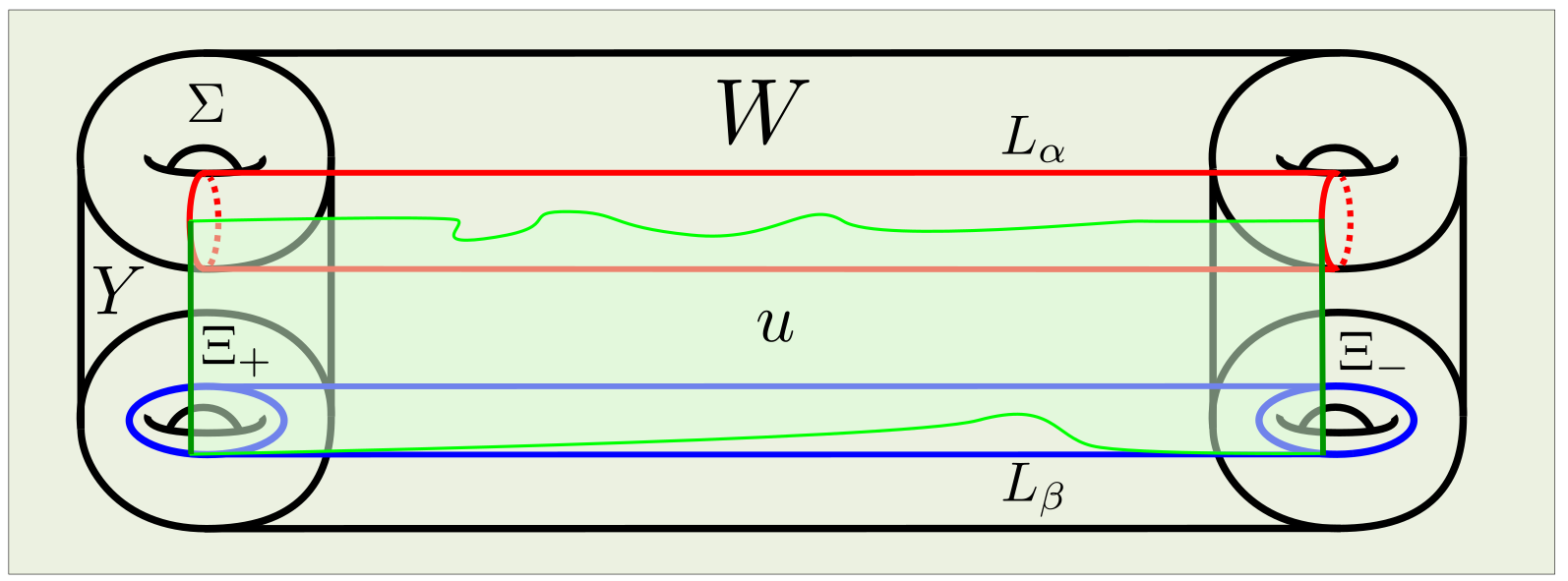}
    \caption{A picture of the setup of Lipshitz' cylindrical formulation of Heegaard-Floer homology.}
    \label{fig:lipshitz_formulation}
\end{figure}

The formulation of Heegaard-Floer homology of Lipshitz \cite{Lipshitz} can now be described as follows. The chain complex $CF_{\text{Lip}}(\Sigma,\alpha,\beta,\mathfrak{z})$ is generated by sets of the form
\[\Xi = \{c_1,\dots,c_g\}\]
where $\Xi$ consists of chords $c_i$ from $\alpha_i$ to $\beta_{\sigma(j)}$ for some permutation $\sigma$, or equivalently a set of intersection points $p_i \in \alpha_i \cap \beta_{\sigma(i)}$. The differential $\partial$ counts pseudo-holomorphic curves of the following form. Let $\Gamma$ and $\Xi$ be two generators. Let $S$ with a Riemann surface with $2g$ boundary punctures, $g$ of which we label positive and $g$ of which we label negative. We consider $J$-holomorphic maps
\[u:(S,\partial S) \to (W,L_\alpha \cup L_\beta)\]
of Fredholm index $1$ (modulo translation), where the positive punctures are asymptotic to the chords corresponding to $\Gamma$ at $+\infty$ of the $\R$ direction in $W$; the negative punctures are asymptotic to the chords corresponding to $\Xi$ at $-\infty$ of the $\R$ direction in $W$; the image of $S$ under $u$ is embedded; and $u(S)$ is disjoint from $Z$. For more technical formulations of these conditions see section 1 in \cite{Lipshitz}. For a depiction of this setup, see Figure \ref{fig:lipshitz_formulation}.

\begin{theorem} \cite{Lipshitz} The homology groups of $CF_{\text{Lip}}(\Sigma,\alpha,\beta,\mathfrak{z})$ are isomorphic to the hat version of the Heegaard-Floer homology groups.
\[H(CF_{\text{Lip}}(\Sigma,\alpha,\beta,\mathfrak{z})) \simeq \widehat{HF}(Y)\]
\end{theorem}

The setup for Lipshitz' construction is very similar to the one for Legendrian ECH. In fact, it may be viewed as a special case of our construction adapted to the stable Hamiltonian manifold
\[(Y,\omega,\theta) = ([0,1] \times \Sigma,\omega_\Sigma,dt) \qquad\text{with ``Legendrians''}\qquad \Lambda_\alpha \cup \Lambda_\beta \subset \partial Y\] 
This is the perspective adopted by Colin-Ghiggini-Honda, and in \cite[\S 4]{echhf1} they prove the following claim:

\begin{claim} \cite[\S 4]{echhf1} The differential $CF_{\text{Lip}}(\Sigma,\alpha,\beta,\mathfrak{z})$ precisely counts ECH index $1$ holomorphic curves $C$ in $\R \times Y$ with boundary $\R \times (\Lambda_\alpha \cup \Lambda_\beta)$. 
\end{claim}

\vspace{3pt}

\noindent The theory articulated in this paper can therefore be viewed as providing a common framework for describing both Heegaard-Floer theory and embedded contact homology as special cases of a single construction, when one includes only the Reeb chords or Reeb orbits as generators. 

\subsubsection{Bordered Embedded Contact Homology And Gluing Formulas} \label{subsubsec:ECH_and_gluing} In the paper \cite{lot2018}, Lipshitz-Oszvath-Thurston formulate a theory of Heegaard-Floer homology for 3-manifolds with boundary, called \emph{bordered Heegaard-Floer homology}. 

\vspace{3pt}

Roughly speaking, the bordered Heegaard-Floer homology groups depend on a $3$-manifold $Y$ and on a diffeomorphism $\phi:F \simeq \partial Y$ of $\partial Y$ with a surface $F$ determined by a datum called a \emph{matched circle} $\mathcal{Z}$. The hat version of this theory comes in two flavors.
\[\widehat{CFA}(Y,\phi) \qquad \text{and}\qquad \widehat{CFD}(Y,\phi)\]
which are, respectively, an $A_\infty$-module and a dg-module over a certain free dg-algebra $\mathcal{A}$ associated to $\mathcal{Z}$. A fundamental application of this theory is the following gluing result.

\begin{theorem} \cite{lot2018} \label{thm:gluing_formula_HF} If $Y = M \cup_F N$ is a closed union of two $3$-manifolds with boundary via the identifications of $\phi:F \simeq \partial M$ and $\psi:F \simeq \partial N$ with $F$, then
\[
\widehat{HF}(Y) \simeq H(\widehat{CFA}(M,\phi) \otimes_{\mathcal{A}} \widehat{CFD}(N,\psi))
\]
Here $\otimes_{\mathcal{A}}$ denotes the tensor product as $A_\infty$-modules over $\mathcal{A}$. \end{theorem}
\noindent This gluing formula has since become a fundamental tool in computing Heegaard-Floer invariants and their knot counterparts. In particular, they are useful in studying submanifolds e.g. knots and surfaces \cite{Hom_bordered, Guth_Gary}.

\vspace{3pt}

Although many variants of symplectic field theory can be formulated for contact manifolds with boundary (e.g. \cite{CGHH}), a general gluing formula in the spirit of Theorem \ref{thm:gluing_formula_HF} has yet to appear. As a first step, one can try to prove Theorem \ref{thm:gluing_formula_HF} where the sutured Heegaard-Floer groups are replaced with variants of ECH for manifolds with boundary.

\vspace{3pt}

To further understand what bordered ECH might look like, we observe that the two types of hat bordered Heegaard-Floer complexes are constructed by extending Lipshitz' cylindrical formulation \cite{Lipshitz} in two ways. 

\vspace{3pt}

First, the complexes are computed using a \emph{bordered Heegaard diagram} $(\Sigma,\alpha,\beta)$ for $Y$, where $\Sigma$ has boundary and some $\alpha$ curves are permitted to have boundary in $\partial\Sigma$. Second, the bordered complexes incorporate Reeb chords between the Legendrians (i.e. points) $\alpha \cap \partial\Sigma$ in the matched circle $\mathcal{Z}$, which is identified with $\partial\Sigma$. The dg-algebra $\mathcal{A}$ associated to $\mathcal{Z}$ is, in fact, a variant of the Chekanov-Eliashberg dg-algebra of the Legendrians in $\mathcal{Z}$. The holomorphic curves counted in bordered Heegaard-Floer theory are allowed to limit to chords of $\alpha \cap \partial\Sigma$ at boundary punctures, and the chain complexes themselves are (roughly speaking) augmented by changing the ordinary Heegaard-Floer complexes to have coefficients in $\mathcal{A}$.

\vspace{3pt}

In light of these observations, we expect that a bordered theory of ECH will require an analogous extension of Legendrian ECH that incorporates Legendrians $\Lambda \subset \partial_\pm Y$ with boundary on $\partial\Lambda$ contained in $\partial(\partial_\pm Y) \simeq \sigma$, the suture, and Reeb chords between the points $\partial\Lambda \subset \sigma$. The precise formulation of this theory will be the subject of future work based on this paper.

\begin{acknowledgements*} We would like to thank our advisor Michael Hutchings for suggesting this project and helpful discussions. We would like to thank Ko Honda and Robert Lipshitz for comments. The third author acknowledges support by the NSF GRFP under Grant DGE 2146752. \end{acknowledgements*}

\section{Intersection Theory} \label{sec:intersection_theory} In this section, we describe an intersection theory for surfaces in symplectic cobordisms with boundary on a Lagrangian cobordism in dimension $4$. This mirrors the theory for surfaces without Lagrangian boundary described by Hutchings in \cite{Hutchings_index_revisited}.

\subsection{Bundle Pairs} We begin by introducing the notion of a bundle pair with punctures over a punctured Riemann surface with boundary.

\begin{definition} A \emph{symplectic bundle pair with punctures} $(E,F) \to (\Sigma,\partial\Sigma)$ consists of
\begin{itemize}
    \item[(a)] A compact, oriented surface with boundary and corners
    \[\Sigma \qquad\text{with}\qquad \partial\Sigma = \partial_\star \Sigma \cup \partial_\circ\Sigma\]
    where $\partial_\star \Sigma$ and $\partial_\circ \Sigma$ are each unions of smooth strata of the boundary.
    \item[(b)] A bundle pair of a symplectic vector-bundle and a Lagrangian sub-bundle
    \[(E,\omega) \to \Sigma \qquad\text{and}\qquad F \to \partial_\circ \Sigma \quad\text{with}\quad F \subset E|_{\partial_\circ \Sigma}
    \]
\end{itemize}
We refer to $\partial_\circ \Sigma$ as the \emph{Lagrangian boundary} and to $\partial_\star \Sigma$ as the \emph{puncture boundary}. We note by our previous notation $\partial_\star \Sigma$ is the same as $\partial_\pm \Sigma$. We equip both $\partial_\circ \Sigma$ and $\partial_\star\Sigma$ with the induced boundary orientation.
\end{definition}

\begin{definition} A \emph{puncture trivialization} $\tau$ of $(E,F) \to (\Sigma,\partial\Sigma)$ consists of the following. First a symplectic trivialization of $E$ over $\partial_\star\Sigma$, which we write as
\[
E|_{\partial_\star\Sigma} \simeq \C^n
\]
so that 
\[
F|_{\partial_\star \Sigma \cap \partial_\circ \Sigma} \simeq \R^n \subset \C^n.
\]
We fix the convention of thinking of $\tau$ as a map from $E \rightarrow \partial_\star\Sigma$ to $\C^n \rightarrow \partial_\star \Sigma$.

\end{definition}

Bundle pairs with punctures admit an integer invariant analogous to the Chern class, generalizing the Maslov number of an ordinary bundle pair.

\begin{proposition}[Maslov Number] \label{prop:maslov_number_of_punctured_pair} For any symplectic bundle pair with punctures $(E,F) \to (\Sigma,\partial \Sigma)$ and puncture trivialization $\tau$, there is a well-defined Maslov number
\[\mu(E,F;\tau) \in \Z\]
Furthermore, the Maslov number is uniquely determined by the following axioms.
\begin{itemize}
    \item[(a)] (Isomorphism) The Maslov number is invariant under isomorphism of the pair and trivialization.
    \[\mu(E,F;\tau) = \mu(E',F';\tau') \qquad \text{if}\qquad (E,F;\tau) \simeq (E',F';\tau')\]
    \item[(b)] (Direct Sum) The Maslov number is additive with respect to direct sum.
    \[\mu(E \oplus E', F \oplus F';\tau \oplus \tau') = \mu(E,F;\tau) + \mu(E',F';\tau')\]
    \item[(c)] (Disjoint Union) The Maslov number is additive under disjoint union.
    \[\mu(E \cup E', F \cup F';\tau \cup \tau') = \mu(E,F;\tau) + \mu(E',F';\tau')\]
    \item[(d)] (Puncture Gluing) Let $(E,F) \to (\Sigma,\partial\Sigma)$ and $(E',F') \to (\Sigma',\partial\Sigma')$ be two bundle pairs with punctures. Let $R \subset \partial_\star\Sigma$ and $R' \subset \partial_\star\Sigma'$ be components of the puncture boundary. Let $-R'$ denote $R'$ with reverse orientation. Assume we have a bundle pair isomorphism
    \[
    \psi:(E|_R) \simeq (E'|_{-R'}) \qquad\text{covering a diffeomorphism} \qquad R \simeq -R'
    \]
    Further, composed with trivializations $\tau$ and $\tau'$, the map $\psi$ takes the form
    \[
     \tau'^{-1} \circ \psi \circ \tau : (\C^n,\R^n) \rightarrow (\C^n,\R^n)
    \]
    over $R$ and $-R'$.
    Then
    \[
    \mu(E \cup_\psi E', F \cup_\psi F', \tau \cup \tau') = \mu(E, F, \tau) + \mu(E', F', \tau')
    \]
    Here $(E \cup_\psi E',F \cup_\psi F')$ is the bundle pair over $(\Sigma,\partial\Sigma)$ acquired by gluing of $(E,F)$ and $(E',F')$, and $\tau \cup_\psi \tau'$ is the glued trivialization. 
    \item[(e)] (Normalization) If $\Sigma = D$ is the 2-disk with $\partial \Sigma = \partial_\circ \Sigma = S^1$, and $(E,F)$ is the bundle pair
    \[E = D \times \C \quad\text{and}\quad F = \{e^{i\theta} \R: \theta \in \partial D \simeq \R/\pi\Z\}\]
    then the Maslov number is given by $\mu(E,F;\tau) = 1$.
\end{itemize}
\end{proposition}

\begin{proof} To prove existence, let $\bar{F} \to \partial\Sigma$ denote the Lagrangian sub-bundle of $E|_{\partial\Sigma}$ given by
\[\bar{F}|_{\partial_\circ \Sigma} = F \quad\text{and}\quad \bar{F}|_{\partial_\star \Sigma} = \tau^{-1}(\R^n)\]
Then $\mu(E,F;\tau)$ is simply the standard Maslov number $\mu(E,\bar{F})$ of the pair $(E,\bar{F})$ (see \cite[\S C.3]{big_ms}).
\[\mu(E,F;\tau) = \mu(E,\bar{F})\]
The properties (a)-(e) except (d) follow immediately from the analogous properties of the ordinary Maslov number (see \cite[Theorem C.3.5(a)-(d)]{big_ms}). 

To see (d), consider for $(E,\bar{F}) \rightarrow (\Sigma, \partial \Sigma)$, there is already a trivialization of $E$ over the boundary punctures. Extend this to a trivialization of $E$ over the entire boundary. This is possible since all complex bundles over $S^1$ are trivial. In the interior of $\Sigma$, we cut out a disk with boundary $(D,\partial D)$, then the trivialization of $E$ can be extended to $E\setminus D$. We further choose a family of Lagrangians over $E|_{\partial D}$, which we denote by $\tilde{F}$. Then by the additivity property of the Maslov index we have 
\[
\mu (E,F;\tau) = \mu (E|_{\Sigma\setminus D}, \bar{F}\cup \tilde{F}|_{\partial \Sigma \cup -\partial D}) + \mu (E|_D, \tilde{F}).
\]
Further more, we observe $E|_{\Sigma\setminus D} \simeq \Sigma\setminus D \times \C^n$, hence $ \mu (E|_{\Sigma\setminus D}, \bar{F}\cup \tilde{F}|_{\partial \Sigma \cup -\partial D})$ is the sum of the Maslov indices of the Lagrangians $\bar{F} \cup \tilde{F} \subset \C^n$ along $\partial \Sigma \cup -\partial D$. An entirely analogous construction holds for $(E' \bar{F}' ) \rightarrow (\Sigma', \partial \Sigma')$ - we just add prime to all of our previous symbols. Then the act of gluing induces a gluing between $(\partial \Sigma\setminus D, \partial \Sigma \cup -\partial D)$ and $(\partial \Sigma'\setminus D', \partial \Sigma' \cup -\partial D')$, then this implies that 
\[
\mu (E|_{\Sigma\setminus D}, \bar{F}\cup \tilde{F}|_{\partial \Sigma \cup -\partial D}) + \mu (E'|_{\Sigma'\setminus D'}, \bar{F}'\cup \tilde{F}'|_{\partial \Sigma \cup -\partial D}) = \mu (E \cup E' |_{(\Sigma \setminus D)\cup _{\psi} (\Sigma' \setminus D')}, F\cup_\psi F' \cup \tilde{F} \cup \tilde{F}')
\]
because the contributions for the Lagrangians $F$ and $F'$ cancel due to our assumptions. The additivity formula then follows from the composition formula in the regular Maslov index case.
\vspace{3pt}

To prove uniqueness, we argue as follows. First, consider the case where $\Sigma$ is the half-disk
\begin{equation} \label{eqn:half_disk_bundle} \Sigma = \mathbb{D} \cap \mathbb{H} \quad\text{with}\quad \partial_\star \Sigma = \partial \mathbb{D} \cap \mathbb{H} \text{ and }\partial_\circ \Sigma = \mathbb{D} \cap \mathbb{R}\end{equation}
\[E = \C^n \qquad F = \R^n \quad\text{and}\quad \tau \text{ is the tautological trivialization}\]
Then by gluing $(E,F)$ to itself along $\partial_\star \Sigma$ and applying axiom (d), we find that
\[
\mu(E,F;\tau) = \frac{1}{2} \cdot \mu(\C^n,\R^n) = 0
\]
Second, the uniqueness of the ordinary Maslov index \cite[Theorem C.3.5]{big_ms} implies that for any bundle pair $(E,F) \to (\Sigma,\partial\Sigma)$ with $\partial_\star\Sigma = \emptyset$, we have
\[\mu(E,F) = \mu(E,F;\tau)\]
Finally, any bundle pair with punctures $(E,F) \to (\Sigma,\partial\Sigma)$ transforms into a bundle pair with no puncture boundary by gluing on trivial bundles over the half-disk along $\partial_\star \Sigma$. Thus the gluing axiom (d) determines $\mu(E,F;\tau)$. Uniqueness then follows from the uniqueness of the Maslov index for surfaces without boundary punctures.\end{proof}

\begin{remark} In the case where $\partial_\star \Sigma = \emptyset$, we omit the (empty) trivialization from the notation and denote the Maslov number by $\mu(E,F)$.
\end{remark}

\begin{example}[Special Cases] \label{ex:cases_of_Maslov_number} The following integer invariants can be viewed as special cases of the Maslov number.
\begin{itemize}
    \item[(a)] The 1st Chern number of a bundle $E \to \Sigma$ of a surface with boundary with respect to a trivialization $\tau$ along $\partial\Sigma$ is
    \[
    c_1(E,\tau) = \frac{1}{2} \cdot \mu(E,\emptyset;\tau).
    \]
    In keeping with our previous conventions, we take $\partial_*\Sigma$ to be the entire boundary of $\Sigma$, and $\partial_o\Sigma = \emptyset$.
    In particular, if $\Sigma$ is closed then $c_1(E) = \frac{1}{2} \cdot \mu(E,\emptyset)$.
    \item[(b)] \cite[Thm C.3.6]{big_ms} Let $\Sigma$ be a compact surface with boundary and let $F:\partial \Sigma \to \op{LGr}(2n)$ be a set of loops in the Lagrangian Grassmannian. Then
\[
\mu(F) := \mu(\Sigma \times \C^n, F)
\]
    \item[(c)] The Maslov index of a loop $\Phi:S^1 \to \Sp(2n)$ is
\[
\mu(\Phi) := \frac{1}{2} \cdot \mu(D \times \C^n,F) \quad\text{where}\quad F_\theta := \Phi_\theta(\R^n) \subset \C^n
\]
\end{itemize}
\end{example}

\begin{example}[Maslov Class] \label{ex:Maslov_class} Let $(X,L)$ be a pair of a symplectic manifold and a Legrangain sub-manifold $L \subset X$. The \emph{Maslov class}
\[\mu(X,L):H_2(X,L) \to \Z\] 
is the map defined so that the value on the class $A$ of a surface $\Sigma \subset X$ with boundary $\partial \Sigma \subset L$ is the Maslov number of $(TX,TL)|_\Sigma$. 

\vspace{3pt}

The Maslov class $\mu(\xi,\Lambda)$ of a contact manifold $(Y,\xi)$ with Legendrian $\Lambda \subset Y$ is defined similarly, and we have
\[\mu(\xi,\Lambda) = \mu(TX,TL) \qquad\text{if}\qquad (X,L) = ([0,1] \times Y,[0,1] \times \Lambda)\]\end{example}

We will require a few properties of the Maslov number. First, we must record how the Maslov number changes under change of end trivializations.

\begin{lemma}[Trivial Bundle]
\label{lem:lagrangian_loop} 
Consider the a bundle pair with punctures
\[(\C^n,F) \to (\partial\Sigma,\partial_\circ\Sigma)\]
where $F:\partial_\circ \Sigma \to \op{LGr}(2n)$ is a collection of arcs in the Lagrangian Grassmanian with ends on $\R^n$. Then
\[\mu(E,F;\tau) = \mu(F) + \mu(\tau^{-1}(\R^n))\]
Here $\mu(F)$ and $\mu(\tau^{-1}(\R^n))$ are the Maslov indices of the (collections of) loops of Lagrangians $F$ and $\tau^{-1}(\R^n)$.
\end{lemma}

\begin{proof} As in Proposition \ref{prop:maslov_number_of_punctured_pair}, the Maslov index $\mu(E,F;\tau)$ is given by
\[\mu(E,F;\tau) = \mu(E,\bar{F})\]
where $\bar{F}$ is the Lagrangian sub-bundle acquired by taking the union of $F$ over $\partial_\circ \Sigma$ and $\tau^{-1}(\R^n)$ over $\partial_\star\Sigma$. The result thus follows immediately from Example \ref{ex:cases_of_Maslov_number}(b).\end{proof}

As with the Chern number, the Maslov number can be interpreted as a signed count of zeros.

\begin{lemma}[Zero Count]
\label{lem:Maslov_zero_count} Let $(E,F) \to (\Sigma,\partial\Sigma)$ be a bundle pair with punctures equipped with boundary trivialization $\tau$. We further assume $E$ is two dimensional. Let $\psi:(\Sigma,\partial_\circ \Sigma) \to (E,F)$ be a section with
\[
\psi \text{ is transverse to }\Sigma \quad\text{and}\quad \psi(\partial_\star \Sigma) \subset \tau( \R \setminus 0)
\]
Then the Maslov number $\mu(E,F;\tau)$ is given by
\begin{equation}
\mu(E,F;\tau) := 2 \cdot \#(\psi \cap \op{int}(\Sigma)) + \#(\psi \cap \partial_\circ\Sigma)
\end{equation}
\end{lemma}

\begin{proof} Let $-\Sigma$ and $-E$ denote the surface and bundle with reversed orientations, respectively. Recall a trivialization $\tau$ is achoice of a map
\[
 \tau: E|_{\partial_*\Sigma} \rightarrow \C
\]
With this trivialization, as with the case how we defined Maslov indices, we extend $F$ over $\partial_*\Sigma$ as
\[
 \tau: (E,F)|_{\partial_*\Sigma} \rightarrow (\C,\R)
\]
Furthermore, let $\bar{\tau}$ denote the composition of $\tau$ with complex conjugation (which is anti-symplectic).
\[
(E,F)|_{\partial_\star\Sigma} \simeq (\C,\R) \xrightarrow{z \to \bar{z}} (\C,\R)
\]
This composition is a trivialization of $-E$ over $\partial_\star\Sigma$ in $-\Sigma$. Viewing $-E$ as a bundle over $-\Sigma$, we have
\[\mu(-E,F;\bar{\tau}) = \mu(E,F;\tau)\]
Moreover, there is a natural (isotopy class of) symplectic bundle map
\[
(-E,F)|_{\partial_\circ \Sigma} \to (E,F)|_{\partial_\circ \Sigma} \qquad\text{covering} \qquad \op{Id}:\partial_\circ\Sigma \to \partial_\circ\Sigma
\]
given by complex conjugation with respect to the totally real subspace $F$.
Thus, we can glue $E$ and $-E$ to a bundle $DE$ over the double $D\Sigma = \Sigma \cup_{\partial_\circ\Sigma} - \Sigma$ of $\Sigma$ along the boundary region $\partial_\circ\Sigma$. Using the composition property of Maslov class (Theorem C.3,5 in \cite{big_ms} and the fact we have fixed trivializations around punctures so that the totally bundle is trivial) gives us that
\[\mu(E,F;\tau) = \frac{1}{2} \mu(DE,\emptyset;\tau) = c_1(DE,\emptyset;\tau)\]
Now take a section $\psi$ as in the lemma statement. We may assume (after a small isotopy leaving the zeros unchanged) that this section doubles to a section $\phi$ of $DE$ that is transverse to the zero section. The relative Chern class of $DE$ with respect to $\tau$ is computed by the zeros of $\phi$, and thus
\[\mu(E,F;\tau) = c_1(DE,\emptyset;\tau) = \#(\phi \cap D\Sigma) = 2 \cdot \#(\psi \cap \op{int}(\Sigma)) + \#(\psi \cap \partial_\circ \Sigma) \qedhere\]
\end{proof}

\subsection{Euler Characteristic} Consider a compact surface $\Sigma$ with boundary and corners $\partial \Sigma = \partial_\circ \Sigma \cup \partial_\star \Sigma$. In this paper, we will use the following version of the Euler characteristic. 

\begin{definition} \label{def:orbifold_Euler_characteristic} The \emph{orbifold Euler characteristic} $\bar{\chi}(\Sigma)$ is the quantity
\[
\chi(\Sigma) - \frac{1}{2} \chi(\partial_\circ \Sigma) \quad\text{or equivalently} \quad \frac{1}{2} \cdot \chi(\Sigma \cup_{\partial_\circ \Sigma} -\Sigma)
\]
\end{definition}

Like the ordinary Euler characteristic, this quantity is a special case of the Maslov number. Specifically, consider the tangent bundle of $\Sigma$. This naturally has the structure of a bundle pair with punctures.
\[
(T\Sigma,T(\partial_\circ\Sigma)) \to (\Sigma,\partial\Sigma)
\] 
Moreover, this bundle pair comes with a canonical isotopy class of trivialization
\[
\tau_{\partial\Sigma}:(T\Sigma,T(\partial_\circ\Sigma))|_{\partial_\star \Sigma} \simeq (\C,\R) \quad\text{with}\quad \tau_{\partial \Sigma}^{-1}(\R) \pitchfork T(\partial_\star \Sigma)
\]
The Maslov number of the tangent pair in the canonical trivialization is precisely twice the orbifold Euler characteristic.

\begin{lemma}[Euler Characteristic] \label{lem:Euler_characteristic} Let $\Sigma$ be a surface with boundary $\partial \Sigma$. Then
\[
\frac{1}{2} \cdot \mu(T\Sigma,T(\partial_o \Sigma);\tau_{\partial\Sigma}) = \bar{\chi}(\Sigma)
\]
\end{lemma}

\begin{proof} First, consider two special cases. If $\Sigma$ is closed, this is equivalent to the fact that $c_1(T\Sigma) = \chi(\Sigma)$. Likewise, if $\Sigma = \mathbb{D}$ with $\partial_\star\Sigma = \partial\mathbb{D}$, then $T(\partial \mathbb{D}) \subset T\mathbb{D} \simeq \mathbb{D} \times \R^2$ is a loop of Lagrangians with Maslov index $2$. Thus by Example \ref{ex:cases_of_Maslov_number}(c), we have
\[\frac{1}{2} \cdot \mu(T\mathbb{D},T(\partial \mathbb{D})) = 1 = \chi(\mathbb{D})\]
This verifies these special cases. In the general case we can consider the double of $\Sigma$ along $\partial_\circ \Sigma$.
\[
S = \Sigma \cup_{\partial_\circ \Sigma} \bar{\Sigma}
\]
Applying the proof of Lemma \ref{lem:Maslov_zero_count} and the property of the Euler characteristic under gluing, we deduce
\[
\mu(TS,\emptyset;\tau_{\partial S}) = 2 \cdot \mu(T\Sigma,T(\partial_\circ \Sigma);\tau_{\partial \Sigma}) \qquad\text{and}\qquad \chi(S) = 2 \cdot \bar{\chi}(\Sigma)
\]

On the other hand, $S$ can be acquired by taking a closed surface and removing a collection of disks to produce puncture boundary. Thus, by Proposition \ref{prop:maslov_number_of_punctured_pair}(d) and the special cases, we have
\[
\mu(TS,\emptyset;\tau_{\partial S}) = \chi(S)
\qedhere \]
\end{proof}

\subsection{Setup And Trivializations} \label{subsec:setup} Before proceeding, it will be useful for us to fix a common setup for the remainder of this section. For each $\bullet \in \{+,-\}$, we fix
\begin{itemize}
\item[(a)] a compact contact $3$-manifold $(Y_\bullet,\xi_\bullet)$ with boundary $\partial Y_\bullet$.
\item[(b)] a closed Legendrian sub-manifold $\Lambda_\bullet \subset \partial Y_\bullet$ in the boundary of $Y_\bullet$.
\item[(c)] a contact form $\alpha_\bullet$ with non-degenerate Reeb orbits and chords.
\item[(d)] an orbit-chord set $\Xi_\bullet = \Gamma_\bullet \cup C_\bullet$ consisting of a Reeb orbit set $\Gamma_\bullet = \{(\gamma_i^\bullet,m_i^\bullet)\}$ and Reeb chord set $C_\bullet = \{(c_i^\bullet,n_i^\bullet)\}$. 
\item[(e)] a symplectic $4$-dimensional cobordism $X:Y_+ \to Y_-$ with symplectic form $\omega$. This is a manifold with boundary and corners, where $\partial X = \partial_+X \cup \partial_\circ X \cup \partial_-X$ such that $\partial_\bullet X \simeq Y_\bullet$.
\item[(d)] a Lagrangian cobordism $L:\Lambda_+ \to \Lambda_-$ contained in the horizontal boundary $\partial_\circ X$ of $X$.
\end{itemize}
\vspace{3pt}
We will also need to refer to a symplectic trivialization of $\xi$ along the orbits and chords in $\Xi$. We fix the following terminology. 

\begin{definition}[Trivialization] \label{def:trivializations_over_orbit_chord_set} A \emph{trivialization} $\tau$ of $\xi$ over $\Xi$ is a symplectic bundle isomorphism
\[\tau:\xi|\gamma_i \simeq \C \qquad \tau:(\xi,T\Lambda)|c_i \simeq (\C,\R)\]
The difference $\sigma - \tau$ between two trivializations over $\Xi$ is defined by
\[
\sigma - \tau = \sum_i m_i \cdot \mu(\sigma \circ \tau^{-1}|_{\gamma_i}(\R)) + \sum_i n_i \cdot \mu(\sigma \circ \tau^{-1}|_{c_i}(\R))
\]
We use $\mathcal{T}(\Xi_+,\Xi_-)$ to denote the space of pairs $\tau = (\tau_+,\tau_-)$ of trivializations $\tau_+$ over $\Xi_+$ and $\tau_-$ over $\Xi_-$. 
\end{definition} 
\noindent Note that a trivialization $\tau:\xi|_\gamma \simeq \C$ of the contact structure over an orbit induces a trivialization over any iterate of $\gamma$. By abuse of notation, we also denote this trivialization by $\tau$.

\subsection{Surface Classes} We next discuss the set of surface classes associated to the pair $(X,L)$. 

\begin{definition} A \emph{proper surface} $S = (\iota,S)$ in $(X,L)$ from $\Xi_+$ to $\Xi_-$ consists of a compact surface $S$ with boundary $\partial S = \partial_+S \cup \partial_-S \cup \partial_\circ S$ and a continuous map $\iota:S \to X$ such that
\begin{itemize}
	\item[(a)] The restriction of $\iota$ to $\partial_\pm S$ is a map
	\[
	\iota:\partial_\pm S \to Y_\pm = \partial_\pm X
	\]
	consisting of a collection of orbits and chords in $Y_\pm$ representing the orbit set $\Xi_\pm$.
	\item[(b)] The boundary region $\partial_\circ S$ maps to the Lagrangian $L$, i.e. $\iota(\partial_\circ S) \subset L$
\end{itemize}
\vspace{3pt}
A proper surface $(\iota,S)$ will be called \emph{well-immersed} if it also satisfies the following criteria.
\begin{itemize}
    \item[(a)] The restriction of $\iota$ to $S \setminus (\partial_+S \cup \partial_-S)$ is an immersion with transverse double points. Note that these double points may occur on $\partial_\circ S$.
    \item[(b)] The map $\iota$ is an embedding in a neighborhood of $\partial_+S \cup \partial_-S$ (except at $\partial_+S \cup \partial_-S$ itself) and is transverse to $\partial_+S \cup \partial_-S$.
\end{itemize}
\end{definition}

\begin{definition} 
\label{def:surface_classes}
The set $S(\Xi_+,\Xi_-)$ of \emph{surface class} $A:\Xi_+ \to \Xi_-$ is the set of proper maps $\iota:S \to X$ bounding $\Xi_+$ and $\Xi_-$ modulo the equivalence relation that
\[\iota \sim \iota' \qquad\text{if}\qquad [S \cup -S'] = 0 \in H_2(X,L)\]
Here $S \cup -S'$ is the $2$-chain in $X$ with boundary on $L$ represented by the map $\iota \cup \iota':S \cup S' \to X$.\end{definition}

 The set of surface classes $S(\Xi_+,\Xi_-)$ is (tautologically) a torsor over the relative homology group $H_2(X,L)$ and we use the notation
\[
A - B \in H_2(X,L)\quad \text{for the class such that}\quad A = B + (A - B)
\]

\begin{example} \label{ex:trivial_branched_cover} Let $(Y,\Lambda)$ be a closed contact manifold and a Legendrian and fix an orbit-chord set $\Xi$. Consider the cobordism
\[X := Y \times [0,1] \qquad L := \Lambda \times [0,1]\]
A \emph{trivial branched covers} of $\Xi$ will refer to a map from a surface $C$ to $Y \times [0,1]$ of the form
\[
C \xrightarrow{\kappa} \Xi \times [0,1] \subset Y \times [0,1]
\]
Here $\kappa$ is a branched cover whose covering multiplicity at $\Xi_+$ and $\Xi_-$ satisfies Definition \ref{def:surface_classes}. 

\vspace{3pt}

Any two collection of orbits and chords representing $\Xi$ (i.e. with the same underlying simple orbits and chords with multiplicity) can be connected by a trivial branched cover. Furthermore, the surface classes of all such maps all agree since $H_2(\Xi,\partial \Xi) = 0$.
\end{example}

There are some important operations on proper surfaces and surface classes that we will use in later parts of this paper.

\begin{definition}[Union] The \emph{union} $S \cup T$ of two proper surfaces $(\iota.S)$ and $(\jmath,T)$ in $(X,L)$ is
\[
\iota \sqcup \jmath: S \sqcup T \to (X,L)
\]
This descends to a map on the space of surface classes of the form
\[S(\Xi_0,\Theta_0) \times S(\Xi_1,\Theta_1) \to S(\Xi_0 \cup \Xi_1,\Theta_0 \cup \Theta_1) \qquad\text{with}\qquad (A,B) \mapsto A \cup B\]\end{definition}

\begin{definition} A \emph{composition} $S \circ T$ of two proper surfaces $S:\Xi_0 \to \Xi_1$ in $(X,L)$ and $T:\Xi_1 \to \Xi_2$ in $(X',L')$ is proper surface of the form
\[S \cup Z \cup T \subset X \circ X'\]
where $Z$ is a trivial branched cover of $\Xi$ connecting $\partial_-S$ and $\partial_+T$. This descends to a map of surface classes of the form
\[S(\Xi_0,\Xi_1) \times S(\Xi_1,\Xi_2) \to S(\Xi_0,\Xi_1) \qquad\text{given by}\qquad (A,B) \to A \circ B\]
\end{definition}

\subsection{Maslov Number Of A Surface Class} Given any proper surface $S$ in $(X,L)$, there is a natural bundle pair with punctures over $S$ induced by pullback.
\[
(\iota^*TX, \iota^*TL) \to (S,\partial_\circ S)
\]
We extend the trivialization of $\xi_{c_i}$ to trivializations of $\iota^*TX|_{c_i}$. And similarly we extend trivilizations of $\xi_{\gamma_i}$ to $\iota^*TX|_{\gamma_i}$. We here describe our choice for such an extension (this is essentially the same choice has we have for Lemma \ref{lem:maslov_splitting}.) We assume near $Y_\pm$ the cobordism $X$ has collar neighborhoods of the form $(1-\epsilon,1]\times Y_+$ and $[0,\epsilon) \times Y_-$. We call the direction in the half open interval the symplectization direction. Then $\xi$ and the plane given by the symplectization direction and the Reeb direction on over chords and orbits span $TX$. We extend $\tau$ in this way to $TX$ over both chords and orbits. We further assume that with this choice of trivialization $\tau: \iota^*(TX)|_{c_i} \rightarrow \C^2|_{c_i}$, we have $\tau (\iota^*(TX)|_{\partial_\pm S \cap \partial_o S} = \R^2 \subset \C^2$.

\begin{definition}\label{def:Maslov_number_of_surface_class} The \emph{Maslov number} $\mu_\tau(A,\tau)$ of a surface class $A$ with respect to a trivialization $\tau$ over $\Xi$ is given by
\[
\mu(A,\tau) := \mu(\iota^*TX, \iota^*TL;\tau) \qquad\text{for any representative}\qquad \iota:S \to X
\]
Here $\mu(\iota^*TX, \iota^*TL;\tau)$ is the Maslov index of $(\iota^*TX, \iota^*TL)$ as a bundle pair (see Proposition \ref{prop:maslov_number_of_punctured_pair}).
\end{definition}

A priori, it is not clear that $\mu(A,\tau)$ is independent of the choice of representative. We now prove that this is the case.

\begin{proposition}\label{prop:properties_of_Maslov_number} The Maslov number $\mu(A,\tau)$ is well-defined and has the following properties.
\begin{itemize}
	\item (Trivialization) Let $\sigma$ and $\tau$ be two trivializations in $\mathcal{T}(\Xi_+,\Xi_-)$. Then
	\[
    \mu(A,\tau) - \mu(A,\sigma) = (\sigma_+ - \tau_+) - (\sigma_- - \tau_-)
	\]
 	\item (Maslov Class) Let $A:\Xi \to \Theta$ and $B:\Xi \to \Theta$ be a pair of surface classes. Then
	\[\mu(A,\tau) - \mu(B,\tau) = \mu(X,L) \cdot (A - B)\]
       \item (Union) Let $A:\Xi_0 \to \Theta_0$ and $B:\Xi_1 \to \Theta_1$ be a pair of surface classes and let $\tau$ be a trivialization that agrees on $\Xi_0 \cap \Xi_1$ and $\Theta_0 \cap \Theta_1$. Then
       \[\mu(A \cup B,\tau) = \mu(A,\tau) + \mu(B,\tau)\]
       
       \item (Composition) Let $A:\Xi_0 \to \Xi_1$ and $B:\Xi_1 \to \Xi_2$ be composible surface classes, and let $\tau$ be a trivialization along $\Xi_0 \cup \Xi_1 \cup \Xi_2$. Then
       \[\mu(A \circ B,\tau) = \mu(A,\tau) + \mu(B,\tau)\]
\end{itemize}
\end{proposition}

\noindent To proceed with the proof, we first consider the case of branched covers as in Example \ref{ex:trivial_branched_cover}.

\begin{lemma} \label{lem:properties-of_Maslov_branced_cylinder_case} Let $\Xi$ be an orbit-chord set in $Y$ and consider a trivial branched cover of $\Xi$, as in Example \ref{ex:trivial_branched_cover}. 
\[\iota:(C,\partial_\circ C) \to (X,L) = ([0,1] \times Y, [0,1] \times \Lambda)\]
Choose a pair of trivializations $\sigma$ over $\Xi_- = \Xi \times \{0\}$ and $\tau$ over  $\Xi_+ = \Xi \times \{1\}$. Then
\[
\mu(\iota^*TX,\iota^*TL;\iota^*(\sigma \cup \tau)) = \tau - \sigma
\]
Here $\iota^*(\sigma \cup \tau)$ is the trivialization of $\iota^*TX$ over $\partial_\star C = \partial_+ C \cup \partial_- C$ induced by $\sigma$ and $\tau$.
\end{lemma}

\begin{proof} The trivialization $\sigma$ extends to a trivialization
\[(\xi,T\Lambda) \simeq (\C,\R) \qquad \text{over}\qquad \Xi \times [0,1]\]
This pulls back to a trivialization of $\iota^*\xi$ over $C$. Then $\tau$ and $\sigma$ are identified, respectively, with
\[
\iota^*(\tau \circ \sigma^{-1}) \quad \text{along}\quad\partial_+C = \iota^{-1}(\partial_+X) \qquad\text{and}\qquad \op{Id} \quad \text{on}\quad\partial_-C = \iota^{-1}(\partial_-X)
\]
Now decompose $\partial_+C$ into regions $S_i = \iota^{-1}(\gamma_i)$ and $T_i = \iota^{-1}(c_i)$. Then as a simple application of Lemma \ref{lem:lagrangian_loop}, we find that
\[
\mu(\iota^*\xi,\iota^*T\Lambda;\iota^*(\sigma \cup \tau)) = \sum_i \op{deg}(\iota|_{S_i}) \cdot \mu(\tau \circ \sigma^{-1}|_{S_i}) + \op{deg}(\iota|_{T_i}) \cdot \mu(\tau \circ \sigma^{-1}|_{T_i})
\]
This is precisely the difference $\tau - \sigma$ of the trivializations over $\Xi$. The same formula follows for $(\iota^*TX,\iota^*TL)$ since it is isomorphic to the direct sum of $(\iota^*\xi,\iota^*T\Lambda)$ and a trivial bundle pair $(\C,\R)$ (spanned by the $\R$-direction and Reeb direction in $Y \times [0,1])$.
\end{proof}

\begin{proof} (Proposition \ref{prop:properties_of_Maslov_number}) The union and composition properties follow immediately from the disjoint union and puncture gluing properties of the Maslov number (Proposition \ref{prop:maslov_number_of_punctured_pair}). Here we will argue the other two properties.

\vspace{3pt}

{\bf Well-Definedness and Maslov Class.} We prove well-definedness and the Maslov class property together, since the argument is the same. Choose representatives of $A$ and $B$ in $S(\Xi_+,\Xi_-)$ respectively.
\[\iota:S \to X \qquad\text{and}\qquad \jmath:T \to X\]
Since $S$ and $T$ bound the same orbit sets $\Xi_+$ and $\Xi_-$, there exist trivial branched covers
\[\kappa_\bullet:C_\bullet \to [0,1] \times \Xi_\bullet \subset [0,1] \times Y_\bullet \qquad\text{for} \qquad \bullet \in \{+,-\}\]
such that $\kappa_+$ connects the positive ends of $S$ and $T$ and $\kappa_-$ connects the negative ends of $S$ and $T$. We can form a new surface
\[
\Sigma = S \cup C_+ \cup C_- \cup -T
\]
by gluing $S$ and $T$ to $C_+$ along their common orbit-chord ends, and likewise at $C_-$ (see Figure \ref{fig:maslov}).
\begin{figure}
    \centering
    \includegraphics[width=.8\textwidth]{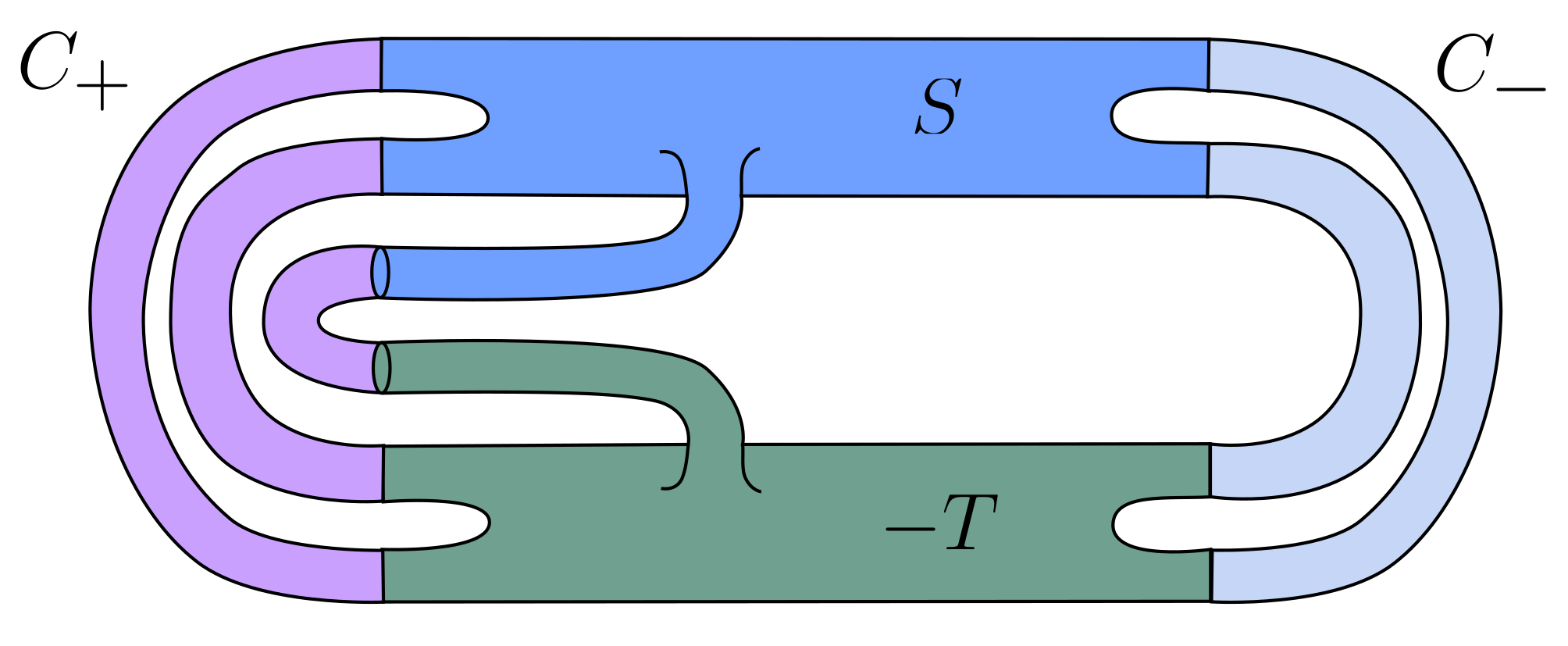}
    \caption{The surface $\Sigma$}
    \label{fig:maslov}
\end{figure}
This surface inherits a continuous map
\[
\kappa:(\Sigma,\partial \Sigma) \to (X,L)
\]
restricting to $\iota$ on $S$, to $\jmath$ on $T$ and to the map $C_\pm \to \Xi_\pm$ on $C_\pm$. By the gluing property of the Maslov number in Proposition \ref{prop:maslov_number_of_punctured_pair}(d), we have 
\begin{equation} \label{eqn:properties_of_maslov_number_1}
\mu(\kappa^*TX,\kappa^*TL) = \mu(\iota^*TX,\iota^*TL;\tau) - \mu(\jmath^*TX,\jmath^*TL;\tau) + \mu(TX|_{C_+},F_{C_+}) + \mu(TX|_{C_-},F_{C_-})
\end{equation}
The left hand side is simply $\mu(X,L) \cdot [\Sigma]$ where $\mu(X,L)$ is the Maslov class (see Example \ref{ex:Maslov_class}) and $[\Sigma] = A - B$. On the right hand side, the cylinder $C_+$ is equipped with the trivialization $\tau|_{\Xi_+}$ on both ends and likewise for $C_-$, so by Lemma \ref{lem:properties-of_Maslov_branced_cylinder_case}
\[\mu(TX|_{C_+},F_{C_+}) = 0 \qquad\text{and} \qquad \mu(TX|_{C_+},F_{C_-}) = 0\]
Therefore, (a) follows from (\ref{eqn:properties_of_maslov_number_1}).

\vspace{3pt}

{\bf Trivialization.} Fix a representative $S$ of $A$ as in (a) and let $C_\pm$ be the trivial (unbranched cover) of $\Xi_\pm \times [0,1]$, as in Example \ref{ex:trivial_branched_cover}. Equip $C_\pm$ with the trivialization $\rho_\pm$ with
\[
\rho_\pm|_{\Xi_\pm \times 0} = \sigma|_{\Xi_\pm} \qquad\text{and}\qquad \rho_\pm|_{\Xi_\pm \times 1} = \tau|_{\Xi_\pm}
\]
We may form a glued surface $T = S \cup C_+ \cup C_-$ and a map $\jmath:T \to X$ representing the class $A$ by a similar process to (a). By applying the puncture gluing property Proposition \ref{prop:maslov_number_of_punctured_pair}(c), we have
\[
\mu(A,\tau) = \mu(A,\sigma) + \mu(C_+,\rho_+) - \mu(C_-,\rho_-)
\]
Therefore, the result follows from \ref{lem:properties-of_Maslov_branced_cylinder_case} since
\[
\mu(C_+,\rho_+) = \mu(\sigma \circ \tau^{-1}|_{\Xi_+}\R) \qquad\text{and}\qquad \mu(C_-,\rho_-) = \mu(\sigma \circ \tau^{-1}|_{\Xi_-}\R)
\qedhere\]
\end{proof} 

\vspace{3pt}

We conclude this section by considering the Maslov number of a symplectic, well-immersed proper surface $(\iota,S)$ with class $A \in S(\Xi_+,\Xi_-)$. In this case, we have a decomposition
\begin{equation} (TX,TL) = (T\Sigma,T(\partial_\circ\Sigma)) \oplus (\nu S, \nu(\partial_\circ S)) \end{equation}
Here $\nu S$ is the symplectic perpendicular to $T\Sigma$ and $\nu(\partial S)$ is any choice of transverse sub-bundle to $T(\partial_\circ S)$ in $TL$.

\vspace{3pt}

Note that $\nu S$ is canonically identified with $\xi_\pm$ near the positive and negative ends of $(\iota,S)$. In particular, a choice of trivialization $\tau$ over $\Xi$ induces a puncture trivialization of $(TX,TL)$ (also denoted $\tau$) by direct sum with the canonical puncture trivialization of the tangent pair $(T\Sigma,T(\partial_\circ \Sigma))$. By the direct sum property of the Maslov number, we thus acquire

\begin{lemma} \label{lem:maslov_splitting} Let $(\iota,S)$ be a symplectic, well-immersed proper surface. Then
\[
\mu(S,\tau) = 2 \bar{\chi}(S) + \mu(\nu S, \nu(\partial_\circ S);\tau)
\]
\end{lemma}

\subsection{Writhe and Linking} We next introduce the notion of the writhe and linking number of a braid, and of an admissible representative of a surface class. 

\vspace{3pt}

To introduce these concepts, we must first clarify the notion of braid that we will use.

\begin{definition} A \emph{braid} $\zeta$ around a Reeb chord or orbit $\eta$ in $Y$ is a disjoint union of arcs (if $\eta$ is a chord) or of loops (if $\eta$ is an orbit) in a tubular neighborhood $\op{Nbhd}(\eta)$ of $\eta$ such that
\[
 \zeta \text{ is transverse to } \xi \qquad\text{and}\qquad \partial\zeta \in \Lambda \text{ if $\eta$ is a chord}
\]
\end{definition}

\begin{definition} Let $S$ be a well immersed surface representative of a surface class $A \in S(\Xi_+,\Xi_-)$. Suppoose for each chord or orbit $\eta_i^\pm$ in the set $\Xi_+$, there is an associated (isotopy class of) braid
\[\zeta^\pm_i \quad\text{around}\quad \eta^\pm_i\]
constructed as follows. Choose a collar neighborhood of pairs
\[(Y_\pm \times [0,1),\Lambda \times [0,1)) \to (X,L) \qquad\text{near}\qquad \partial_\pm X\]
Then, under this identification, we define $\zeta^\pm_i$ to be the intersection
\[
\zeta^\pm_i := S \cap (\op{Nbhd}(\eta) \times \{\epsilon\}) 
\]
Here $\op{Nbhd}(\eta)$ is a small tubular neighborhood disjoint from other orbits and chords in $\Xi$ determined by a choice of trivialization of the contact structure along the orbit or chord:
\[
\phi_\tau:\op{Nbhd}(\eta) \simeq [0,1] \times D^2 \qquad\text{or}\qquad \phi_\tau:\op{Nbhd}(\eta) \simeq S^1 \times D^2
\]
In the chord case, $\phi_\tau$ maps the boundary $\partial\eta$ to the $x$-axis $(\R \times 0) \cap D^2$ in $0 \times D^2$ and $1 \times D^2$. We take  $\epsilon > 0$ to be small. Suppose the braids constructed as above does not depend on $\epsilon$ if it is small enough. Then we say the surface $S$ is an admissible representative. 
We collectively refer to these braids as the \emph{ends} of $S$.
\end{definition}
Given a trivialization $\tau$ along Reeb orbits or chords, as in the above definition, this gives rise to a choice of tubular neighborhood around the chord or Reeb orbit of the form $[0,1] \times D^2$ (resp. $S^1 \times D^2$). 
Via the projection map $D^2 \to \R$ to the $x$-axis, we acquire projection maps
\[
\pi:[0,1] \times D^2 \to [0,1] \times \R \qquad\text{and}\qquad \pi:S^1 \times D^2 \to S^1 \times \R
\]

\begin{definition}[Writhe] \label{def:writhe} The \emph{writhe} $w_\tau(\zeta) \in \Z$ of a braid $\zeta$ around $\eta$ with respect to a trivialization $\tau$ is the signed count of self-intersections of the curve
\[\pi \circ \phi_\tau(\tilde{\zeta}) \]
Here $\tilde{\zeta}$ is a perturbation of $\zeta$  relative to $\partial \zeta$ such that $\pi \circ \phi_\tau(\zeta')$ has only transverse double points. 
The sign convention in ECH is that anticlockwise rotations in the disk $D^2$ contributes a positive crossing. This is described, for instance after Definition 2.7 in \cite{Hutchings_index_revisited}

\vspace{3pt}

Likewise, the \emph{writhe} $w_\tau(S)$ of an admissible representative $S$ of a surface class $A$ with respect to a trivialization $\tau \in \mathcal{T}(\Xi_+,\Xi_-)$ is the signed sum of writhes of the ends.
\begin{equation} \label{eqn:writhe_surface_def} 
w_\tau(S) := \sum_i w_\tau(\zeta^+_i) - \sum_j w_\tau(\zeta^-_j)\end{equation}
\end{definition}

\begin{definition} The \emph{linking number} $l_\tau(\zeta,\zeta') \in \frac{1}{2}\Z$ of a pair of disjoint braids $\zeta$ and $\zeta'$ around $\eta$ is half of the signed count of intersections between the pair of curves
\[\pi \circ \phi_\tau(\zeta) \qquad\text{and}\qquad \pi \circ \phi_\tau(\zeta')\]
The \emph{linking number} $l_\tau(S,S')$ of a pair of disjoint, admissible representatives of two surface classes $A$ and $B$ is the analogous signed sum of linking numbers to (\ref{eqn:writhe_surface_def}).\end{definition}

We will require a few elementary properties of the linking number and the writhe. 

\begin{lemma} \label{lem:writhe_link_properties} The linking number and the writhe satisfy the following properties.
\begin{itemize}

\item (Union) The linking number and writhe are related by the formula
\begin{equation}
\label{eq:relationship_writhe_linking}
w_\tau(\zeta \cup \zeta') = w_\tau(\zeta) + w_\tau(\zeta') + 2 \cdot l_\tau(\zeta,\zeta')\end{equation}
\item (Trivialization) If $\sigma$ and $\tau$ are two trivializations of $\xi$ along $\eta$ then
\[
w_\tau(\zeta) - w_\sigma(\zeta) = m(m-1)\cdot (\tau - \sigma)
\]
where $m$ is the number of strands in the braid $\zeta$.
\end{itemize}
\end{lemma}

\subsection{Classical Intersection Pairing} We next consider a half-integer valued intersection pairing associated to a pair of a $4$-manifold with a $2$-manifold in its boundary. This is analogous to the intersection pairing of closed $4$-manifolds.

\begin{definition}[Intersection] \label{def:classical_intersection_pairing} Let $M$ be a compact 4-manifold with corners and let $N \subset \partial M$ be an oriented embedded $2$-manifold with boundary. The \emph{intersection pairing}
\[Q:H_2(M,N) \otimes H_2(M,N) \to \frac{1}{2}\Z\]
is defined as follows. Given classes $A,B \in H_2(M,N)$, choose immersed representatives
\[\iota:(S,\partial S) \to (M,N) \qquad\text{and}\qquad \jmath:(T,\partial T) \to (M,N)\]
of $A$ and $B$, respectively, that intersect transversely (including along $N$). Then each intersection point $p \in S \cap T$ is isolated and has a well-defined sign $\iota(M,S,T,p) = \pm 1$ (cf. \cite{w2017}). We let
\[\#_M(\op{int}(S) \cap \op{int}(T)) = \sum_{p \in \op{int}(S) \cap \op{int}(T)} \iota(M,S,T,p) \qquad\text{and}\qquad \#_M(\partial S \cap \partial T) = \sum_{p \in \partial S \cap \partial T} \iota(M,S,T,p)\]
The boundaries $\partial S$ and $\partial T$ are also transverse as sub-manifolds of $N$. Thus, we have well-defined intersect numbers in $N$ for each $p \in \partial S \cap \partial T$, denoted
\[\iota(N,\partial S,\partial T,p) = \pm 1 \qquad\text{and}\qquad \#_N(\partial S \cap \partial T) = \sum_{p \in \partial S \cap \partial T} \iota(N,\partial S,\partial T,p)\]
After an isotopy of $S$ (or $T$) that fixes $\partial S$, we may assume that
\begin{equation} \label{eqn:equal_intersection} 
\iota(M,S,T,p) = \iota(N,\partial S, \partial T,p) \qquad\text{for each}\qquad p \in S \cap T
\end{equation}
Note that this isotopy may introduce intersections in the interior of $S$ and $T$, but we may perturb $S$ and $T$ to be transverse to each other. We now define the intersection pairing $Q$ as follows.
\begin{equation}  \label{eqn:intersection_pairing_def}
Q(A,B) := \#_M(\op{int}(S) \cap \op{int}(T)) + \frac{1}{2} \cdot \#_M(\partial S \cap \partial T)
\end{equation}
\end{definition}

In general, the intersection numbers $\#_N(\partial S \cap \partial T)$ and $\#_M(\partial S \cap \partial T)$ may not agree. However, we can always isotope $S$ so that this is the case.

\begin{lemma} \label{lem:boundary_intersection_replacement} Let $p \in \partial S \cap \partial T$ be a boundary intersection and assume that
\[\iota(M,S,T,p) = -\iota(N,\partial S,\partial T,p)\]
Then there is a small neighborhood $U$ of $p$ and a surface $S'$ isotopic to $S$ such that
\[\partial S' = \partial S \qquad S = S' \text{ on } M \setminus U \qquad S' \cap T \cap U = \{p,q\}\]
Here the boundary intersection point $p$ has intersection numbers
\[\iota(N,\partial S', \partial T,p) = -\iota(N,\partial S,\partial T,p)\]
and $q$ is an interior intersection point such that
\[
\iota(M,S',T,q) = \iota(M,S,T,p)
\]
\end{lemma}

\begin{proof} By choosing coordinates and (possibly) reversing the orientation of $T$, we can reduce to the local picture where $M$ and $N$ are given by
\[M := \{(z_1,z_2) \; : \; \op{Im}(z_1 + z_2) \ge 0\} \subset \mathbb{C}^2 \quad\text{and}\quad N = \mathbb{R}^2\]
Moreover, we may assume that $S$ and $T$ take the form
\[S := (\C \oplus 0) \cap M \qquad\text{and}\qquad T := (0 \oplus \C) \cap M\]
Now consider the sub-space
\[
S'' := \R \oplus i\R \cap M \qquad\text{and}\qquad T'' := i\R \oplus \R \cap M
\]
Note that $S'' \cap T'' = S \cap T = (0,0)$ and we have
\[
\#_M(S \cap T) = -\#_M(S'' \cap T'') = 1 \qquad\text{and}\qquad \#_N(S \cap T) = \#_N(S'' \cap T'') = 1
\]
Now note that there is a pair of braids in the hemisphere $S^3 \cap M$ given by
\[B = (S \cup T) \cap S^3 \qquad\text{and}\qquad B'' = (S'' \cup T'') \cap S^3.\]
We can define homotopies of the components $S_t$ from $S$ to $S''$ and $T_t$ from $T$ to $T''$ as follows.
\[S_t = \text{span}_\R(1 \oplus 0, v_t) \cap S^3 \cap M \qquad\text{where}\qquad v_t = \cos(\frac{\pi t}{2})\cdot (i \oplus 0) + \sin(\frac{\pi t}{2}) \cdot (0 \oplus i) \]
\[T_t = \text{span}_\R(0 \oplus 1, w_t) \cap S^3 \cap M \qquad\text{where}\qquad w_t = \cos(\frac{\pi t}{2})\cdot (0 \oplus i) + \sin(\frac{\pi t}{2}) \cdot (i \oplus 0)\]
Note that $S_t$ and $T_t$ are disjoint, except at $t = \frac{1}{2}$. We can use $S_t$ and $T_t$ to form a pair of surfaces
\[
\Sigma(S) = \{(2 - t) \cdot z \; : \;0 \le t \le 1 \text{ and } z \in S_t\}
\]
\[
\Sigma(T) = \{(2 - t) \cdot z \; : \;0 \le t \le 1 \text{ and } z \in T_t\}
\]
These surfaces intersect at one point with sign $-1$. We now let $S'$ and $T'$ be, respectively, smoothings of the $C^0$ embedded surfaces
\[
(S \cap B^3(1)) \cup \Sigma(S) \cup (S'' \setminus B^3(2)) \qquad\text{and}\qquad (T \cap B^3(1)) \cup \Sigma(S) \cup (T'' \setminus B^3(2)) 
\]
These surfaces are smooth except along some curves that are disjoint from their intersections, so the intersections of $S'$ and $T'$ are given by
\[\op{int}(S') \cap \op{int}(T') = S'' \cap T'' = \{0\} \qquad \partial S' \cap \partial T' = \Sigma(S) \cap \Sigma(T) = \{\frac{3i}{2} \cdot (1,1)\}\]
This is precisely the local picture required in the lemma, so we are done. \end{proof}

\begin{corollary} \label{cor:boundary_intersection_replacement} There are isotopic surfaces $S'$ and $T'$ to $S$ and $T$ respectively such that
\[
Q(S',T') = Q(S,T) \qquad\text{and}\qquad \#_M(\partial S \cap \partial T) = \#_N(\partial S \cap T\partial T)
\]
\end{corollary}


\begin{proposition} \label{prop:classical_Q_well_defined} The intersection pairing $Q$ in Definition \ref{def:classical_intersection_pairing} is well-defined.
\end{proposition}

\begin{proof} Fix classes $A,B \in H_2(M,N)$. Choose immersed representatives
\[\iota:(S,\partial S) \to (M,N) \qquad \iota':(S',\partial S') \to (M,N) \quad\text{and}\quad \jmath:(T,\partial T) \to (M,N)\]
of the classes $A,A$ and $B$, respectively, such that $S$ and $S'$ are both transverse to $T$. Moreover, we may assume by Corollary \ref{cor:boundary_intersection_replacement}  that
\[\#_M(S \cap T) = \#_N(S \cap T) \qquad\text{and}\qquad \#_M(S' \cap T) = \#_N(S' \cap T)\]
To prove well-definedness, it now suffices to show that
\begin{equation} \label{eqn:classical_Q_proof_1} Q(S,T) - Q(S',T) = 0\end{equation}

To prove (\ref{eqn:classical_Q_proof_1}), we must consider the standard short exact sequence of the pair $(M,N)$.
\begin{equation} \label{eqn:rel_exact_sequence_M_N}
H_2(N) \to H_2(M) \to H_2(M,N) \to H_1(N)
\end{equation}
By hypothesis, we have an immersion of $S \cup -S'$ in the homology class $\iota_*[S] - \jmath_*[S'] = 0$ in $H_2(M,N)$. Therefore, $\partial S \cup -\partial S'$ is an immersed null-homologous $1$-manifold in $N$, and
\begin{equation} \label{eqn:classical_Q_proof_2} \#((\partial S\cup -\partial S') \cap \partial T) = 0\end{equation}
Furthermore, we may choose a map $\kappa:R \to N$ from a compact surface $R$ bounding $\partial S \cup -\partial S'$, and acquire an immersion of a closed surface.
\[
f:\Sigma = S \cup R \cup -S' \to X
\]
The class $[\Sigma] \in H_2(M)$ maps to $[S \cup -S] = 0 \in H_2(M,N)$ under the map $H_2(M) \to H_2(M,N)$ in (\ref{eqn:rel_exact_sequence_M_N}). Thus by the exactness of (\ref{eqn:rel_exact_sequence_M_N}), we know that $[\Sigma \cup Z] = 0$ for some closed immersed surface $Z \to N$. We may assume that $[\Sigma] = 0$ by absorbing $Z$ into the choice of bounding surface $R$.

\vspace{3pt}

Now choose a collar neighborhood $N \times (-1,1)_t$ of $N$ in $\partial M$ and extend this to a collar into $M$ as $U = N \times (-1,1)_t \times [0,1)_s$. By choosing this collar to be very small, we can assume that
\[
U \cap (f(\Sigma) \cap \jmath(T)) = \jmath(T) \cap (R \times 0 \times 0) \subset N \times 0 \times 0
\]
We can choose a vector-field $v = \phi(s) \cdot \partial_t$ where $\phi:[0,1) \to [0,1)$ is a non-negative function with $\phi = 0$ near $s = 1$. By flowing $\Sigma$ along $v$ for a small amount of time, we perturb $f$ to a new map $f'$ such that
\[f'(\Sigma) \cap T \cap U = \emptyset\]

\vspace{3pt}

Finally, we can smooth $f'$ to a smooth immersion $\phi:\Sigma \to \op{int}(X)$ agreeing with $f'$ away from a neighborhood of $\partial X$ that satisfies
\[
\#(\Sigma \cap T) = \#(S \cap T) - \#(S' \cap T) = Q(S,T) - Q(S',T)
\]
We now simply observe that since $[\Sigma] = [S] - [S'] = 0$, we have
\[
Q(S,T) - Q(S',T)= \#(\Sigma \cap T) = \op{PD}[T] \cdot [\Sigma]  = 0 \text{ where }[T] \in H_2(X,\partial X) \qedhere
\]
\end{proof}

There is also an analogue of the intersection pairing for a $3$-manifold with boundary, equipped with a $1$-manifold in the boundary. 

\begin{definition} \label{def:intersection_3_manifold} Let $Y$ be a compact $3$-manifold with boundary and corners, and let $Z \subset \partial Y$ be a closed $1$-manifold in $\partial Y$. The \emph{intersection pairing}
\[Q:H_1(Y,Z) \otimes H_2(Y,Z) \to \frac{1}{2}\Z\]
is defined as follows. Consider the $4$-manifold
\[M = [0,1] \times Y \quad \text{with sub-manifold} \quad N = [0,1] \times Z \subset \partial X\]
Let $S = [0,1] \times \Gamma$ where $\Gamma \subset Y$ is a $1$-manifold representing a class $A \in H_1(Y,Z)$ and let $T$ be a $2$-manifold representing $B$ in $H_2(M,N) \simeq H_2(Y,Z)$ contained in $s \times Y$ for $s \in (0,1)$, and intersecting $\Gamma$ transversely. 
\[Q(A,B) := \#_M(\op{int}(S) \cap \op{int}(T)) + \frac{1}{2} \cdot \#(\partial S \cap \partial T)\] 
\end{definition}

\noindent The proof of well-definedness is analogous to Proposition \ref{prop:classical_Q_well_defined}.

\subsection{Relative Intersection Number} Next, we generalize the classical intersection number to an intersection number for surface classes.

\begin{definition} \label{def:intersection_map} Let $(X,L)$ be a pair of a symplectic cobordism with boundary and a Lagrangian cobordism $L \subset X$, and let $A \in S(\Xi,\Theta)$ be a surface class. The associated \emph{intersection map}
\[q_A:H_2(X,L) \to \frac{1}{2}\Z \qquad\text{given by}\qquad B \mapsto q_A(B)\]
is defined as the same count of intersections in (\ref{eqn:intersection_pairing_def}) where $S$ is a representative of the surface class $A$ and $T$ is a compact surface with boundary in the interior of $X$ representing $B$. \end{definition}

 The proof of well-definedness is analogous to Proposition \ref{prop:classical_Q_well_defined}. Moreover, we can prove the following lemma.

\begin{lemma} \label{lem:intersection_3fld_vs_cobordism} Consider a trivial symplectic cobordism pair $(X,L)$ of the form
\[
([0,1] \times Y,[0,1] \times \Lambda) \qquad\text{for a contact manifold $Y$ with Legendrian $\Lambda \subset \partial Y$}
\]
Let $A \in S(\Xi,\Theta)$ be any surface class between orbit-chord sets in $1$-homology class $\Gamma = [\Xi] = [\Theta]$. Then
\[q_A(B) = Q(\Gamma,B) \qquad\text{for any}\qquad B \in H_2(X,L) \simeq H_2(Y,\Lambda)\]
\end{lemma}

\begin{proof} Choose a representative surface $S \subset [0,1] \times Y$ of the surface class $A$ that is a cylinder $[0,\epsilon) \times \eta$ over a $1$-manifold $\eta$ representing $\Gamma$ in $[0,\epsilon) \times Y$. We may choose an immersed $2$-manifold $T \subset Y$ with boundary $\partial T \subset \Lambda$ transverse to $\eta$. Then the corresponding count of intersections computes both $q_A(B)$ and $Q(\Gamma,B)$.
\end{proof}

More generally, we define the intersection number between two surface classes as follows. 

\begin{definition}[Relative Intersection] \label{def:relative_intersection} Fix a pair of surface classes 
\[A:\Xi \to \Theta \qquad\text{and}\qquad B:\Xi' \to \Theta'\]
and a trivialization $\tau$ of $\xi$ along $\Xi \cup \Xi'$ and $\Theta \cup \Theta'$. The \emph{relative intersection pairing} of $A$ and $B$ with respect to $\tau$ is the half-integer
\[Q_\tau(A,B) \in \frac{1}{2}\Z\]
is defined as follows. Pick admissible surfaces $S$ and $T$ representing $A$ and $B$, respectively.
\[\iota:S \to X \qquad\text{and}\qquad \jmath:T \to X\]
Assume that $S$ and $T$ are disjoint near $\partial_\pm X$ (except along $\partial_\pm X$) and transversely intersecting away from $\partial_+X \cup \partial_-X$. Then we let
\begin{equation}
Q_\tau(A,B) := \#(\op{int}(S) \cap \op{int}(T)) + \frac{1}{2} \cdot \#(\partial_\circ S \cap \partial_\circ T) - l_\tau(S,T)
\end{equation}
\end{definition}

The relative intersection number satisfies a number of very useful axioms. We now prove these axioms in detail.

\begin{proposition} \label{prop:properties_of_relative_intersection} The relative intersection pairing $Q_\tau$ is well-defined and has the following properties.
\begin{itemize}
    \item (Trivialization)  Let $A:\Xi_+ \to \Xi_-$ and $B:\Theta_+ \to \Theta_-$ be two surface classes and let $\sigma$ and $\tau$ be two trivializations that differ only along one orbit or chord
    \[\eta \quad\text{of multiplicity $m$ in $\Xi$ and $n$ in $\Theta$}\]
    Then the self-intersection numbers of $\sigma$ and $\tau$ differ as follows.
    \[Q_\tau(A,B) - Q_\sigma(A,B) = m \cdot n \cdot (\tau - \sigma)\]
    \item (Difference) Let $A,A':\Xi_+ \to \Xi_-$ and $B \in S(\Psi_+,\Psi_-)$ be surface classes between the same orbit-chord sets. Then
    \[
    Q_\tau(A,B) - Q_\tau(A',B) = q_B(A - A')
    \]
    \item (Union) Let $A,A'$ and $B$ be surface classes. Then
    \[
    Q_\tau(A \cup A',B) = Q_\tau(A,B) + Q_\tau(A',B)
    \]
    \item (Composition) Let $A,A':\Xi_0 \to \Xi_1$ and $B,B':\Xi_1 \to \Xi_2$ be two pairs of composible surface classes in $(X,L)$ and $(X',L')$ respectively. Then
    \[Q_\tau(A \circ B,A' \circ B') = Q_\tau(A,A') + Q_\tau(B,B')\]
\end{itemize}
\end{proposition}

In order to prove these axioms, we will need the following lemma.

\begin{lemma} \label{lem:luya_ziwen} Let $\eta$ be a Reeb chord or orbit. Let $\zeta^+,\zeta^-$ and $\beta$ be disjoint braids such that $\zeta^+$ and $\zeta^-$ have the same degree. Then there exists a symplectically immersed cobordism $Z$ from $\zeta^+$ to $\zeta^-$ such that
\[
\#([0,1] \times \beta \cap Z) + \frac{1}{2} \#([0,1] \times \partial\beta \cap \partial_\circ Z) = l_\tau(\zeta^+,\beta) - l_\tau(\zeta^-,\beta)
\]
\end{lemma}

\begin{proof} We prove the result for braids near a Reeb chord, as the Reeb orbit case can be treated identically. It suffices to consider the following setup. Let 
\[Y :=  [0,1]_t \times D^2 \qquad\text{and}\qquad \Lambda := 0 \times I  \cup 1 \times I\]
Here $I := (\R \times 0) \cap D^2$ is the segment of the x-axis on the disk. We consider the braid diagram given by the projection onto the $(x,t)$-plane. 
\[
\pi:Y \subset \R^3 \to \R^2
\]

Recall that any two braids of the same degree can be related by a series of braid isotopies and the following crossing changes (and the reverse moves).
\begin{center}
\includegraphics[width=\textwidth]{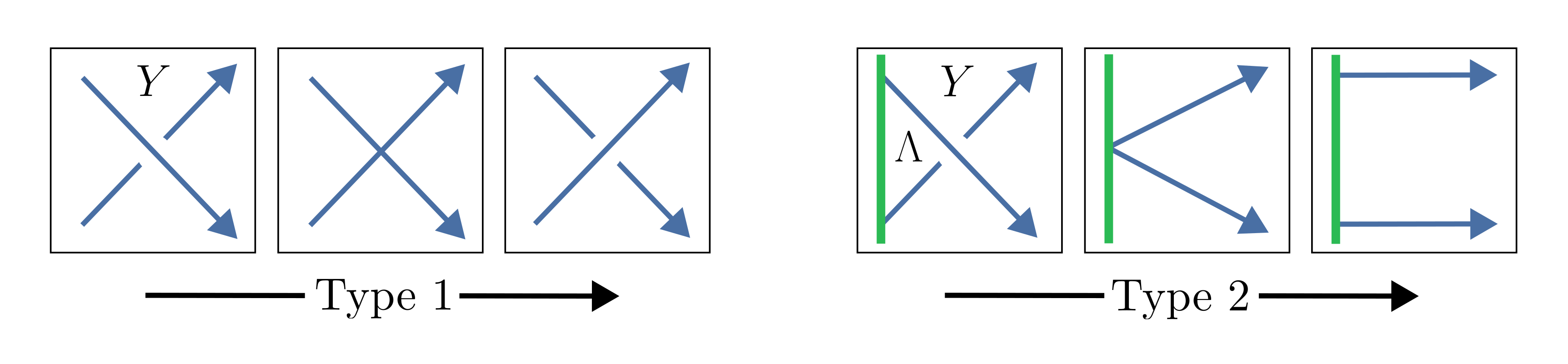}
\label{fig:type_1_2}
\end{center}
We call these Type 1 and Type 2 moves. A sequence of braid isotopies, Type 1 and Type 2 moves can be viewed as a regular homotopy and thus lifted to an immersed cobordism
\[\iota:\Sigma \to [0,1]_s \times Y\]
from the initial braid at $s = 0$ to the final braid at $s = 1$. This immersed cobordism is symplectic as in Lemma 2.4 of Golla-Etnyre \cite{hats}. A Type 1 move yields a transverse interior self-intersection of $\Sigma$ with sign $-1$ and a Type 2 move yields a transverse boundary self-intersection of $\Sigma$ with sign $-1$. The reverse moves yield sign $+1$. 

\vspace{3pt}

Now choose a sequence of isotopies and Type 1/2 moves from  $\zeta^+ \cup \beta$ to $\zeta^- \cup \beta$. We may assume that these moves leave $\beta$ fixed, so that the corresponding cobordism $\Sigma$ decomposes as
\[
\Sigma = Z \cup [0,1] \times \beta
\]
Here $Z$ is an immersed cobordism from $\zeta^+$ to $\zeta^-$ induced by a regular homotopy of braids $\zeta_t$ and intersections between $Z$ and $[0,1] \times \beta$ correspond to Type 1 and 2 moves between the braids $\zeta_t$ and $\beta$. A Type 1 move adds $-1$ intersections and changes the linking number by $-1$. A Type 2 move adds $-1/2$ intersections and changes the linking number by $-1/2$. This yields the result. 
\end{proof}

\begin{proof} (Proposition \ref{prop:properties_of_relative_intersection}) We demonstrate each property individually. Note that the
argument for the difference property suffices to prove well-definedness, since the latter follows from the same argument by taking $A = A'$. 

\vspace{3pt}

{\bf Trivialization.} This follows immediately from the corresponding transformation law for the linking number.

\vspace{3pt}

{\bf Well-Definedness And Difference.} Let $T$ be an admissible representative of a class $B$ in $S(\Psi_+,\Psi_-)$ and let
\[\iota:S \to X \qquad \text{and}\qquad \iota':S' \to X\]
be two admissible representatives of classes $A$ and $A'$ in $S(\Xi_+,\Xi_-)$, respectively, that satisfy the transversality conditions in Definition \ref{def:relative_intersection} with respect to $T$. We construct an immersion $\phi:(\Sigma,\partial\Sigma) \to (X,L)$ transverse to $T$ such that
\begin{equation} \label{eqn:Q_well_defined_property_a} 
[\Sigma] = A - A' \in H_2(X,L) \quad\text{and}\quad \#(\Sigma \cap T) + \frac{1}{2} \#(\partial \Sigma \cap \partial T) = Q_\tau(S,T) - Q_\tau(S',T)
\end{equation}
The result will then  follow since $q_B([\Sigma])$ is precisely given by $\#(\Sigma \cap T) + \frac{1}{2} \#(\partial \Sigma \cap \partial T)$.

\vspace{3pt}

To begin the construction, fix the following notation for a collar near $\partial_\pm X$. 
\[
V_\pm := \partial_\pm X \times [0,\mp \epsilon]\subset X \quad\text{and}\quad U = \overline{X \setminus (V_+ \cup V_-)} \text{ with }\partial_\pm U := \partial_\pm X \times \{\mp\epsilon\}
\]
Here $\epsilon$ is an arbitrarily small parameter. Note that $\iota^{-1}(U) \subset S$ and $[\iota']^{-1}(U) \subset S'$ are equal to $S$ and $S'$ minus collars near $\partial_\pm S$ and $\partial_\pm S'$, respectively. The underlying surface $\Sigma$ is of the form
\[
\Sigma = \iota^{-1}(U) \cup Z_+ \cup Z_- \cup -[\iota']^{-1}(U)
\]
Here $Z_\pm$ is a surface with boundary and corners that we will specify shortly.

\vspace{3pt}

To define the immersion $\phi$ (and in the process, $Z_\pm$), we proceed as follows. First, we let
\begin{equation} \label{eqn:well_defined_Q_proof_1} 
\phi = \iota \text{ on }\iota^{-1}(U) \qquad\text{and}\qquad \phi = \iota' \text{ on }[\iota']^{-1}(U)
\end{equation} 
Note that the image of $\phi$ along $\partial_\pm U$ consists of the braids at the positive and negative ends of $S$ and $S'$, and the remaining boundary lies in $L$. Next, we assert that $\phi(Z_\pm)$ is contained in a union of disjoint tubular neighborhoods of the orbits and chords in $\Xi_\pm$. 
\begin{equation} \label{eqn:well_defined_Q_proof_2}
    \phi(Z_\pm)\subset \bigsqcup_i \op{Nbhd}(\eta^\pm_i) \times [0,\mp\epsilon] \subset V_\pm
\end{equation}
To describe $\phi$ in each neighborhood, fix an chord or orbit $\eta = \eta^\pm_i$ in $\Xi_\pm$ and a tubular neighborhood $N = \op{Nbhd}(\eta^\pm_i)$ of $\eta$ in $Y_\pm$. Denote the braids of $S, S'$ and $T$ around $\eta^\pm_i$ by $\zeta, \zeta'$ and $\beta$ respectively. Note that the braid $\beta$ will be empty if $\eta$ is not in the orbit set $\Psi$. For definiteness let's focus on what happens on $V_+$. By Lemma \ref{lem:luya_ziwen}, we may choose a cobordism $Z$ in $N \times [0,-\epsilon)$ from $\zeta$ to $\zeta'$ with
\[\#(\op{int}(Z) \cap \beta \times [0,-\epsilon]) + \frac{1}{2}\cdot \#(\partial_\circ Z \cap \beta \times [0,\epsilon]) = l_\tau(\zeta,\beta) - l_\tau(\zeta',\beta)\]
The definition for $V_-$ is similar.
We thus may associate a surface $Z^\pm_i$ to the end chord or orbit $\eta^\pm_i$, namely
\[
Z^\pm_i = Z \cup \zeta'_\pm \times [0,\epsilon]
\]
We may view $Z^\pm_i$ as a smooth surface with boundary and corners that is topologically embedded into $V_\pm$ and meets $\iota^{-1}(U)$ and $[\iota']^{-1}(U)$ along $\partial_\pm U$. We then let
\begin{equation} \label{eqn:well_defined_Q_proof_3}
Z_\pm := \bigsqcup_i Z^\pm_i
\end{equation}
and we let $\phi|_{Z_\pm}$ be given by a smoothing of the tautological map $Z_\pm \to V_\pm$. Note that since the asymptotic braids $\beta$ and $\zeta'$ are disjoint, we have
\begin{equation}\label{eqn:well_defined_Q_proof_4}
\#((Z_+ \cup Z_-) \cap T) = l_\tau(\zeta,\beta) - l_\tau(\zeta',\beta) 
\end{equation}

The map $\phi:\Sigma \to \Z$ is smoothly immersed and satisfies the criteria in (\ref{eqn:Q_well_defined_property_a}) by construction. This concludes the proof of well-definedness and the difference property.

\vspace{3pt}

{\bf Union.} We can choose well-immersed representatives $S,S'$ and $T$ of $A,A'$ and $B$ respectively so that $S,S'$ and $S \cup S'$ are transverse to $T$ away from $\partial_\pm X$, disjoint from $T$ near (but not along) $\partial_\pm X$ and transverse to $\partial_\pm X$. Then
\[Q_\tau(A \cup A',B) = \#(\op{int}(S \cup S') \cap \op{int}(T)) + \frac{1}{2} \cdot \#(\partial_\circ (S \cup S') \cap \partial_\circ T) - l_\tau(S \cup S',T) \]
\[
= Q_\tau(A,B) + Q_\tau(A',B)
\]

\vspace{3pt}

{\bf Composition.} Choose representatives $S,S',T$ and $T'$ of $A,A',B$ and $B'$, respectively. Via Lemma \ref{lem:luya_ziwen}, we may assume that the braid of $S$ at the positive end agrees with the braid of $S'$ at the negative end, and likewise for $T$ and $T'$. Then
\[
S \cup T \qquad\text{and}\qquad S' \cup T' 
\]
are representatives in $(X \circ X',L \circ L')$ of $A \circ B$ and $A' \circ B'$, respectively. We can smooth both to surfaces $U$ and $V$ so that near $\partial_-X = \partial_+X'$, they agree with the negative braid of $S \cup S'$ (or equivalently, the positive braid of $T \cup T'$). Then we have
\[
\#(\op{int}(U) \cap \op{int}(V) = \#(\op{int}(S) \cap \op{int}(S')) + \#(\op{int}(T) \cap \op{int}(T'))
\]
\[
\#(\partial_\circ(U) \cap \partial_\circ(V) = \#(\partial_\circ(S) \cap \partial_\circ(S')) + \#(\partial_\circ(T) \cap \partial_\circ(T'))
\]
\[
l_\tau(U,V) = l_\tau(S,S') + l_\tau(T,T')
\]
Note that the last formula above involves the cancelation of the linking of the negative ends of $S$ and $S'$ with the linking at the positive ends of $T$ and $T'$. This proves the result.
 \end{proof}

\subsection{Topological Adjunction} \label{subsection:adjunction} We are now ready to prove a topological version of Legendrian adjunction. The holomorphic curve version will be proven in Section \ref{sec:J_holomorphic_currents} after the appropriate discussion of background.

\begin{theorem}[Topological Adjunction] \label{thm:topological_adjunction} Let $S:\Xi \to \Theta$ be a symplectic, well-immersed surface in a symplectic cobordism $(X,L)$ with a trivialization $\tau$ along $\partial S$. Then $S$ satisfies the adjunction formula
\[ \mu(S,\tau) = 2(\bar{\chi}(A) + Q_\tau(S) + w_\tau(S) - 2\delta(S) - \epsilon(S)) \]
If $S$ is simply smooth (and not necessarily symplectic) then the adjunction formula holds mod 2. \end{theorem}

\begin{proof} By Lemma \ref{lem:maslov_splitting}, we simply need to show that
\begin{equation} \label{eqn:normal_bundle_adjunction}
\mu_\tau(\nu S,\nu(\partial_\circ S)) = 2(Q_\tau(S,S) + w_\tau(S) - 2\delta(S) - \epsilon(S))
\end{equation}
where $\nu(\partial_\circ S)$ is the normal bundle of $\partial_\circ S \subset L$ as before. To prove (\ref{eqn:normal_bundle_adjunction}), choose a section
\[\phi:(S,\partial S) \to (\nu S,\nu(\partial_\circ S)) \qquad\text{satisfying}\qquad 
\phi|_{\partial_\pm S} \neq 0 \qquad \phi \pitchfork \partial_\circ \Sigma \quad\text{and}\quad \phi \pitchfork \op{int}(\Sigma)
\]
Also assume that, near the Reeb chords and orbits in $\partial S$, the section $\phi$ is constant with respect to our chosen trivialization $\tau$.
By Lemma \ref{lem:Maslov_zero_count}, the Maslov index is given by the following count of intersections
\begin{equation} \label{eqn:adjunction_proof_1}
\mu_\tau(\nu S,\nu(\partial_\circ S)) = 2 \cdot \#(\phi \cap \op{int}(S)) + \#(\phi \cap \partial_\circ S).
\end{equation}
On the other hand, let $S'$ be the perturbation of $S$ by the section $\phi$ in a neighborhood of $S$. In both $\op{int}(S)$ and $\partial_\circ S$, there is a single transverse intersection in $S \cap S'$ for each $0$ of $\phi$ and two intersections in $S \cap S'$ each transverse double point. Thus
\begin{equation} \label{eqn:adjunction_proof_2}
\#(\phi \cap \op{int}(S)) = \#(\op{int}(S') \cap \op{int}(S)) - 2\delta(S)\end{equation}
\begin{equation} \label{eqn:adjunction_proof_3}
\#(\phi \cap \partial_\circ S) = \#(\partial_\circ S' \cap \partial_\circ S) - 2 \epsilon(S)
\end{equation}
Finally, we note that by Definition \ref{def:relative_intersection}, we know that
\begin{equation} \label{eqn:adjunction_proof_4}
Q_\tau(S,S) + w_\tau(S) = Q_\tau(S,S) + l_\tau(S,S') = \#(\phi \cap \op{int}(S)) + \frac{1}{2} \cdot \#(\partial_\circ S' \cap \partial_\circ S)
\end{equation}
Combining (\ref{eqn:adjunction_proof_1})-(\ref{eqn:adjunction_proof_4}) proves (\ref{eqn:normal_bundle_adjunction}), and thus the proposition.

\vspace{3pt}

If $S$ is simply a smooth, well-immersed surface then we still have a decomposition
of oriented, real bundle pairs
\[(TX,TL) = (T S,T(\partial_\circ S)) \oplus (\nu S, \nu(\partial_\circ S))\]
After isotopy of $S$ near the ends, we may assume that this splitting respects the symplectic structure near $\partial S$. Modifying the symplectic stricture on $TX|_S$ away from $\partial S$ is equivalent to direct summing with a symplectic bundle on $S^2$, and thus changes $\mu(TX,TL)$ by an even integer. Therefore
\[\mu_\tau(TX|_S,TL|_{\partial S}) = \mu_\tau(TS,T(\partial_\circ S)) + \mu(\nu S,\nu(\partial_\circ S))\]
The proof now reduces to (\ref{eqn:normal_bundle_adjunction}), and proceeds by essentially the same argument.\end{proof}

\section{$J$-Holomorphic Curves with Boundary} \label{sec:J_holomorphic_currents} In this section, we examine holomorphic curves with Legendrian boundary conditions in convex sutured contact manifolds. 

\vspace{3pt}

\subsection{Sutured Contact Manifolds} \label{subsec:sutrued_contact_manifolds} We stary by reviewing contact manifolds with sutured boundary, and the appropriate classes of contact forms and complex structures. 

\begin{remark} Our setting is most similar to \cite{CGHH}. We will only state our definitions for $3$- manifolds, but we note many of our definitions have higher dimensional versions, as found in \cite{CGHH}. \end{remark}

\begin{definition} A \emph{sutured 3-manifold} $Y$ is a compact 3-manifold $Y$ with boundary and corners, a closed sub-manifold $\Gamma \subset \partial Y$ of codimension $2$ and a neighborhood of $\Gamma$ of the form
\begin{equation} \label{eqn:sutured_boundary_neighborhood} 
V(\Gamma) \simeq [-1,1]_t \times (-\epsilon,0]_\tau \times \Gamma \qquad\text{with}\qquad \Gamma \simeq 0 \times 0 \times \Gamma
\end{equation}
The boundary $\partial Y$ must divide into smooth strata $\partial _-Y, \partial_\sigma Y$ and $\partial_+Y$ such that, in the chart (\ref{eqn:sutured_boundary_neighborhood})
\[
\partial_\sigma Y \simeq [-1,1] \times 0 \times \Gamma \qquad\text{and}\qquad \partial_\pm Y \cap V(\Gamma) \simeq \pm 1 \times (-\epsilon,0] \times \Gamma
\]
\end{definition}

\begin{definition} \cite[\S 2]{CGHH} A contact form $\alpha$ on a sutured manifold $Y$ is \emph{adapted} to $Y$ if
\begin{itemize}
    \item[(a)] $\alpha|_{\partial_\pm Y}$ is a Liouville form on $\partial_\pm Y$. 
    \item[(b)] In the neighborhood $V(\Gamma)$ in (\ref{eqn:sutured_boundary_neighborhood}), $\alpha$ is given by
    \[
    \alpha = C \cdot dt + e^\tau \cdot \beta
    \]
    where $C>0$ is a constant and $\beta$ is a one-form on $\Gamma$ independent of $t$ and has no $dt$-term.
    Consequently in this neighborhood the Reeb vector field is given by $\frac{1}{C}\partial_t$.
    
\end{itemize}
\end{definition}

\begin{remark} When $Y$ is equipped with an adapted contact form, we may extend $t$ to a function
\[t:\op{Nbhd}(\partial Y) \to \R\]
on a neighborhood of $\partial Y$ by setting $t(\partial_\pm Y) = \pm 1$ and then extending $t$ to be Reeb invariant near $\partial_\pm Y$. We will regularly use this extension without further comment. \end{remark}

\begin{definition} A \emph{(convex) sutured contact manifold} $(Y,\xi)$ is a sutured manifold $Y$ and a contact structure $\xi$ on $Y$ that is the kernel of an adapted contact form.
\end{definition}

\begin{definition}\label{def:exact legendrian}
A collection of Legendrians $\Lambda$ in the horizontal boundary of the convex sutured convex contact manifold is called \emph{exact} if the Liouville form $\beta_\pm$ vanishes on $\Lambda$.
\end{definition}

Consequently on the exact Legendrians the contact distribution $\xi$ is tangent to the horizontal boundary.

\begin{remark}
We give some examples where one can find exact Legendrians. It suffices to find a two dimensional Liouville manifold $(\Sigma,\beta)$ where $\beta$ is the Liouville form, and a collection of Lagrangians $\{L_i\}$ on which $\beta$ vanishes. The simplest possible example is $(D^*S^1,\beta)$, the codisk bundle of $S^1$; and $\beta$ is the canonical one form. The Lagrangian in question is then the zero section.

More generally speaking,  let $\{L_i\}$ denote a collection of circles. Consider the Liouville manifolds $(D^*L_i, \beta_i)$ as above. We attach one handles between this collection of Liouville manifolds to form the Livioulle manifold $(S,\beta)$. Then the collection of curves $\{L_i\}$ can be taken to be the required collection of Lagrangians. 
\end{remark}

\begin{definition} \cite[\S 3.1]{CGHH} A complex structure $J$ on $\xi$ that is \emph{tailored} to $(Y,\Lambda)$ if
\begin{itemize}
    \item[(a)] $J$ is Reeb invariant near $\partial Y$.
    \item[(b)] Consider the completion of $\partial_\pm Y$ near $V(\Gamma)$ given by $\partial_\pm Y \cup [0,\infty) \times \Gamma$, with Liouville form $e^\tau \beta$ in a neighborhood of the form  $(-\epsilon,\infty)\times \{\pm1\} \times \Gamma$. We require that in both $V(\Gamma)$ and small tubular neighborhoods $\partial_+Y  \times [1,1-\epsilon]$ and $\partial_-Y \times [-1,-1+\epsilon]$, the almost complex structure $J$ is the pullback via the natural projection $\pi:\xi \to T(\partial_\pm Y)$ of a complex structure $J_0$ that is $\beta$ compatible on $(-\epsilon,\infty)\times \Gamma$, and compatible with $\alpha|_{\partial_\pm Y}$ on $\partial_\pm Y$.
\end{itemize}

We let $\hat{J}$ denote the complex structure induced on $\R \times Y = \R_s \times Y$, given by
\[\hat{J}(\partial_s) = R \qquad \hat{J}|_\xi \, \text{is a tailored complex structure} \,J.\]
\end{definition}
\begin{remark}
Given a sutured contact manifold, it is also helpful to think about its completion as in Section 2.4 in \cite{CGHH}.
First, ``vertically'' complete $V(\Gamma)$ by gluing $[1,\infty) \times \partial_+ Y$ and $(-\infty, -1] \times \partial_- Y$ with the forms $C dt + \alpha|_{\partial_+Y}$ and $C dt + \alpha|_{\partial_- Y}$ respectively. Now the boundary is $\{0\} \times \R \times \Gamma$. Second, ``horizontally'' complete by gluing $[0, \infty) \times \R \times \Gamma$ with the form $C dt + e^\tau \beta$. We denote the completion as $(M^*, \alpha^*)$.
\end{remark}

\subsection{Holomorphic Maps} We now recall the basic definitions regarding $J$-holomorphic maps in SFT, as required in this paper.

\vspace{3pt}

Let $(\Sigma,j)$ be a Riemann surface with boundary and corners, and assume that $\Sigma$ is equipped with a boundary decomposition into smooth components
\[\partial\Sigma = \partial_+\Sigma \cup \partial_\circ\Sigma \cup \partial_- \Sigma\]
We think of of $\partial_\circ \Sigma$ as boundaries of the surface. For $\partial_\pm \Sigma$, we think of $\Sigma$ as surface with boundary/interior punctures, and each near boundary puncture we equip the puncture with a semi-infinite strip-like (resp. cylindrical) neighborhood $[0,\pm \infty] \times [0,1]$ (resp. $[0,\infty] \times S^1$, and we think of $\partial_\pm \Sigma$ as the components of the boundary at infinity, of the form $\pm \infty \times [0,1]$ (resp. $\pm \infty \times S^1$).

Let $(Y,\xi)$ be a convex sutured contact manifold with closed Legendrians $\Lambda\subset \partial Y_\pm$ and fix an adapted contact form $\alpha$ and tailored complex structure $J$. A \emph{$J$-holomorphic map}
\[u:(\Sigma,\partial_\circ \Sigma) \to (\R \times Y,\R \times \Lambda)\]
is a smooth map $u$ from $\Sigma \setminus (\partial_+ \Sigma \cup \partial_- \Sigma)$ to the symplectization $\R \times Y = \R \times Y$ that maps  $\partial_\circ\Sigma$ to $\R \times \Lambda = \R \times \Lambda$, and that satisfies the non-linear Cauchy-Riemann equations
\begin{equation} \label{eqn:Cauchy_Riemann}
du \circ j = \hat{J} \circ du.
\end{equation}
A connected component $C$ of $\partial_+\Sigma \cup \partial_-\Sigma$ is called a \emph{puncture} of $u$. A puncture is \emph{positive} if it is in $\partial_+\Sigma$ and \emph{negative} if it is in $\partial_-\Sigma$. Likewise, a closed component is an \emph{interior} puncture and a component with boundary is a \emph{boundary} puncture. We require positive punctures under the map $u$ are at $+\infty$ of the symplectization direction, and negative punctures map to $-\infty$ in the symplectization direction.

\vspace{3pt}

The \emph{energy} of a $J$-holomorphic curve $u$ is given by
\[
E(u) = \sup_\phi \Big(\int_{\Sigma} u^*d(\phi(s) \alpha)\Big) \in [0,\infty)
\]
Here the supremum is over all non-decreasing smooth maps $\phi:\R \to [0,1]$. 

\vspace{3pt}

\subsection{Local Maximum Principles} \label{subsec:local_maximum_principles} Tailored complex structures satisfy two harmonicity results (which may be viewed as maximum principles) that will be used throughout this paper.

\begin{lemma} \cite[Lem. 5.6]{CGHH} \label{lem:t_maximum_principle}  Let $u:\Sigma \to \R \times Y$ be a $J$-holomorphic map (possibly with boundary) with respect to a tailored $J$. Then
\[
t \circ u:\Sigma \to \R \qquad\text{defined on}\qquad u^{-1}(\op{Nbhd}(\partial_+Y \cup \partial_-Y))
\]
is harmonic in a neighborhood of $\partial_+Y$ and $\partial_-Y$.
\end{lemma}
\begin{proof}
We follow the proof of Lemma 5.6 in \cite{CGHH}. Without loss of generality, we focus on a tubular neighborhood of $\partial_+ Y$ of the form $\partial_+ Y \times [1+\epsilon,1-\epsilon)$ (we imagine extending the $t$ coordinate slightly in the upwards direction). The key observation is that since $\partial_+ Y$ is two dimensional, $J_0$ is Stein, which means there exists $f: \partial_+Y \rightarrow \mathbb{R}$ with the boundary of $\partial_+ Y$ its level set. Further, the 1-form $\beta' : = -df \circ J_0$ gives $\partial_+Y$ the structure of a Liouville manifold, and the symplectic form $d\beta'$ is $J_0$ compatible. If we take the contact form $\alpha' =dt + \beta'$ on $(1-\epsilon,+\infty)_t \times \partial_+ Y$ (which we think of the upwards completion of $Y$, see \cite{CGHH}), then the same proof as Lemma 5.6 in \cite{CGHH} tells us that $t\circ u$ is harmonic in this region.
\end{proof}

\begin{lemma} \cite{CGHH} \label{lem:tau_maximum_principle}  Let $J$ be a tailored almost complex structure as above. Restricted to the Liouville manifold
\[(-\epsilon,0]_\tau \times \Gamma \quad\text{with Liouville form}\quad e^\tau \alpha|_\Gamma,\] the almost complex structure $J$ is compatible with $\alpha|_\Gamma$.
Then the $\tau$ coordinate $\tau:(-\epsilon,0]_\tau \times \Gamma \to (-\epsilon,0]$ is pluri-subharmonic. 
\end{lemma}
\begin{proof}
See lemma 5.5 in \cite{CGHH}.
\end{proof}

\subsection{Local Properties Of Boundary Singularities} We will also need a number of local results governing the singularities of holomorphic curves with boundary. To state these results, we fix the following notation. Let $\mathbb{U}$ denote the upper half-disk
\[
\mathbb{U} := \mathbb{D} \cap \mathbb{H} \qquad\text{with}\qquad \partial_\circ \mathbb{U} := \mathbb{U} \cap \R
\]
We adopt coordinates $z = s + it$ on $\mathbb{U} \subset \mathbb{C}$. We also denote the upper half-ball in $\C^n$ as follows.
\[U^{2n} := B^{2n} \cap (\mathbb{H} \times \C^{n-1}) \qquad \text{with}\qquad \partial_\circ U^{2n} = U^{2n} \cap \R \times \C^{n-1}\]
Finally, we let $R^n := B^{2n} \cap \R^n$ denote the Lagrangian in $\partial_\circ U^{2n}$ given by the real unit disk. 

\vspace{3pt}

\begin{lemma} \label{lem:nice_coordinates_for_boundary_map} Let $J$ be an almost complex structure on $U^4$ such that $J(TR^n) \cap TR^n = 0$ and consider a $J$-holomorpic map
\[
u:(\mathbb{U},\partial_\circ \mathbb{U}) \to (U^{2n},R^n) \qquad\text{with}\qquad u(0) = 0 \quad\text{and}\quad du(0) \neq 0 \
\]
Then there is an open neighborhood $\Omega \subset U^{2n}$ of $0$ and a local diffeomorphism $\phi:\Omega \to U^{2n}$ such that
\[
\phi \circ u(z) = (z,0,\dots,0) \qquad\text{and}\qquad J = \phi^*J_0 \quad \text{on}\quad \mathbb{U} \times 0 \times \dots \times 0.
\]
Here $J_0$ is the standard almost complex structure on $\mathbb{C}$.
\end{lemma}

\begin{proof} Choose a complex bundle isomorphism
\[\tau:(\C^n,J_0) \simeq u^*(\C^n,J) \qquad\text{with}\qquad \tau(\R^n) = \R^n \text{ along }\partial_\circ \mathbb{U}.\]
Also assume that $\tau$ satisfies the following constraint
\[\tau_{u(z)}(1,0,\dots,0) = \frac{\partial u}{\partial s} \qquad\text{in coordinates $z = s + it$ on $\mathbb{U}$}.\]
Then we may define $\phi$ in terms of the exponential map as follows.
\[\phi^{-1}(z,z_2,\dots,z_n) := \exp_{u(z)}(\tau_{u(z)}(0,z_2,\dots,z_n))\] 
Verifying the claimed properties of $\phi$ is standard (see \cite[Lem. 2.4.2]{big_ms}).
\end{proof}

Next we prove that boundary intersections of (locally) distinct curves are isolated. The analogue for interior intersections is proven in \cite[Lem. 2.4.3]{big_ms} and the proof is directly analogous.

\begin{lemma} \label{lem:intersection_accumulation} Fix an almost complex structure $J$ on $U^{2n}$ with $J(TR^n) \cap TR^n = 0$ and consider a pair of $J$-holomorphic maps
\[u,v:(\mathbb{U},\partial_\circ\mathbb{U}) \to (U^{2n},R^n) \qquad\text{with}\qquad u(0) = v(0) \quad\text{and}\quad du(0) \neq 0\]
Suppose that there are  sequences $\{z_k\}$ and $\{w_k\} \in \mathbb{U}$ such that
\[u(z_k) = v(w_k) \qquad \lim_{k \to \infty} z_k = \lim_{k \to \infty} w_k = 0 \quad\text{and}\quad z_k \neq 0\]
Then there exists a neighborhood $\Omega \subset \mathbb{U}$ of $0$ and a map $\phi:\Omega \to \mathbb{U}$ such that
\[\phi(0) = 0 \qquad\text{and}\qquad v = u \circ \phi\]
\end{lemma}

\begin{proof} By Lemma \ref{lem:nice_coordinates_for_boundary_map}, we can assume that $u(z) = (z,0,\dots,0)$ and that $J = J_0$ along $u$. Write
\[
v(z) = (v_1(z),\tilde{v}(z)) \qquad\text{for maps}\qquad v_1:\mathbb{U} \to \mathbb{U} \quad\text{and}\quad \tilde{v}:\mathbb{U} \to \C^{n-1}
\]

We claim that $\tilde{v}$ vanishes to infinite order at $z = 0$. Indeed, suppose that $\tilde{v}$ is order $l \ge 1$, i.e. that $\tilde{v}(z) = O(|z|^l)$ and $\tilde{v}(z) \neq O(|z|^{l+1})$. Since $J = J_0$ along $u$, we know that
\[
J(v(z)) = J_0 + O(|z|^l)
\]
and therefore that the order $l$ Taylor series of $v$ is holomorphic, i.e. that
\[
v_1(z) = p(z) + O(|z|^{l+1}) \qquad \text{and}\qquad \tilde{v}(z) = \tilde{a} z^l + O(|z|^{l+1})
\]
for a polynomial $p$ of order $l$ and a non-zero constant $\tilde{a}$. In particular, $\tilde{v}(z) \neq 0$ in some neighborhood of $0$ and the intersections of $v$ and $u$ are isolated away from $0$, contradicting the hypotheses. Finally, note that since $J = J_0$ along $u$, we have (for all $w = (w_1,\tilde{w})$)
\[
\frac{\partial J(w_1,0)}{\partial x_1} = \frac{\partial J(w_1,0)}{\partial y_1} = 0 \quad\text{and thus}\quad |\frac{\partial J(w)}{\partial x_1}| + |\frac{\partial J(w)}{\partial y_1}| \le C|\tilde{w}|
\]
Since $v$ is $J$-holomorphic, this implies that
\begin{equation} \label{Equation:Aronszajn}
|\Delta \tilde{v}| \le C(|\tilde{v}| + |\partial_s\tilde{v}| + |\partial_t\tilde{v}|)
\end{equation}
and we can apply Aronszajn's theorem (cf. \cite[Thm 2.3.4]{big_ms}). In particular we extend $\tilde{v}$ smoothly past the boundary, then Equation \ref{Equation:Aronszajn} continues to hold in a small neighborhood of the origin, and the origin is still a zero of infinite order. Then we can apply Aronszajn's theorem in its original form to conclude that $\tilde{v} = 0$ identically. Since $\C \times 0 \times \dots \times 0 \cap U^{2n}$ is precisely the (embedded) image of $u$, this proves the result. See also the results in \cite{OhKwon}.
\end{proof}

\subsection{Singularities} We now apply the local maximum principles of \S \ref{subsec:local_maximum_principles} to study the singularities of holomorphic curves with boundary in the symplectization of $Y$. 

\vspace{3pt}

For the rest of this sub-section, fix a convex sutured contact 3-manifold $(Y,\xi)$ and a closed (possibly disconnected) Legendrian $\Lambda \subset \partial_+Y \cup \partial_-Y$. Also fix a non-degenerate, adapted contact form $\alpha$ and a tailored complex structure $J$.

\begin{lemma}[Boundary Immersion] \label{lem:boundary_immersion} Let $u:(\Sigma,\partial_\circ\Sigma) \to (\R \times Y,\R \times \Lambda)$ be a finite energy $\hat{J}$-holomorphic map with boundary. Then
\[u^{-1}(\R \times \partial Y) = \partial_\circ \Sigma \qquad\text{and}\qquad u \text{ is transverse to }\R \times \partial Y\]
\end{lemma}

\begin{proof} Since $u$ has finite energy, it is asymptotic to $\Xi_\pm$ at $\pm \infty$, in the sense that $u$ exponentially approaches the trivial cylinders and strips over $\Xi_\pm$ at $\pm \infty$. Recall by our notation $\Xi_\pm$ denotes a collection of orbits and chords at $\pm \infty$ in the symplectization direction. The specific manner in which they approach the orbits and chords are detailed in \cite{siefring2008relative} and \cite{abbas1999finite}. Since the Reeb chords and orbits are transverse to $\partial_\pm Y$ and disjoint from $\partial_\sigma Y$ (we  assume implicitly the Legendrians are in the interior of $\partial_\pm Y$) , our theorem holds on
\[
u^{-1}((-\infty,-T] \cup [T,\infty) \times Y) \subset \Sigma \qquad \text{for any sufficiently large} \qquad T > 0 
\]
To show that the result holds elsewhere, consider
\[
A := u^{-1}([-T,T] \times \op{Nbhd}(\partial_\sigma Y)) \subset \Sigma .
\]
For sufficiently large $T$ and sufficiently small $\op{Nbhd}(\partial_\sigma Y)$, this is a smooth surface, whose boundary is disjoint from $\R \times \partial_\sigma Y$. Thus by Lemma \ref{lem:tau_maximum_principle}, $u|_A$ (and thus $u$) is disjoint from $\partial_\sigma Y$. Finally, consider the compact surface with boundary
\[
B := u^{-1}([-T,T] \times [-1, -1+\epsilon] \times \partial_-Y) \subset \Sigma
\]
The restriction $t \circ u|_{B}$ is harmonic by Lemma \ref{lem:t_maximum_principle}. Thus, the minimum $-1$ is achieved only on the boundary of $B$, which is the inverse image of
\[\{-T,+T\}_s \times [-1,-1+\epsilon]_t \times \partial_\circ Y \cup [-T,T]_s \times \{-1,-1+\epsilon\}_t \times \partial_\circ Y \subset [-T,T] \times Y\]
For large $T$ the minimum cannot be achieved on
\[
\{-T,+T\} \times (0,\epsilon] \times \partial_\circ Y \qquad\text{or}\qquad [-T,T] \times \epsilon \times \partial_\circ Y
\]
Thus $t \circ u$ only takes the value of $-1$ on $\partial \Sigma \cap B = u^{-1}([-T,T] \times \{-1\} \times \partial_\circ Y)$. A similar argument holds near $\partial_+Y$, and this proves that
\[
u^{-1}(\R \times \partial Y) = \partial\Sigma 
\]

Finally, we argue that $u$ is an immersion near $\R \times \partial Y$. This is evident away from $[-T,T] \times Y$ by the normal forms of the curve near punctures, as in \cite{siefring2008relative} and \cite{abbas1999finite}. Thus it suffices to show that
\[
t \circ u:u^{-1}([-T,T] \times \op{Nbhd}(\partial_- Y)) \to \R 
\]
has no critical points along $\partial_\circ \Sigma$ (and similarly for $\partial_+Y$). Thus, pick any $p \in \partial_\circ \Sigma$ with $u(p) \in \partial_-Y$. In a neighborhood of $p$, since $t \circ u$ is harmonic, it is modelled on the real part of a map holomorphic map with a $0$ of order $k$, i.e.
\[
z \mapsto z^k \qquad \text{for}\qquad k \ge 1
\]
In particular, the number of connected components of the set
\[
(t \circ u)^{-1}((0,\infty)) \cap \op{Nbhd}(p)
\]
is equal to the order $k$ for small enough $\op{Nbhd}(p)$. Since $t \circ u > 0$ in $\op{Nbhd}(\partial_-Y)$, we thus know that $k = 1$. This concludes the proof.
\end{proof}

\begin{lemma}[Boundary Submersion] \label{lem:boundary_submersion} Let $u:(\Sigma,\partial_\circ \Sigma) \to (\R \times Y,\R \times \Lambda)$ be a finite energy $\hat{J}$-holomorphic map with boundary and let $\Gamma$ be a connected component of $\partial_\circ \Sigma$. Then the map
\[\pi_\R \circ u:\Gamma \xrightarrow{u} \R \times \Lambda \xrightarrow{\pi_\R} \R\]
is a proper submersion (and therefore a diffeomorphism). \end{lemma}

\begin{proof} By Lemma \ref{lem:boundary_immersion}, $u$ is transverse to $\R \times \partial Y$. On the otherhand, suppose that
\[
d(\pi_\R \circ u|_\Gamma) = d\pi_\R \circ du|_\Gamma = 0 \qquad\text{at some }p \in \Gamma
\]
Then $du$ maps $T\Gamma$ to $0 \oplus TY$ at $p$, since the kernel of $d\pi_\R$ is precisely $TY$. Since $u(\partial_\circ \Sigma) \subset \R \times \Lambda$, we must then have
\[
du_p(T\Sigma) \subset (0 \oplus T\Lambda) \oplus (0 \oplus J(T\Lambda)) = 0 \oplus \xi 
\]
Moreover, $\xi_p = T(\partial Y)_p$ at $p$ since $\Lambda$ is an exact Legendrian. Thus $u$ is tangent to $\R \times \partial Y$ at $p$. This contradicts Lemma \ref{lem:boundary_immersion}, so this proves that $\pi_R \circ u|_\Gamma$ is a submersion. Since $\pi_\R$ and $u$ are proper, $\pi_\R \circ u|_\Gamma$ is also proper. 
\end{proof}

As a consequence of the boundary immersion property, we can easily prove that there are a finite number of critical points. 

\begin{lemma}[Finite Critical Points] \label{lem:finite_critical_points} Let $u:(\Sigma,\partial_\circ\Sigma) \to (\R \times Y,\R \times \Lambda)$ be a finite energy, non-constant $J$-holomorphic curve. Then the set of critical points $\op{Crit}(u) \subset \Sigma$ is finite.
\end{lemma}

\begin{proof} By Lemma \ref{lem:boundary_immersion}, the set $\op{Crit}(u) \subset \Sigma$ is a closed (and thus compact) subset of the interior of $u^{-1}([-T,T] \times Y)$ for some large $T > 0$. By \cite[\S 2.4, p. 26]{big_ms}, interior critical points of a non-constant holomorphic curve are isolated. Therefore, $\op{Crit}(u)$ is finite.
\end{proof}

We next state a key proposition that will be essential to the proof of $d^2=0$.

\begin{proposition}\label{prop:no closed boundary components}(No closed boundary components)
Let $u:(\Sigma, \partial_\circ\Sigma) \rightarrow (\R \times Y,\R \times \Lambda)$ be a finite energy $J$-holomorphic map, then no component of $\partial_o \Sigma$ is closed unless $u$ is constant.
\end{proposition}
\begin{proof}
Suppose not, let $S^1$ denote a closed component of $\partial_o\Sigma$ with coordinate $q\in S^1$. It is mapped by $u$ to $\R \times L$, where $L$ is an exact Legendrian. Let $\pi_\R$ denote the projection to the symplectization direction, and assume $\pi_\R u|_{S^1}$ achieves its maximum at $p$. Then $\frac{d}{dq} \pi_\R u(p) =0$. Thus at $p$ the tangent space of $u$ coincides with the contact distribution $\xi$, which at $p$ is tangent to the horizontal boundary. But this implies the curve $u$ is not transverse to the boundary of $\R \times Y$ at $p$, which contradicts the local form in Lemma \ref{lem:boundary_immersion}, \ref{lem:nice_coordinates_for_boundary_map}.
\end{proof}

\subsection{Simple And Somewhere Injective Curves} We can now prove an equivalence between the class of simple curves and somewhere injective curves, generalizing the analogous statement in the case of curves without boundary.

\begin{definition} A point $p$ in the domain of a $J$-holomorphic map $u:(\Sigma,j) \to (X,J)$ is \emph{injective} if
\begin{equation} \label{eqn:injective_point} u^{-1}(u(p)) = \{p\} \qquad \text{and}\qquad du_p \neq 0\end{equation}
The $J$-holomorphic map $u$ is called \emph{somewhere injective} if it has an injective point.
\end{definition}

\begin{definition} A $J$-holomorphic map $u:(\Sigma,j) \to (X,J)$ is \emph{simple} if it does not factor as
\[(\Sigma,j) \xrightarrow{\phi} (S,i) \xrightarrow{v} (X,J)\]
where $\phi$ is a branched cover of degree $2$ or more, and $v$ is $J$-holomorphic.
\end{definition}

\begin{lemma}[Factorization] \label{lem:factorization}

Any finite energy, proper $J$-holomorphic map $u:(\Sigma,\partial_\circ \Sigma) \to (\R \times Y,\R \times \Lambda)$ factors as
\[
(\Sigma,j) \xrightarrow{\phi} (S,i) \xrightarrow{v} (X,J)
\]
where $\phi$ is a branched cover and $v$ is a proper $J$-holomorphic map whose injective points form a dense and open set. 
\end{lemma}

\begin{proof} Let $Cr$ denote the set of critical values of $u$ and let $B \subset u(\Sigma) \setminus Cr$ be the set of non-critical values where multiple branches of $u$ meet. That is, $y \in B$ if and only if there are points $p,q$ such that
\[
y = u(p) = u(q) \qquad p \neq q \quad \text{and}\quad u(\op{Nbhd}(p)) \neq u(\op{Nbhd}(q)) 
\]
for any neighborhoods $\op{Nbhd}(p)$ and $\op{Nbhd}(q)$. The set $Cr$ is finite by Lemma \ref{lem:finite_critical_points} and the set $B$ is discrete in $u(\Sigma) \setminus Cr$ by \cite[Lem. 2.4.3]{big_ms} and Lemma \ref{lem:intersection_accumulation}. Therefore, $S' = u(\Sigma) \setminus (B \cup Cr) \subset \R \times Y$ is a Riemann surface with (interior and boundary) punctures, equipped with a tautological map $v:S' \to \R \times Y$. The punctures of $S'$ mapping to $B$ and $C$ are removable, and so we can choose extensions of $S'$ and $v$
\[S' \subset S \qquad\text{and}\qquad \qquad v:S \to \R \times Y\]
so that the resulting map is proper (the extension on the interior works the same was as Proposition 2.5.1 in \cite{big_ms}, the extension near the boundary comes from the boundary immersion property in Lemma \ref{lem:boundary_immersion}). The map $\Sigma \setminus u^{-1}(B \cup Cr) \to S'$ extends across $u^{-1}(B \cup Cr)$ to a map $\phi:\Sigma \to S$ of Riemann surfaces with punctures. The maps $\phi:\Sigma \to S$ and $v:S \to \R \times Y$ are precisely the desired maps as we observe that $v$ is simple by construction, and its injective points form an open and dense set.
\end{proof}
\begin{remark} We emphasize that the results of this lemma \emph{do not} hold for general holomorphic curves with Lagrangian (or totally real) boundary conditions. Indeed, simple curves with Lagrangian boundary need not have finite self-intersections and may not be determined by their image. An example given in Remark 2.5.6 in \cite{big_ms}: a disc mapping in $S^2$ with boundary on the equator can wrap one and half times around the sphere. Here we are able to obtain nice results because we restricted the Lagrangian to lie at the boundary of our manifold, and $t\circ u$ is harmonic as in Lemma \ref{lem:boundary_immersion}. This allows us to conclude that $B$ is \emph{finite} near the boundary (which is not the case for the partially wrapped disk around $S^2$ in the example we just gave), and ensures the factorization Lemma \ref{lem:factorization} holds.
\end{remark}

\begin{lemma}[Simple Is Mostly Injective] \label{lem:simple_vs_injective} Fix a non-constant, simple, finite energy $J$-holomorphic map 
\[
u:(\Sigma,\partial \Sigma) \to (\R \times Y,\R \times \Lambda)
\]
Then the set of injective points is cofinite.
\end{lemma}

\begin{proof} Let $S(u) \subset \Sigma \times \Sigma$ denote the set of pairs $(p,q)$ that satisfy $u(p) = u(q)$ and $p \neq q$. Then the set of non-injective points in $\Sigma$ is precisely
\[
\pi_1(S(u)) \cup \op{Crit}(u) \qquad \text{where} \qquad \pi_1(p,q) = p
\]
$\op{Crit}(u)$ is finite by Lemma \ref{lem:finite_critical_points}, and so it suffices to show that $S(u)$ is finite. Note that by Lemma \ref{lem:boundary_immersion}, we can decompose $S(u)$ as
\[
S(u) = S_\partial(u) \cup S_{\op{int}}(u)
\]
where $S_\partial(u)$ consists of pairs in $\partial\Sigma \times \partial\Sigma$ and $S_{\op{int}}(u)$ consists of pairs $\op{int}(\Sigma) \times \op{int}(\Sigma)$. 

\vspace{3pt} 

Now, note that $u$ is locally injective, i.e. for every $z \in \Sigma$ there exists a neighborhood $U$ of $z$ such that $u|_U$ is injective. If $z \in \partial\Sigma$, then this follows from the fact that $du_z \neq 0$ by Lemma \ref{lem:boundary_immersion}. If $z \in \op{int}(\Sigma)$, see \cite[Rmk. E.1.3]{big_ms}. This implies that $S(u)$ is separated from an open neighborhood of the diagonal.

\vspace{3pt}

Next, note that $S(u)$ is a discrete set. Indeed, Lemma \ref{lem:intersection_accumulation} implies that any pair $(z,w) \in S_{\partial}(u)$ is isolated from other pairs in $S(u)$. Moreover, pairs in $S_{\op{int}}(u)$ are isolated by \cite[Thm. E.1.2, Claim (ii)]{big_ms}. 

\vspace{3pt}

Finally, note that the image of pairs of points in $S(u)$ must be contained in a compact region $[-T,T] \times Y$ because near the punctures of $u$, we have asymptotic formulas (see \cite{siefring2008relative} and the Appendix). Since $u$ is proper, this implies that $S(u)$ is bounded in $\Sigma \times \Sigma$.

Thus, $S(u)$ is a closed, bounded, isolated subset of $(\Sigma \times \Sigma) \setminus \Delta$ and is therefore finite. This proves the result.
\end{proof}

\begin{corollary}  Fix a pair of connected, simple, finite energy $J$-holomorphic maps
\[
u:(\Sigma,\partial \Sigma) \to (\R \times Y,\R \times \Lambda) \qquad\text{and}\qquad v:(\Sigma',\partial \Sigma') \to (\R \times Y,\R \times \Lambda)
\]
Suppose that $u(\Sigma) = v(\Sigma')$. Then there is a biholomorphism $\varphi:\Sigma \simeq \Sigma'$ such that $u = v \circ \varphi$. \end{corollary}

\begin{proof} 
The same proof as Corollary 2.5.4 in \cite{big_ms} follows through.
\end{proof}

\subsection{Counts Of Singularities And Legendrian Adjunction} We can now associate a count of singularities to the $J$-holomorphic maps of interest. We first must explain the local situation, starting with the interior case.

\begin{proposition} (cf. \cite[Lemma 2.6]{w2017}) \label{prop:count_of_interior_singularities} Let $J$ be an almost complex structure on $B^4 \subset \C^2$ and let $\Omega = \D_1 \cup \dots \cup \D_m$ where $\D_i = \D$ is a copy of the disk with origin $z_i$. Consider a $J$-holomorphic map
\[
v:(\Omega,i) \to (B^4,J) \qquad\text{with}\qquad v(z_i) = 0 \quad\text{and}\quad v(\partial\Omega) \subset \partial B^4
\]
Assume that the only non-injective points of $v$ are $z_i$. Then there is a unique positive integer
\[
\delta(v) = \delta(v;0) \in \Z_+
\]
with the following property: there exists a symplectic immersion
\[
\tilde{v}:\Omega \to B^4 \qquad\text{with}\qquad \tilde{v} = v \text{ on }\op{Nbhd}(\partial\Omega)
\]
with precisely $\delta(v)$ positive, transverse double points . \end{proposition}

We will use the following analogue of Proposition \ref{prop:count_of_interior_singularities} for boundary singularities. The proof is significantly simpler than analogous results (e.g. \cite[Lemma 2.6]{w2017}) in the closed case because we may assume that there are no critical points near the boundary.

\begin{proposition} \label{prop:count_of_boundary_singularities} Let $J$ be an almost complex structure on $U^4 = B^4 \cap \mathbb{H} \times \C$ and let $\Omega = \mathbb{U}_1 \cup \dots \cup \mathbb{U}_m$ where $\mathbb{U}_i = \D \cap \mathbb{H}$ is a copy of the upper half-disk with origin $z_i$. Consider a $J$-holomorphic immersion
\[
v:(\Omega,i) \to (U^4,J) \qquad\text{with}\qquad v(z_i) = 0 \quad v(\partial_\circ \U_i) \subset R^2 \quad\text{and}\quad v(\partial\Omega) \subset \partial U^4
\]
Assume that the only non-injective points of $v$ are $z_i$. Then there is a unique positive integer
\[
\epsilon(v) = \epsilon(v;0) \in \Z_+
\]
with the following property: there exists a symplectic immersion
\[
\tilde{v}:\Omega \to U^4 \qquad\text{with}\qquad \tilde{v} = v \text{ on }\op{Nbhd}(\partial \Omega \setminus \partial_\circ \Omega) 
\]
with precisely $\epsilon(v)$ positive, transverse boundary double points.
\end{proposition}
\begin{proof}
We first consider the simple case where $v|_{\partial \mathbb{U}_i}$ and $v|_{\partial \mathbb{U}_j}$ intersect transversely on the Lagrangian boundary $\R \times \R \subset \C \times \C$. Then the intersection viewed as an intersection between two surfaces in $\C \times \C$ is positive, since it is modelled on the intersection of two complex planes in $\C^2$. 

In the next case consider a collection of $v|_{\partial \mathbb{U}_i}, i=1,\cdots k$ intersecting at a single point $p\in \R^2$, but such that pairwise their intersections at $p$ are transverse. Then by perturbing the $v|_{ \mathbb{U}_i}$ slightly we recover the conclusion of the theorem.

The most difficult case is when $v|_{\partial \mathbb{U}_i}$ and $v|_{\partial \mathbb{U}_j}$ are tangent to each other at $p\in \R^2$. We turn to this situation next.

Let $U^4(\epsilon)$ be the half-ball of radius $\epsilon$. Note that the image of $v$ is disjoint from $0 \times \C$ except at $0$. Therefore, by Lemma \ref{lem:intersection_accumulation}, we may acquire a (transverse) braid given by
\[
L = \partial B^4(\epsilon) \cap v(\Omega) \quad \text{with components}\quad L_i = \partial B^4(\epsilon) \cap v(\mathbb{U}_i) 
\]

\vspace{3pt}

To prove this, fix $i \neq j$ and consider the maps $v|_{\mathbb{U}_i}$ and $v|_{\mathbb{U}_j}$. By Lemma \ref{lem:nice_coordinates_for_boundary_map}, we can (after a change of coordinates) assume that
\[
v|_{\mathbb{U}_i}(z) = (z,0) \qquad\text{and}\qquad J|_{\mathbb{U} \times 0} = J_0
\]
In these coordinates, we may write $v$ restricted to $\mathbb{U}_j$ as follows.
\[
v|_{\mathbb{U}_j}(z) = (z,a\cdot z^l) + O(|z|^{l+1})
\]
Write the second factor of $v|_{\mathbb{U}_j}$ as $w(z) = a \cdot z^l + O(|z|^{l+1})$. After an orientation check, we find that the linking number of $v|_{\mathbb{U}_i}$ and $v|_{\mathbb{U}_j}$ is now precisely the (half-integer) winding number of the path
\[
[0,1] \to \C \setminus 0 \qquad\text{given by}\qquad \theta \mapsto w(e^{\pi i \theta})
\]
Since $w(z) = a \cdot z^l + O(|z|^{l+1})$, this winding number is simply $l/2 > 0$. This concludes the proof. \end{proof}

The most important application of Proposition \ref{prop:count_of_boundary_singularities} for our purposes is the following resolution of singularities result. It states that we can replace any somewhere injective $J$-curve in our setting with a well-behaved surface in the same surface class. 

\begin{lemma} \label{lem:curve_to_admissible_surface} Fix a non-constant, simple, finite energy $J$-holomorphic map 
\[
u:(\Sigma,\partial \Sigma) \to (\R \times Y,\R \times \Lambda)
\]
Then there exists an symplectic, well-immersed surface
\[\iota:(S,\partial_\circ S) \to ([0,1] \times Y,[0,1] \times \Lambda)\]
representing the surface class $A$ corresponding to $u$ such that
\begin{itemize}
    \item[(a)] $S$ has $\delta(C)$ transverse, positive double points in the interior $\op{int}(S)$. 
    \item[(b)] $S$ has $\epsilon(C)$ transverse, positive double points in the boundary $\partial_\circ \Sigma$.  
    \item[(c)] The domain $S$ is the same as the domain of $u$.
    \item[(d)] $\iota$ agrees with $C$ in a neighborhood of $\partial_\pm \Sigma$.
\end{itemize}
\end{lemma}

\begin{proof} By Lemma \ref{lem:simple_vs_injective}, we know that $u$ has a finite number of isolated non-injective points $\op{Sing}(u) \subset \Sigma$. To acquire a map $\iota:S \to \R \times Y$ satisfying $(c)$ and $(d)$, we simply modify $u$ in a neighborhood the points $\op{Sing}(u)$. To ensure that this map satisfies (a) and (b), we perform this modification by applying Proposition \ref{prop:count_of_interior_singularities} at the interior singularities and Proposition \ref{prop:count_of_boundary_singularities} to boundary singularities.  
\end{proof}

As a corollary, we can prove the Legendrian adjunction result stated in the introduction.

\begin{corollary}[Theorem \ref{intro:legendrian_adjunction}] \label{cor:legendrian_adjunction_body} Let $C$ be an finite energy, simple holomorphic curve in $\R \times Y$ with boundary on $\R \times \Lambda$. Then
\[\mu(C,\tau) = 2(\bar{\chi}(C) + Q_\tau(C) + w_\tau(C) - 2\delta(C) - \epsilon(C)) \]
\end{corollary}

\begin{proof} By Lemma We \ref{lem:curve_to_admissible_surface}, we can find a symplectic, well-immersed surface $S$ representing the same surface class of $C$ that has the same domain, writhe and count of singularities, so this follows from topological adjunction (Theorem \ref{thm:topological_adjunction}).
\end{proof}

\section{$J$-Holomorphic Currents With Boundary} We can now formally introduce $J$-holomorphic currents and describe their basic properties.

\begin{definition} A \emph{$J$-holomorphic current with boundary} $C$ in an almost complex manifold with boundary $(X,J)$ is a finite set of pairs
\[(C_i,m_i) \qquad\text{for}\qquad i = 1,\dots,k\] where $C_i$ is a distinct, simple, somewhere injective $J$-holomorphic curves $C \subset X$ with boundary on $\partial X$ and $m_i$ is a positive integer called the \emph{multiplicity} of $C_i$. 
\end{definition}

We will utilize a specific topology on the space of currents, introduced by Taubes \cite[\S 2.4]{Hutchings_lectures_on_ECH}.

\begin{definition} \label{def:convergence_of_currents} A sequence of $J$-holomorphic currents $\mathcal{C}^\nu = \{(C^\nu_i,m^\nu_i)\}$ for $\nu \in \N$ in an almost complex manifold $(X,J)$ \emph{converges} to an $J$-holomorphic current $\mathcal{C} = \{(C_i,m_i)\}$ if
\begin{itemize}
    \item $\mathcal{C}^\nu$ converges to $\mathcal{C}$ as a point set. More precisely, the point sets $S^\nu = \cup_i \; C_i^\nu$ converge to the set $S = \cup_i \; C_i$ in the Hausdorff metric, on each compact subset of $X$.
    \item $\mathcal{C}^\nu$ converges to $\mathcal{C}$ as a current, i.e. for every compactly supported $2$-form $\sigma$ on $X$, we have
    \[
    \lim_{\nu \to \infty} \; \sum_i \Big(  m^\nu_i \cdot \int_{C^\nu_i} \sigma \Big) \to \sum_i \Big( m_i \cdot \int_{C_i} \sigma \Big)  
    \]
\end{itemize}
\end{definition}

A key property of holomorphic currents, which does not hold for curves, is a general compactness property given a uniform action bound.  

\begin{theorem} (Taubes, cf. \cite{Hutchings_lectures_on_ECH}) Let $(X,\omega)$ be a symplectic manifold with boundary and compatible almost complex structure $J$. Let $C_i$ be a sequence of $J$-holomorphic currents with boundary on $\partial X$ such that
\[\int_{C_i} \omega < C \qquad\text{for all $i$ and some $C > 0$}\]
Then there is a subsequence that converges to a $J$-holomorphic current $\mathcal{C}$.
\end{theorem}

\subsection{Currents With Boundary In Symplectizations} \label{subsec:currents_in_symplectizations} We are interested in currents in the symplectization of a pair $(Y,\Lambda)$ of a convex sutured contact $3$-manifold $Y$ and a pair of closed Legendrians $\Lambda_\pm \subset \partial_\pm Y$. 

\vspace{3pt}
In this setting, $J$-holomorphic maps, curves and currents admit a natural $\R$-action, denoted by
\[
\mathcal{C} \mapsto \mathcal{C} + s \qquad\text{for any}\qquad s \in \R
\]
On curves, this action is given by composing the map $u:\Sigma \to \R \times Y$ with translation $\R \times Y \to \R \times Y$ by $s$. Moreover, the translation invariant $J$-holomorphic currents coincide with cylinders over sets of orbits and chords, as follows.

\begin{example} The \emph{trivial current} $\R \times \Xi$ over an orbit-chord set $\Xi$ consisting of simple orbits and chords $\gamma_i$ of multiplicity $m_i$ is the current of pairs
\[(\R \times \gamma_i,m_i)\]\end{example}

\noindent Every tame enough current in $\R \times Y$ is asymptotic to an orbit-chord sets near infinity, as follows.

\begin{lemma} Let $\mathcal{C}$ be a $J$-holomorphic current in $(\R \times Y,\R \times \Lambda)$ such that each curve $C_i$ in $\mathcal{C}$ is proper and finite energy. Then there are orbit sets $\Xi_+$ and $\Xi_-$ such that
\begin{equation} \label{eqn:convergence_of_ends_current}
\mathcal{C} + s_i \xrightarrow{\text{Def  \ref{def:convergence_of_currents}}} \R \times \Xi_\pm \qquad\text{for any sequence $s_i \in \R$ with $s_i \to \pm \infty$}
\end{equation}
\end{lemma}

\begin{definition} A $J$-holomorphic current $\mathcal{C}$ from $\Xi_+$ to $\Xi_-$ is a current satisfying (\ref{eqn:convergence_of_ends_current}). We denote the space of currents from $\Xi_+$ to $\Xi_-$ by
\[\mathcal{M}(\Xi_+,\Xi_-)\]
We denote the quotient by the $\R$-action on currents by $\mathcal{M}(\Xi_+,\Xi_-)/\R$. More generally, a \emph{broken $J$-holomorphic current} $\bar{\mathcal{C}}$ from $\Xi_+$ to $\Xi_-$ is a sequence
\[\mathcal{C}_i \in \mathcal{M}(\Xi_i,\Xi_{i+1})/\R \qquad \text{for }i = 1,\dots,m\qquad\text{where}\qquad \Xi_1 = \Xi_+ \text{ and }\Xi_{m+1} = \Xi_-\]
The space of broken $J$-holomorphic currents from $\Xi_+$ to $\Xi_-$ is denoted by $\bar{\mathcal{M}}(\Xi_+,\Xi_-)$.  \end{definition}

\begin{remark} Any $J$-holomorphic current in $(\R \times Y,\R \times \Lambda)$ appearing in the rest of this paper will be assumed to be finite energy and proper, so that
\[\mathcal{C} \in \mathcal{M}(\Xi,\Theta) \qquad\text{for some orbit-chord sets $\Xi$ and $\Theta$}\]
\end{remark}

There is an tautological map sending a finite energy, proper $J$-holomorphic curve $C$ from orbit-chord set $\xi$ to orbit-chord set $\Theta$ to a corresponding surface class $[C]$. This extends to a map
\begin{equation}
\bar{\mathcal{M}}(\Xi,\Theta) \to S(\Xi,\Theta) \qquad\text{denoted by}\qquad \bar{\mathcal{C}} \mapsto [\bar{\mathcal{C}}]
\end{equation}
that is compatible with union and composition in the sense that
\begin{align*}
[\mathcal{C} \cup \mathcal{D}] = [\mathcal{C}] \cup [\mathcal{D}] & \qquad\text{for any pair of currents $\mathcal{C}$ and $\mathcal{D}$}\\
[\bar{\mathcal{C}} \circ \bar{\mathcal{D}}] = [\bar{\mathcal{C}}] \circ [\bar{\mathcal{D}}] \: &  \qquad\text{for any pair of broken currents $\bar{\mathcal{C}}$ and $\bar{\mathcal{D}}$}\end{align*}

\subsection{Gromov Compactness} We next introduce the notion of Gromov compactness of broken currents. This compactness result resembles SFT compactness of Bourgeouis-Eliashberg-Hofer-Wyosocki-Zehnder, but does not depend on it. 

\begin{definition} A sequence of (equivalence classes of) $J$-holomorphic currents 
\[\mathcal{C}^\nu \in \mathcal{M}(\Xi_+,\Xi_-)/\R \qquad\text{for}\qquad \nu \in \N\] \emph{converges in the Gromov topology} to a  broken $J$-holomorphic current $\bar{\mathcal{C}} = (\mathcal{C}_1,\dots,\mathcal{C}_m)$ if there are representatives of $\mathcal{C}^\nu$ and $\mathcal{C}_i$, denoted respectively by
\[\mathcal{S}^\nu \in \mathcal{M}(\Xi_+,\Xi_-) \qquad\text{and}\qquad \mathcal{S}_i \in \mathcal{M}(\Xi_i,\Xi_{i+1})\]
and a set of sequences $s^\nu_i \in \R$ for each $i = 1,\dots,m$ such that
\[\mathcal{S}^\nu + s^\nu_i \xrightarrow{\text{Def \ref{def:convergence_of_currents}}} \mathcal{S}_i \qquad \text{as }\nu \to \infty\]\end{definition}

The using the same argument as \cite{Hutchings_lectures_on_ECH}, we have the following compactness result for currents.

\begin{proposition}[Gromov Compactness] \label{thm:Gromov_compactness} Let $\mathcal{C}^\nu \in \mathcal{M} (\Xi_+,\Xi_-)$ be a sequence of currents. Then, after passing to a subsequence, there is a broken $J$-holomorphic current $\bar{\mathcal{C}}$ from $\Xi_+$ to $\Xi_-$ such that
\[\mathcal{C}^\nu \to \bar{\mathcal{C}} \qquad\text{and}\qquad [\bar{\mathcal{C}}] = [\mathcal{C}^\nu]\]
\end{proposition}

\subsection{Intersection Numbers And Linking Of Currents} It will be convenient to extend the intersection number and linking number to currents under a certain disjointness hypothesis.

\begin{definition}  A pair of $J$-holomorphic currents $\mathcal{C}$ and $\mathcal{D}$ have \emph{distinct components} if the component curves $(C_i,m_i)$ and $(D_j,n_j)$ satisfy
\[C_i \neq D_j \qquad\text{for all }i,j\]
\end{definition}

Currents with disjoint components have a well-defined geometric intersection number that counts actual intersections (with multiplicity). 

\begin{definition} \label{def:geometric_intersection_number} The \emph{geometric intersection number} of $J$-holomorphic currents  $\mathcal{C}$ and $\mathcal{D}$ with distinct components is the non-negative half-integer given by 
\[\mathcal{C} \cdot \mathcal{D} \qquad\text{is given by}\qquad \mathcal{C} \cdot \mathcal{D} := \sum_{i,j} m_i \cdot n_j \cdot (C_i \cdot C_j)\]
where the geometric intersection $C \cdot D$ of two distinct, somewhere injective $J$-holomorphic curves $C$ and $D$ is defined as the following count of singularities.
\[
C \cdot D := (\epsilon(C \cup D) - \epsilon(C) - \epsilon(D)) + \frac{1}{2}( \delta(C \cup D) - \delta(C) - \delta(D))
\]\end{definition}

\begin{remark} It is possible to generalize the geometric intersection number to a self-intersection number of curves and currents (see \cite{Hutchings_index_revisited}). However, we will not carry this out in this paper.
\end{remark}

The finiteness (and thus well-definedness) of the geometric intersection number follows from the analysis of singularities of $J$-holomorphic curves in our setting (see Section \ref{sec:J_holomorphic_currents}).

\begin{definition}\label{def:linking_of_currents} The \emph{linking number} of two currents $\mathcal{C}$ and $\mathcal{D}$ with disjoint components with respect to a trivialization $\tau$, denoted by
\[l_\tau(\mathcal{C},\mathcal{D})\]
is defined as follows. Let $(C_i,m_i)$ and $(D_j,n_j)$ denote the component curves of $\mathcal{C}$ and $\mathcal{D}$. Then
\[
l_\tau(\mathcal{C},\mathcal{D}) = \sum_{i,j} m_i \cdot n_j \cdot l_\tau(C_i,D_j)
\]
Here we adopt the following conventions for the linking number of curves.
\begin{itemize}
\item If $C$ is non-trivial curve with an end on $\eta$ and $D = \R \times \eta$ is a trivial curve, then
\[
l_\tau(C,D) = \op{wind}_\tau(\zeta_+) - \op{wind}_\tau(\zeta_-)
\]
where $\zeta_\pm$ re the asymptotic braids of $C$ at $\pm \infty$.
\item If $C$ and $D$ are both non-trivial, then
\[
l_\tau(C,D) = l_\tau(\zeta_+,\xi_+) - l_\tau(\zeta_-,\xi_-)
\]
where $\zeta_\pm$ and $\xi_\pm$ are the asymtptotic braids of $C$ and $D$, respectively.
\end{itemize}
Note that the asymptotic braids can be empty. In this case, the winding number and linking number with any other braid are zero by convention. \end{definition}

This extended linking number transforms in precisely the same way as its topological counterpart under change of trivialization.

\begin{lemma} \label{lem:current_lk_trivialization} Let $\mathcal{C} \in \mathcal{M}(\Xi_+,\Xi_-)$ and $\mathcal{D} \in \mathcal{M}(\Theta_+,\Theta_-)$ be $J$-holomorphic currents with distinct components.
\[\eta \quad\text{of multiplicity $m$ in $\Xi_\pm$ and $n$ in $\Theta$}\]
Then the linking number of $\mathcal{C}$ and $\mathcal{D}$ differ as follows.
\[l_\tau(\mathcal{C},\mathcal{D}) - l_\sigma(\mathcal{C},\mathcal{D}) = m \cdot n \cdot (\sigma - \tau)\]
\end{lemma}

The proof is straightforward so we omit it. The linking number is also related to the geometric intersection number as follows.

\begin{lemma} \label{lem:intersection_vs_lk_currents} Let $\mathcal{C}$ and $\mathcal{D}$ be $J$-holomorphic currents with distinct components. Then
\[\mathcal{C} \cdot \mathcal{D} = Q_\tau(\mathcal{C},\mathcal{D}) + l_\tau(\mathcal{C},\mathcal{D})\]
\end{lemma}

\begin{proof} By bilinearity with respect to union, we assume that $\mathcal{C}$ and $\mathcal{D}$ are somewhere injective (and even connected) curves, and that $C$ is non-trivial. Then $C \cup D$ is somewhere injective and we can use Lemma \ref{lem:curve_to_admissible_surface} to replace $C$ and $D$ with well-immersed symplectic surfaces $S$ and $T$ in the same surface class that satisfy
\[C \cdot D = \#(\op{int}(S) \cap \op{int}(T)) + \frac{1}{2} \cdot \#(\partial_\circ S \cap \partial_\circ T) \quad\text{and}\quad l_\tau(C,D) = l_\tau(S,T)\]
Note that we are using the fact that, when $D$ is a trivial cylinder over an orbit or chord
\[l_\tau(C,D) =  \op{wind}_\tau(\zeta_+) - \op{wind}_\tau(\zeta_-)\]
The result follows from Definition \ref{def:relative_intersection}.\end{proof}

\section{The Conley-Zehnder And Fredholm Indices} \label{sec:CZ_and_Fredholm_indices} In this section, we review several versions of the Conley-Zehnder index. In particular, we introduce the Conley-Zehnder index of a Lagrangian path. 

\subsection{Paths Of Lagrangians and the Conley-Zehnder Index}

Recall that given a path of symplectic matrices, $\Phi_t, t\in [0,1]$ satisfying $\Phi_0 = Id$ and $\Phi_1$ does not have $1$ as an eigenvalue by nondegeneracy, as in \cite{wendl2016lectures} we can associate an integer valued Conley-Zehnder index to this path of symplectic matrices, which we call the Conley-Zehnder index. We refer the reader to lecture 3 of \cite{wendl2016lectures} for properties of this index.

 We next provide a definition for the Conley-Zehnder index of a path of Lagrangian sub-spaces. We adopt the perspective that this is naturally a \emph{half-integer}.

\begin{definition} A path $L:[0,1] \to \op{LGr}(2n)$ is \emph{non-degenerate} if the ends are transverse.
\[L_1 \pitchfork L_0\]
\end{definition}

\begin{definition} \label{def:CZ_for_Lagrangians} The \emph{Conley-Zehnder index} of a non-degenerate path of Lagrangians $L:[0,1] \to \op{LGr}(2n)$ is the half-integer
\[
\CZ(L) := \mu(\bar{L}) + \frac{n}{2} \in \frac{1}{2} \Z
\]
Here $\bar{L}:S^1 \to \op{LGr}(2n)$ is the loop of Lagrangians constructed as follows. Choose a complex structure $J$ on $\C^n$ such that
\[
J \text{ is compatible with }\omega_0 \qquad\text{and}\qquad J L_1 = L_0
\]
Then $\bar{L}$ is the loop acquired by joining $L$ to the path $\exp(-J \cdot \frac{\pi t}{2})$ for $t \in [0,1]$ from $L_1$ to $L_0$.  
\end{definition}

\begin{remark} By abuse of terminology and notation, we will refer to $\mu(\bar{L})$ as the Maslov index of the path $L$ and denote it by $\mu(L)$.
\end{remark}

We can also define the Conley-Zehnder index of a Lagrangian path as half of the ordinary Conley-Zehnder index of an associated path of symplectic matrices. See \cite{Gutt2014} for an exposition of how various flavours of Conley-Zehnder indices are related to each other.

\subsection{Bundles Over $1$-Manifolds} Next, we introduce the notion of an asymptotic operators and discuss the Conley-Zehnder index of an asymptotic operator. 

\vspace{3pt}

Let $S$ be an oriented $1$-manifold with boundary. Let $(E,F) \to (S,\partial S)$ be a symplectic bundle pair and fix a symplectic connection
\[\nabla:C^\infty(S,\partial S;E,F) \to C^\infty(S;\Omega^1S \otimes E) \]

\begin{definition} The connection $\nabla$ is \emph{non-degenerate} if the kernel of $\nabla$ is trivial.
\end{definition}

Given a choice of trivialization $\tau:(E,F) \simeq (\C^n,\R^n)$, we can assign a Conley-Zehnder index to the above data. To explain how, let $\eta \subset S$ be a component of $S$ and choose an oriented identification
\[\phi:\R/\Z \simeq \eta \qquad\text{or}\qquad \phi:[0,1] \simeq \eta\]
In either case, parallel transport via the connection $\nabla$ determines a family of symplectic maps
\[\op{PT}_t:E_{\phi(0)}\to E_{\phi(t)} \quad\text{for each}\quad t \in [0,1]\]
Given the choice of trivialization $\tau$, this then determines a $1$-parameter family of matrices
\[\Phi^\tau:[0,1] \to \op{Sp}(2n) \quad\text{with}\quad \Phi^\tau_t := \tau_{\phi(t)} \circ \op{PT}_t \circ \tau^{-1}_{\phi(0)}\]
and a family of Lagrangians $L^\tau := \Phi^\tau(\R^n)$. 

\begin{definition} \label{def:CZ index} The \emph{Conley-Zehnder index} $\CZ(\eta,\tau)$ of a component $\eta \subset S$ with respect to a trivialization $\tau$ of $(E,F)$ is
\[\CZ(\Phi^\tau) \text{ if }\eta \simeq \R/\Z \qquad\text{and}\qquad \CZ(L^\tau) \text{ if }\eta \simeq [0,1]\]
The Conley-Zehnder index $\CZ(S,\tau)$ of $(E,F,\nabla)$ with respect to $\tau$ is simply the sum
\[\CZ(S,\tau) := \sum_\eta \CZ(\eta,\tau)\]\end{definition}

Let $S$ be a compact $1$-manifold diffeomorphic to the interval $[0,1]$ and let $(E,F)\rightarrow (S,\partial S)$ be a symplectic bundle pair. From this point onward assume $E$ is two-dimensional. Given a nowhere vanishing vector field $R$ on $S$, a compatible complex structure $J$ on $E$ and a symplectic connection $\nabla$, we define the associated asymptotic operator to be
\begin{equation*}
    A:C^\infty(S,\partial S;E,F)\rightarrow C^\infty(S,E) \quad AX\coloneqq -J\nabla_RX
\end{equation*}

\begin{lemma}
The asymptotic operator $A$ satisfies the following basic properties:
\begin{enumerate}
    \item $A$ has trivial kernel if and only if $\nabla$ is non-degenerate.
    \item All eigenvalues of $A$ are simple.
    \item All eigenfunctions of $A$ are nowhere vanishing.
\end{enumerate}
\end{lemma}

\begin{proof}
The first assertion is immediate because the operators $A$ and $\nabla$ have the same kernel. Let $\lambda$ be an eigenvalue of $A$. Since $A-\lambda$ is a linear first order ordinary differential operator, the initial value homomorphism
\begin{equation*}
    \operatorname{ker}(A-\lambda)\rightarrow F_p\quad X\mapsto X(p)
\end{equation*}
is injective. Here $p\in\partial S$ is a boundary point of $S$. Thus the eigenspace of $\lambda$ has dimension $1$. This proves the second assertion. Let $X\in\operatorname{ker}(A-\lambda)\setminus\{0\}$ be an eigenvector. Again because $A-\lambda$ is a linear first order ordinary differential operator, $X$ vanishes at one single point $q\in S$ if and only if it is identically equal to zero. The third assertion follows.
\end{proof}

In particular, given an asymptotic operator $A$, a symplectic trivialization $\tau$ of $(E,F)$ and an eigenvalue $\lambda$ of $A$, there is a well-defined half-integer-valued winding number $w(\tau,A;\lambda)\in\frac{1}{2}\Z$ associated to the eigenfunction of $A$ with eigenvalue $\lambda$. To be specific, under our choice of trivialization $\tau$, if $e(t) : [0,1] \rightarrow \C$ is such an eigenfunction that maps the end points $\{0,1\}$ of $[0,1]$ to $\R \subset \C$, using the fact that $e(t)$ is never vanishing, the winding number is $\frac{1}{2}$ times the number of half turns $e(t)$ makes about the origin.

\begin{lemma}\label{lem:winding monotonic}
The function $w(\tau,A;\cdot):\operatorname{Spec}A\rightarrow\frac{1}{2}\Z$ is monotonic and a bijection.
\end{lemma}

\begin{proof}
We solve this problem by deforming to a simplified operator $A_0$ with the same Conley-Zehnder index (see below for description). 
Recall the asymptotic operator $A$ is of the form $-J_0\partial_t - S(t) :C^\infty([0,1],\{0,1\};\R^2,\R\times 0)\mapsto C^\infty([0,1],\R^2)$. Its Conley-Zehnder index is computed by considering the path of symplectic matrices $\Phi(t)$ satisfying
\begin{itemize}
    \item $\Phi(0) = Id$
    \item $-J_0 \partial_t \Phi(t) - S(t) \Phi(t) =0$
\end{itemize}
and considering the loop of Lagrangians $\Phi_t (\R) \subset \C$. In particular, for general operators of this form, the path of Lagrangians $\Phi_t (\R)$ being nondegenerate implies that $A$ has trivial kernel.

Now suppose we have an asymptotic operator $A_0$ satisfying the conclusion of the lemma (we will specify what this operator is directly in a bit) such that $A$ and $A_0$ have the same Conley-Zehnder index as described above, then we can 
choose a path $(A_s)_{s\in [0,1]}$ of nondegenerate asymptotic operators connecting $A_0$ to $A_1=A$. Let $\lambda_0^{k/2}$ for $k\in\Z$ denote the eigenvalues of $A_0$ labelled by the winding number $k/2$. Since the spectrum of $A_s$ is simple for all $s\in [0,1]$, we may continuously extend $\lambda_0^{k/2}$ to a family $\lambda_s^{k/2}$ such that the spectrum of $A_s$ is given by $\lambda_s^{k/2}$ for $k\in\Z$. By assumption, the function $k\mapsto\lambda_0^{k/2}$ is monotonic. Thus the same must continue to hold for $s>0$. We may choose families $X_s^{k/2}$ of associated eigenvectors. Since $X_s^{k/2}$ is nowhere vanishing for all $s$, the winding number of $X_1^{k/2}$ agrees with the winding number of $X_0^{k/2}$, which is equal to $k/2$. Thus the winding number of $\lambda_1^{k/2}$ is actually equal to $k/2$. This proves that the lemma holds for $A=A_1$.\\

Let us prove the lemma for the following explicit asymptotic operator $A_0$.
\begin{equation*}
    A_0 : C^\infty([0,1],\{0,1\};\R^2,\R\times 0)\mapsto C^\infty([0,1],\R^2)\quad X\mapsto -J_0\partial_t X - \pi/2(1+2l) X
\end{equation*}
This can be solved explicitly as follows. Consider $(x(t),y(t)) \in \R^2$, this satisfies the equation
\[
-
\begin{pmatrix}
0&-1\\
1&0
\end{pmatrix}
\begin{pmatrix}
x'\\
y'
\end{pmatrix}
- \pi/2(1+2l) 
\begin{pmatrix}
x\\
y
\end{pmatrix}
=\lambda 
\begin{pmatrix}
x\\
y
\end{pmatrix}
\]
This is the system of equations
\[
y'=(\lambda +\pi/2(1+2l)) x, \quad  x'=-(\lambda +\pi/2(1+2l)) y
\]
differentiating twice to get
\[
y''=-(\lambda +\pi/2(1+2l))^2y, \quad  x''=-(\lambda +\pi/2(1+2l)) ^2x
\]
with boundary conditions:
\[
y(0)=y(1)=0
\]
Hence the eigenvalues of the asymptotic operator are $\lambda = \pi/2 (2n+1)$ ($n\in \Z$), and the corresponding eigenvectors are
\[
y_n(t)=a\sin((\lambda + \pi/2(2l+1)) t), \quad
x_n(t) = a\cos((\lambda + \pi/2(2l+1)) t)
\]
For each choice of eigenvalue $\lambda$ of $A_0$, we may compute the winding number $w$ from these explicit eigenvectors and see that the theorem is satisfied.

In addition, with this choice of $A_0$ we also compute the Conley-Zehnder index associated to this choice of trivialization. We consider the matrix satisfying
\[
\partial_t \Phi = \pi/2(1+2l) J\Phi
\]
Then identifying $\R^2$ with $\C$, the matrix $\Phi$ can be identified with $e^{i \pi/2(1+2l)t}$. We consider the path of Lagrangians $\Phi(t)(\R)$ and append the path $e^{-i\frac{\pi}{2} t} (i\R)$ to make this a loop of Lagrangians. We see this loop has Maslov index $l$, and hence the Conley-Zehnder index of the operator is $l+\frac{1}{2}$. The fact that $\{l+\frac{1}{2}|l\in\Z\}$ enumerates all possible Conley-Zehnder indices that an asymptotic operator $A$ can have, proves that the path $(A_s)_{s\in[0,1]}$ we chose exists.
\end{proof}

\begin{lemma}\label{lem:index and winding of the asymptotic operator}
Suppose that $\nabla$ is non-degenerate. We have 
\begin{equation*}
    \operatorname{CZ}(S,\tau) = \min\{w(\tau,A;\lambda)\mid \lambda\in\operatorname{Spec}(A)\cap\R_{>0}\}
    +
    \max\{w(\tau,A;\lambda)\mid \lambda\in\operatorname{Spec}(A)\cap\R_{<0}\}
\end{equation*}
\end{lemma}

\begin{proof}
We assume the Conley-Zehnder index of $A$ is $l+\frac{1}{2}$ ($l\in\Z$). Same as what we did in the proof of Lemma \ref{lem:winding monotonic}, we begin by choosing a path of non-degenerate asymptotic operators $(A_s)_{s\in[0,1]}$ such that $A_1=A$ and $A_0=-J_0\partial_t-\pi/2(1+2l)$. As in the proof of Lemma \ref{lem:winding monotonic}, let $\lambda_s^{k/2}$ denote the unique eigenvalue of $A_s$ with winding number $k/2$. Since for every $s$, the asymptotic operator $A_s$ is non-degenerate, we observe that $0$ is never an eigenvalue of $A_s$, and hence the paths of eigenvalues $(\lambda_s^{k/2})_{s\in[0,1]}$ do not cross $0$. It follows that
\[
\min\{w(\tau,A_s;\lambda)\mid \lambda\in\operatorname{Spec}(A_s)\cap\R_{>0}\}
    +
    \max\{w(\tau,A_s;\lambda)\mid \lambda\in\operatorname{Spec}(A_s)\cap\R_{<0}\}
\]
is constant for $s\in[0,1]$, and hence it suffices to verify the claim for $A_0$. 

By our explicit calculation in the proof of Lemma \ref{lem:winding monotonic}, the smallest positive eigenvalue of $A_0$ is $\frac{\pi}{2}$ with winding number $\frac{l+1}{2}$, and the largest negative eigenvalue of $A_0$ is $-\frac{\pi}{2}$ with winding number $\frac{l}{2}$, so the claim holds true for $A_0$.
\end{proof}

\subsection{Fredholm Index}  \label{subsec:fredholm index}
We the previous section on Conley-Zehnder indices we are now prepared to discuss the Fredholm index for pseudo-holomorphic curves in our setting.
\begin{theorem}
\label{theorem:index_formula}
Let $Y$ be a sutured contact manifold of dimension $2n-1$. Let $u:(\Sigma,\partial_\circ\Sigma)\rightarrow (\R \times Y,\R \times \Lambda)$ be a pseudoholomorphic curve. Suppose that $u$ is positively asymptotic to orbits and chords $\gamma_1^+,\dots,\gamma_k^+$ and negatively asymptotic to orbits and chords $\gamma_1^-,\dots,\gamma_\ell^-$. Then the Fredholm index of $u$ is given by
\begin{equation*}
    \operatorname{ind}(u) = (n-3)\bar{\chi}(\Sigma) + \mu(u,\tau) + \sum\limits_{i=1}^k \operatorname{CZ}_\tau(\gamma_i^+) - \sum\limits_{i=1}^\ell \operatorname{CZ}_\tau(\gamma_i^-)
\end{equation*}
\end{theorem}
The way to interpret the above formula is as follows: if $u$ is somewhere injective, and $J$ is generic, then $u$ lives in a moduli space that has dimension given by the index formula.

Let $\Sigma$ be a compact connected Riemann surface, possibly with boundary. Let $\Gamma\subset\Sigma$ be the finite set of punctures (previously we denoted this by $\partial_o \Sigma$), possibly on the boundary. Let $\Gamma=\Gamma_i\cup \Gamma_b$ be the partition of $\Gamma$ into interior and boundary punctures. Moreover, let $\Gamma=\Gamma^+\cup \Gamma^-$ be a partition of $\Gamma$ into positive and negative punctures. Let $\dot{\Sigma}=\Sigma\setminus \Gamma$ denote the punctured surface. Let $(E,F)\rightarrow\dot{\Sigma}$ be a bundle pair and let $\tau$ denote a trivialization near the punctures.

\begin{lemma}
\label{lem:index_formula_abstract_cauchy_riemann_operators}
Let
\begin{equation*}
    D : W^{k,p}(E,F)\rightarrow W^{k-1,p}(\Lambda^{0,1}T^*\dot{\Sigma}\otimes E)
\end{equation*}
be a Cauchy-Riemann type operator with non-degenerate asymptotic operators $A_z^\tau$ at the punctures $z\in \Gamma$. Then $D$ is Fredholm and its index is given by
\begin{equation*}
\operatorname{ind}(D) = n\cdot \bar{\chi}(\dot{\Sigma}) + \mu^\tau(E,F) + \sum\limits_{z\in\Gamma^+} \operatorname{CZ}(A^\tau_z)-\sum\limits_{z\in\Gamma^-} \operatorname{CZ}(A^\tau_z).
\end{equation*}
\end{lemma}

\begin{proof}
The form of the asymptotic operators $A_z^\tau$ is described in \cite{ct2021}, and using our Definition \ref{def:CZ index} we can associate to it a Conley-Zehnder index.
The lemma is exactly the statement of the Fredholm index formula proved in \cite{ct2021}. Note that \cite{ct2021} uses different definitions for the Euler characteristic of a punctured surface with boundary and for the Conley-Zehnder index. The Euler characteristic $X(\Sigma,\Gamma^\pm)$ appearing in \cite{ct2021} is related to the orbifold Euler characteristc $\bar{\chi}(\dot{\Sigma})$ via
\begin{equation*}
    X(\Sigma,\Gamma^\pm) = \bar{\chi}(\dot{\Sigma}) + \frac{1}{2}\#\Gamma_b^+ - \frac{1}{2}\#\Gamma_b^-.
\end{equation*}
At interior punctures the Conley-Zehnder index $\operatorname{\mu_{CZ}}(A_z^\tau)$ showing up in \cite{ct2021} agrees with our Conley-Zehnder index $\operatorname{CZ}(A_z^\tau)$. At boundary punctures, the relationship is given by
\begin{equation*}
    \operatorname{\mu_{CZ}}(A_z^\tau) = \operatorname{CZ}(A_z^\tau) - \frac{n}{2}
\end{equation*}
where $n$ denotes the complex rank of $E$. Plugging these identities into the index formula given in \cite{ct2021} readily yields the lemma.
\end{proof}

\begin{lemma} [Section 3, \cite{wendlauto}]
\label{lem:dimension_of_teichmueller_slice}
The real dimension of the moduli space of punctured Riemann surfaces with boundary of topological type $(\Sigma,\Gamma=\Gamma_i\cup\Gamma_b)$ is given by
\begin{equation*}
    -3\chi(\Sigma) + 2\#\Gamma_i+\#\Gamma_b = -3\bar{\chi}(\dot{\Sigma})-\#\Gamma_i-\frac{1}{2}\#\Gamma_b
\end{equation*}
\end{lemma}

\begin{proof}[Proof of Theorem \ref{theorem:index_formula}]
The moduli space containing $u$ is the zero set of a certain Fredholm section. For a general description of this setup see Section 3.3 of \cite{wendlauto} and Section 5 of \cite{LCH}. The linearization of this Fredholm section at $u$ is a Fredholm operator of the form 
\begin{equation*}
D:W^{k,p,\delta}(u^*T\R \times Y,u^*T\R \times \Lambda)\oplus V \oplus T \rightarrow W^{k-1,p,\delta}(\Lambda^{0,1}T^*\dot{\Sigma}\otimes u^*T\R \times Y).
\end{equation*}
Here $\delta$ is a small positive exponential weight. The space $W^{k,p,\delta}(u^*T\R \times Y,u^*T\R \times \Lambda)$ means we use a weighted Sobolev space with weight of the form $e^{\delta |s|}$ near each puncture of $\dot{\Sigma}$ around which we have already chosen cylindrical coordinates of the form $(s,t) \in [0,\pm\infty) \times S^1$ or $[0,\pm\infty) \times [0,1]$ The dimension of the tangent space $V$ of the space of asymptotic markers is given by $2\#\Gamma_i+\#\Gamma_b$. By Lemma \ref{lem:dimension_of_teichmueller_slice} the dimension of the Teichm\"{u}ller slice $T$ is given by $-3\bar{\chi}(\dot{\Sigma})-\#\Gamma_i-\frac{1}{2}\#\Gamma_b$. The restriction $D'$ of $D$ to $W^{k,p,\delta}(u^*T\R \times Y,u^*T\R \times \Lambda)$ is a Cauchy-Riemann type operator. Let $z\in\Gamma$ be a puncture and let $\gamma$ denote the corresponding Reeb chord or orbit. Then the asymptotic operator of $D'$ at $z$ takes the form $A_z^\tau\oplus -i\partial_t$ where $A_z^\tau$ acts on sections of $\gamma^*\xi$ and $-i\partial_t$ acts on sections of the trivial bundle spanned by $\partial_s$ and $R$. It follows from Lemma \ref{lem:index_formula_abstract_cauchy_riemann_operators} that if the exponential weight $\delta$ is sufficiently small, the index of $D'$ is given by
\begin{equation*}
\operatorname{ind}(D') = n\cdot\bar{\chi}(\dot{\Sigma}) + \mu(u,\tau) + \sum\limits_{z\in \Gamma^+} \operatorname{CZ}(A_{z}^\tau\oplus (-i\partial_t+\delta))-\sum\limits_{z\in\Gamma^-} \operatorname{CZ}(A_{z}^\tau\oplus (-i\partial_t-\delta)).
\end{equation*}
The Conley-Zehnder index is additive under direct sums, i.e.
\begin{equation*}
    \operatorname{CZ}(A_z^\tau\oplus (-i\partial_t\pm \delta)) = \operatorname{CZ}(A_z^\tau) + \operatorname{CZ}(-i\partial_t\pm \delta)
\end{equation*}
A direct computation shows that if we regard $-i\partial_t\pm\delta$ as an asymptotic operator over the circle, then
\begin{equation*}
    \operatorname{CZ}(-i\partial_t\pm\delta) = \mp 1
\end{equation*}
for $\delta>0$ sufficiently small. If we regard it as an asymptotic operator over the interval, then
\begin{equation*}
    \op{CZ}(-i\partial_t\pm\delta) = \mp 1/2.
\end{equation*}
Combining these identities we compute
\begin{equation*}
    \operatorname{ind}(D) = \operatorname{ind}(D') + \operatorname{dim}V + \operatorname{dim}T = (n-3)\bar{\chi}(\Sigma) + \mu(u,\tau) + \sum\limits_{i=1}^k \operatorname{CZ}_\tau(\gamma_i^+) - \sum\limits_{i=1}^\ell \operatorname{CZ}_\tau(\gamma_i^-) \qedhere
\end{equation*}
\end{proof}

\subsection{Regular Almost Complex Structures} We conclude this section by reviewing the notion of regular almost complex structures. 

\begin{definition} A tailored almost complex structure $J$ on $(Y,\Lambda)$ is \emph{regular} if every somewhere injective, finite energy $J$-holomorphic curve in $(\R \times Y,\R \times \Lambda)$ is transversely cut out. 
\end{definition}

A standard argument shows that regular, tailored almost complex structures are generic. In particular, we have the following proposition.

\begin{proposition} \label{prop:regular_is_generic} The set $\mathcal{J}_{\op{reg}}(Y,\Lambda)$ of regular, tailored almost complex structures on $(Y,\Lambda)$ is comeager in the space of all tailored almost complex structures on $(Y,\Lambda)$.
\end{proposition}

The proof follows the analogous arguments in the case of $J$-holomorphic curves without boundary in  symplectizations (cf. \cite[\S 7,8]{wendl2016lectures} and more specifically \cite[Thm 7.2]{wendl2016lectures}). 

\begin{proposition} \label{prop:regular_implies_manifold} If $J$ is a regular, tailored almost complex structure on $(Y,\Lambda)$, then the set of simple, finite energy, $J$-holomorphic curves with boundary is a smooth manifold and the dimension at a curve $u$ is given by $\op{ind}(u)$. 
\end{proposition}

This also follows from analogues of the arguments in \cite[\S 7,8]{wendl2016lectures} for somewhere injective curves in symplectizations, and the fact that simple implies somewhere injective in our setting (Lemma \ref{lem:simple_vs_injective}).

\begin{remark} In general, for a curve $u$ with Lagrangian boundary condition, there is no guarantee that $u$ can be factored into a branched cover and a somewhere injective (and therefore regular) curve. It is key that we can use Lemma \ref{lem:simple_vs_injective} to replace somewhere injective with simple in the statement of Proposition \ref{prop:regular_implies_manifold}. It implies that the underlying $J$-holomorphic curves in any $J$-holomorphic current are transversely cut out when $J$ is regular. \end{remark}

\section{Legendrian ECH Index And Its Properties}
\label{section: index_inequality} In this section, we introduce the Legendrian ECH index and establish its basic properties. 

\subsection{ECH Conley-Zehnder Term} We begin by introducing the Conley-Zehnder term that appears in the Legendrian ECH index. 

\begin{definition} \label{def:LECH_CZ_term} The \emph{ECH Conley-Zehnder term} of an orbit-chord set $\Xi$ with respect to a trivialization $\tau$ along $\Xi$ is the unique half-integer
\[\CZ^{ECH}_\tau(\Xi) \in \frac{1}{2}\Z\]
that is additive with respect to disjoint union, in that
\[\CZ^{ECH}_\tau(\Xi \sqcup \Theta) = \CZ^{ECH}_\tau(\Xi) + \CZ^{ECH}_\tau(\Theta)\]
and that is given by the following formula for a simple Reeb orbit $\gamma$ and Reeb chord $c$
\[\CZ^{ECH}_\tau((\gamma,m)) = \sum_{i=1}^m \CZ_\tau(\gamma^i) \quad\text{and}\quad \CZ^{ECH}_\tau((c,m)) = \frac{m}{2} + \frac{m(m + 1)}{2} \cdot (\CZ_\tau(c) - \frac{1}{2})\]
The ECH Conley-Zehnder term of a surface class $A \in S(\Xi_+,\Xi_-)$ is defined to be
\[
\CZ^{ECH}_\tau(A) := \CZ^{ECH}_\tau(\Xi_+) - \CZ^{ECH}_\tau(\Xi_-)
\]
We adopt the same definition for $J$-holomorphic curves and currents. \end{definition}
 
We will make heavy use of the following change of trivialization rule.

\begin{lemma}[Trivialization] \label{lem:ECH_CZ_change_of_triv}  Let $\sigma,\tau$ be two trivializations of $\xi$ along an  simple orbit or chord $\eta$. Then 
\begin{equation} 
\label{eqn:ECH_CZ_change_of_triv} \CZ^{ECH}_\tau((\eta,m)) - \CZ^{ECH}_\sigma((\eta,m)) = \frac{m(m+1)}{2} \cdot (\sigma - \tau) \in \frac{1}{2}\Z\end{equation}
\end{lemma}

\begin{proof} Recall that the Conley-Zehnder index itself transforms as
\[\CZ_\tau(\eta) - \CZ_\sigma(\eta) = 2(\sigma - \tau) \in \Z \qquad\text{by Lemma \ref{lem:ECH_CZ_change_of_triv}}\]
Note that this formula holds in the Reeb orbit and chord case (and in the former case, the difference $\sigma - \tau$ is a whole integer). In the case of a Reeb orbit, we thus have
\[\CZ^{ECH}_\tau((\eta,m)) - \CZ^{ECH}_\sigma((\eta,m)) = \sum_{i=1}^m i \cdot (\sigma - \tau) = \frac{m(m+1)}{2} \cdot (\sigma - \tau)\]
In the case of a chord, the difference formula follows directly from the definition of $\CZ^{ECH}_\tau$. \end{proof}

\subsection{Partition Conditions} \label{subsec:partition_conditions} We next briefly review the ECH partition conditions. Let $P(m)$ denote the set of partitions of an integer $m \in \N$.

\begin{definition} The \emph{positive partition} $p^+_\gamma(m)$ and \emph{negative partition} $p^-_\gamma(m)$ associated to a non-degenerate Reeb orbit $\gamma$ in $(Y,\alpha)$ are partitions in $P(m)$ defined as follows. 

\vspace{3pt}

Fix a trivialization of $\xi$ along $\gamma$ and let $L$ be the period of $\gamma$. Consider the linearized flow
\[
\phi_\tau = \tau_{\gamma(0)} \circ d\phi_L(\gamma(0))|_\xi \circ \tau_{\gamma(0)}^{-1} \in \Sp(2)
\]
\begin{itemize}
    \item If $\gamma$ is \emph{positive hyperbolic} (i.e. $\phi_\tau$ has positive real eigenvalues) then
    \[
    p^+_\gamma(m) = p^-_\gamma(m) = (1,1,\dots,1)
    \]
    \item If $\gamma$ is \emph{negative hyperbolic} (i.e. $\phi_\tau$ has negative real eigenvalues) then
    \[
    p^+_\gamma(m) = p^-_\gamma(m) = 
    \left\{\begin{array}{cc}
    (2,2,\dots,2) & \text{if $m$ is even}\\
    (2,2,\dots,2,1) & \text{if $m$ is odd}
    \end{array}\right.
    \]
    \item If $\gamma$ is \emph{elliptic} (i.e. $\phi_\tau$ has unit complex eigenvalues) then $\phi_\tau$ is conjugate to a rotation by angle $2\pi \theta$ for $\theta \in \R/\Z$. Then
    \[p^+_\gamma(m) = p^+_\theta(m) \qquad\text{and}\qquad p^-_\gamma(m) = p^-_\theta(m)\]
\end{itemize}

Here $p^\pm_\theta(m)$ are partitions defined using only the rotation angle $\theta$ as follows. Let $\Lambda^+_\theta(m) \subset \R^2$ be the maximal concave polygonal graph with vertices at lattice points in $\Z^2$ that starts on $(0,0)$, ends on $(m,\lfloor m\theta\rfloor)$ and lies below the line $y = \theta x$ where $\theta \in (0,1)$. Then $p^+_\theta(m)$ is the sequence of horizontal displacements of the consecutive vertices of $\Lambda^+_\theta(m)$. We can define $p^-_\theta(m) = p^+_{-\theta}(m)$.
\end{definition}

\begin{definition} A somewhere injective $J$-holomorphic curve $C$ in $(Y,\Lambda)$ from orbit set $\Xi_+$ to orbit set $\Xi_-$ \emph{satisfies the ECH partition conditions} if
\begin{itemize}
    \item For each orbit $\gamma$ of multiplicity $m \ge 1$ in $\Gamma_+$, the partition of $\gamma$ determined by the positive end of $C$ at $\gamma$ is $p_+(m,\gamma)$.
    \item For each orbit $\gamma$ of multiplicity $m \ge 1$ in $\Gamma_-$, the partition of $\gamma$ determined by the negative end of $C$ at $\gamma$ is $p_-(m,\gamma)$.
\end{itemize}
\end{definition}

Given a simple,  elliptic Reeb orbit $\gamma$ of $Y$ and two partitions $p,q \in P(m)$, we write
\[
p \prec_\gamma q
\]
if there is a degree $m$ $J$-holomorphic branched cover
\[
u:\Sigma \to \R \times \gamma \subset \R \times Y 
\]
with Fredholm index $0$ and whose ends yield the partition $q$ on the positive end and $p$ at the negative end. 

\begin{lemma} \cite[Exercise 3.1 and A.1.4]{Hutchings_lectures_on_ECH} \label{lem:cylinder_partial_order} The relation $\prec_\gamma$ is a partial order on $P(m)$. Moreover, $p^+_\gamma(m)$ and $p^-_\gamma(m)$ are respectively minimal and maximal with respect to this partial order.
\end{lemma} 

A helpful consequence of Lemma \ref{lem:cylinder_partial_order} is the following lemma controlling the triviality of trivial cylinders over closed Reeb orbits.

\begin{lemma} \label{lem:trivial_cylinder_covers} Let $\gamma$ be a simple closed Reeb orbit in $Y$. Fix a sequence of partitions $p_1,\dots,p_n$ of $m$ and let $Z_i$ be a branched cover of $\R \times \gamma$ with positive partition $p_i$ and negative partition $p_{i+1}$. Assume that
\begin{equation} 
p_1 = p_n = p^+_\gamma(m) \qquad\text{or}\qquad p_1 = p_n = p^-_\gamma(m)
\end{equation}
Then $Z_i$ is a trivial cover and $p_i = p^\pm_\gamma(m)$ for each $i$. \end{lemma}

\begin{proof} We may view ${\bf Z} = (Z_1,\dots,Z_n)$ as an SFT building (cf. \cite{abbasSFTcompactness}) with Fredholm index
\[\op{ind}({\bf Z}) = \sum_i \op{ind}(Z_i) = 0\]
By \cite[Lem. 1.7(a)]{obs1}, we know that $\op{ind}(Z_i) \ge 0$ for each $i$. Therefore $\op{ind}(Z_i) = 0$ for each $i$. Moreover, \cite[Lem. 1.7(b)]{obs1} states that the cover is unbranched if $\gamma$ is hyperbolic. If $\gamma$ is elliptic, then
\[p^\pm_\gamma(m) \prec_\gamma p_i \prec_\gamma p^\pm_\gamma(m)\]
for each $i$. It follows from the fact that $\prec_\gamma$ is a partial order that $p_i = p_\gamma^\pm(m)$ for each $i$. The Fredholm index is then given by
\[\op{ind}(Z_i) = -\chi(\Sigma_i) = 0\]
where $\Sigma_i$ is the domain of the cover $Z_i$. Riemann-Hurwitz then implies that $Z_i$ is unbranched. \end{proof}

There is also an analogous lemma for chords, which we include here for completeness.

\begin{lemma} \label{lem:trivial_strip_covers} Let $c$ be a Reeb chord in $(Y,\Lambda)$ and let $Z$ be an $m$-fold branched cover of $\R \times c$. Then $\op{ind}(Z) \ge 0$ with equality if and only if $Z$ is unbranched.
\end{lemma}

\begin{proof} Let $\Sigma$ be the domain of $Z$. We may double $\R \times c$ to a cylinder $\R \times S^1$ and double $\Sigma$ to a surface $S$. Then the curve $Z$ is an equivalence class of branched cover
\[u:S \to \R \times S^1\]
By choosing a trivialization of $\xi$ along $c$ and extending it along $Z$, we find that
\[\op{ind}(Z) = -\bar{\chi}(\Sigma) = \frac{-\chi(S)}{2}\]
Note that $S$ has at least $2$ punctures, and thus has at non-negative Euler characteristic. Thus $\op{ind}(Z) = 0$. If $\op{ind}(Z) = -\frac{\chi(S)}{2} = 0$, then Riemann-Hurwitz implies that the cover $u$ is unbranched. This implies that $Z$ is unbranched. \end{proof}

\subsection{Writhe And Linking Bounds}\label{sec:writhe and linking bounds} We are now ready to prove the Legendrian version of the writhe bound, and a separate linking bound that will be used to prove the ECH index inequality.

\begin{lemma}[Writhe-Linking Bound At Chord] \label{lemma:writhe_bound_chord} Let $C$ be a somewhere injective, $J$-holomorphic curve with respect to a tailored $J$ on $(Y,\Lambda)$, asymptotic to the orbit-chord sets $\Xi_\pm$ at $\pm \infty$. Then
\begin{itemize}
\item For any chord $\eta$ in the orbit-chord set $\Xi_+$, the asymptotic braid $\zeta$ of $C$ in $\op{Nbhd}(\eta)$ satisfies
\[w_\sigma(\zeta) \le 0 \quad\text{in the unique trivialization $\sigma$ with }\CZ_\sigma(\eta) = \frac{1}{2} \]
and the component braids $\zeta_i$ of $\zeta$ satisfy
\[
l_\sigma(\zeta_i,\zeta_j) \le 0 \qquad\text{and}\qquad \op{wind}_\sigma(\zeta_i) \le 0
\]
\item For any chord $\eta$ in the orbit-chord set $\Xi_-$, the asymptotic braid $\zeta$ of $C$ in $\op{Nbhd}(\eta)$ satisfies
\[w_\sigma(\zeta) \ge 0 \quad\text{in the unique trivialization $\sigma$ with }\CZ_\sigma(\eta) = -\frac{1}{2}\]
and the component braids $\zeta_i$ of $\zeta$ satisfy
\[
l_\sigma(\zeta_i,\zeta_j) \ge 0 \qquad\text{and}\qquad \op{wind}_\sigma(\zeta_i) \ge 0
\]
\end{itemize}
Moreover, equality occurs if and only if $\zeta$ is the trivial braid in the trivialization $\sigma$.
\end{lemma}

\begin{proof} We prove the result for a chord in $\Xi_+$. The proof is entirely analogous in the other case.

\vspace{3pt}

Consider a strip-like end of the holomorphic curve $C$ which is positively asymptotic to the chord $\eta$. For $s_0 \gg 0$ sufficiently large, the intersection of $C$ with the half infinite cylinder $[s_0,\infty)\times Y$ is given by the graph of a function
\[ [s_0,\infty)\times [0,1]\rightarrow [s_0,\infty)\times Y \quad\text{given by}\quad (s,t)\mapsto (s,\exp_{\eta(t)}U(s,t)) \]
Here $\exp$ denotes the exponential map of an auxiliary Riemannian metric on $Y$ with the property that the Legendrian $\Lambda$ consists of closed geodesics. Moreover, $U$ is a family of sections
\[U:([s_0,\infty) \times [0,1], [0,\infty) \times \{0,1\}) \to \eta^*(\xi,T\Lambda)\]
of the bundle pair $\eta^*(\xi,T\Lambda)$. The following result of Abbas states that $U$ decays exponentially as $s \to \infty$.

\begin{theorem}(\cite{abbas1999finite}) \label{thm:chord_end_decay}
There exists an eigenvector $e$ of the asymptotic operator of $\eta$ with negative eigenvalue $\lambda$ and a family $r(s,t)$ of sections of $\eta^*\xi$ exponentially decaying as $s$ tends to $\infty$ such that
\[ U(s,t) = e^{\lambda s}(e(t)+r(s,t)) \]
\end{theorem}

We require two ingredients from other parts of this paper for this proof. First, we have the following refinement of Theorem \ref{thm:chord_end_decay}. Consider two strip-like ends of $C$ positively asymptotic to $\eta$. Let $U$ and $V$ denote the associated families of sections of $\eta^*\xi$.

\begin{theorem}
\label{thm:relative_asymptotic_behaviour}
Assume that $U$ and $V$ do not agree identically. Then there exist an eigenvector $e$ of the asymptotic operator of $\eta$ with negative eigenvalue $\lambda$ and a family $r(s,t)$ of sections of $\eta^*\xi$ exponentially decaying as $s$ tends to $\infty$ such that
\[U(s,t)-V(s,t) = e^{\lambda s} (e(t) + r(s,t))\]
\end{theorem}

\noindent The proof of Theorem \ref{thm:relative_asymptotic_behaviour} requires a rather long analytical digression and is deferred to the Appendix. Second, we require the following result on the winding number of $e$ in Theorem \ref{thm:relative_asymptotic_behaviour}.

\begin{lemma}
\label{lem:winding_eigenvectors_asymptotic_operator}
Let $\sigma$ be a representative of the unique homotopy class of trivializations of $\eta^*\xi$ such that
\[\CZ_{\sigma}(\eta)=\frac{1}{2}\]
Then the winding number of $e$ with respect to $\sigma$ is strictly positive if $\lambda>0$ and non-positive if $\lambda < 0$.
\end{lemma}
\noindent This is an immediate corollary of Lemma \ref{lem:winding monotonic} and Lemma \ref{lem:index and winding of the asymptotic operator} in Section \ref{sec:CZ_and_Fredholm_indices}. 

\vspace{3pt}

We now proceed with the proof of Lemma \ref{lemma:writhe_bound_chord}. Let $\zeta_i$ denote the components of the braid $\zeta$. By Theorem \ref{thm:relative_asymptotic_behaviour} and Lemma \ref{lem:winding_eigenvectors_asymptotic_operator}, the winding numbers and pairwise linking numbers satisfy
\[\op{wind}_\sigma(\zeta_i) \le 0 \qquad\text{and}\qquad l_\sigma(\zeta_i,\zeta_j) \le 0 \qquad\text{for all }i \neq j\]
Hence it follows from \eqref{eq:relationship_writhe_linking} that the writhe $w_{\sigma}(\zeta)$ is non-positive, proving the inequality.

\vspace{3pt}

Finally, we claim that the braid $\zeta$ must be trivial with respect to $\sigma$ if $w_\sigma(\zeta)=0$. The asymptotic operator of $\eta$ has a $1$-dimensional eigenspace whose elements have zero winding number with respect to $\sigma$. Let $e$ be a non-zero vector in this eigenspace, $\lambda<0$ be the associated eigenvalue and  $U_i(s,t)$ be the section of $\eta^*\xi$ associated to the braid $\zeta_i$. Since $l_\sigma(\zeta_i,\zeta_j)$ is non-positive for all $i\neq j$ and the writhe $w_\sigma(\zeta)$ vanishes by assumption, we can deduce that $l_\sigma(\zeta_i,\zeta_j)=0$ for all $i\neq j$. Thus, for every tuple $i\neq j$, there exist a non-zero real number $a_{ij}$ and an exponentially decaying family of sections $r_{ij}(s,t)$ such that
\begin{equation}
\label{eqn:diffence_between_branches}
U_i(s,t)-U_j(s,t) = e^{\lambda s}(a_{ij}e(t)+r_{ij}(s,t))
\end{equation}
We may replace the trivialization $\sigma$ by a homotopic one with the property that $e(t)$ is constant with respect to the new trivialization. Then it follows from \eqref{eqn:diffence_between_branches} that the braid $\zeta$ is trivial.
\end{proof}

We can now prove Proposition \ref{prop:LECH_writhe_inequality}. It largely follows from the local version for chords. 

\begin{proposition}[Writhe Bound] \label{prop:LECH_writhe_inequality} Let $C$ be a somewhere injective $J$-holomorphic curve in $(Y,\Lambda)$ asymptotic to $\Gamma_\pm = (\gamma^\pm_i)$ at $\pm\infty$. Then
\[w_\tau(C)\le \CZ^{ECH}_\tau(C) - \CZ_\tau(\Gamma_+) + \CZ_\tau(\Gamma_-)\]
Moreover, there is equality only if  following conditions are satisfied.
\begin{itemize}
\item The orbit parts of $\Gamma_+$ and $\Gamma_-$ satisfy the ECH partition conditions.
\item The chords in $\Gamma_-$ are multiplicity one. \end{itemize}
\end{proposition}

\begin{proof} The writhe bound is additive under disjoint union of simple orbits and chords. Therefore, it suffices to consider each orbit and chord in $\Xi_\pm$ independently. For Reeb orbits, the result is proven in \cite[\S 5.1]{Hutchings_lectures_on_ECH}. For Reeb chords, we may write
\[\CZ^{ECH}_\tau(c,m) - m \cdot \CZ_\tau(\eta) = \frac{m(m-1)}{2} \cdot (\CZ_\tau(c) - \frac{1}{2})\]
Thus, it is sufficicient to prove to the following pair of statements.

\begin{itemize}
    \item If $\eta$ is a Reeb chord that appears in $\Xi_+$ with multiplicity $m$ and $\zeta$ is the corresponding asymptotic braid of $C$, then \begin{equation}
    w_\tau(\zeta)\le  \frac{m(m-1)}{2} \cdot (\CZ_\tau(\eta) - \frac{1}{2})\end{equation}
    \item If $\eta$ is a Reeb chord that appears in $\Xi_-$ with multiplicity $n$ and $\zeta$ is the corresponding asymptotic braid of $C$, then \begin{equation}
    w_\tau(\zeta)\ge \frac{m(m-1)}{2} \cdot (\CZ_\tau(\eta) - \frac{1}{2}),\end{equation}
    Moreover, equality holds if only if $m=1$.
\end{itemize}
Also, observe that the above inequalities are equivalent for different choices of the trivialization $\tau$ on $c$. If $\sigma$ and $\tau$ are two trivializations of $\xi$ along $c$ with difference $\tau - \sigma \in \frac{1}{2}\Z$, then  \begin{equation} \label{eq:w_difference}
w_\tau(\zeta)-w_\sigma(\zeta)= m(m-1) \cdot (\sigma - \tau) \qquad\text{by Lemma \ref{lem:writhe_link_properties}} \end{equation}
\begin{equation} CZ_\tau(\eta)-CZ_\sigma(\eta)= 2 \cdot (\sigma - \tau))  \qquad\text{by Lemma \ref{lem:ECH_CZ_change_of_triv}} 
\end{equation} So it suffices to prove the inequalities for any choice of trivialization. 

For the first claim, we choose the trivialization $\tau$ so that $\CZ_\tau(\eta)=\frac{1}{2}$. Then Proposition \ref{lemma:writhe_bound_chord} implies that
\begin{equation}
    w_\tau(\zeta)\le 0= \frac{m(m-1)}{2} \cdot (\CZ_\tau(\eta)-\frac{1}{2}).
\end{equation} For the second claim, we choose $\tau$ so that $\CZ_\tau(\eta)=-\frac{1}{2}$. Proposition \ref{lemma:writhe_bound_chord} implies that
\begin{equation}
    w_\tau(\zeta)\ge 0\ge -\frac{m(m-1)}{2} = \frac{m(m-1)}{2} \cdot ( \CZ_\tau(\eta) - \frac{1}{2})
\end{equation}
with equality holds only if the multiplicity $n=1$. \end{proof}

\begin{remark}[Writhe Asymmetry] The observant reader will notice that there is an asymmetry in the chord condition required for the writhe bound to yield an equality. This is an artifact of our convention for the Legendrian ECH index. We could have alternatively defined it as
 \[-\frac{m}{2} + \frac{m(m+1)}{2} \cdot (\CZ_\tau(c) + \frac{1}{2})\]
Then the partition conditions imposed on the positive chords and negative chords would have been reversed. An asymmetry of this type is more or less unavoidable. 
\end{remark}

There is also an analogous statement to the writhe inequality about the linking number of holomorphic currents. 

\begin{proposition} \label{prop:linking_bound_currents} Let $\mathcal{C}$ and $\mathcal{D}$ be a pair of $J$-holomorphic currents with boundary that have disjoint components. Then
\[\CZ^{ECH}_\tau(\mathcal{C} \cup \mathcal{D}) \ge \CZ^{ECH}_\tau(\mathcal{C}) + \CZ^{ECH}_\tau(\mathcal{D}) + 2 \cdot l_\tau(\mathcal{C},\mathcal{D})\]
\end{proposition}

\begin{proof} Let $(C_i,m_i)$ and $(D_j,n_j)$ denote the components of $\mathcal{C}$ and $\mathcal{D}$. Fix an orbit or chord $\eta$ and let
\[\zeta^\pm_i \qquad\text{and}\qquad \xi^\pm_j\]
denote the braids of $C_i$ and $D_j$ asymptotic to $\eta$ at the positive and negative ends respectively. We fix the shorthand notation
\[
l_\tau(\zeta^+,\xi^+) = \sum_{i,j} m_i \cdot n_j \cdot l_\tau(\zeta^+_i,\xi^+_j) \qquad \text{and}\qquad l_\tau(\zeta^-,\xi^-) = \sum_{i,j} m_i \cdot n_j \cdot l_\tau(\zeta^-_i,\xi^-_j)
\]
Finally, let $m_\pm$ and $n_\pm$ denote the multiplicity of $\eta$ at the positive and negative ends of $\mathcal{C}$ and $\mathcal{D}$, respectively. It suffices to prove that
\begin{equation} \label{eqn:linking_bound_currents_p} \CZ_\tau^{ECH}((\eta,m_++n_+)) - \CZ_\tau^{ECH}((\eta,m_+)) - \CZ_\tau^{ECH}((\eta,n_+)) - 2 \cdot l_\tau(\zeta^+,\xi^+) \ge 0 \end{equation}
\begin{equation} \label{eqn:linking_bound_currents_n} \CZ_\tau^{ECH}((\eta,m_-+n_-)) - \CZ_\tau^{ECH}((\eta,m_-)) - \CZ_\tau^{ECH}((\eta,n_-)) - 2 \cdot l_\tau(\zeta^-,\xi^-) \le 0 \end{equation}
for any choice of $\eta$. This is proven by Hutchings \cite[ Eq. 5.6 on p. 47]{Hutchings_index_revisited} if $\eta$ is a Reeb orbit. Thus, we assume that $\eta$ is a Reeb chord. Finally, as in the proof of Proposition \ref{prop:LECH_writhe_inequality}, Lemma \ref{lem:current_lk_trivialization} and Lemma \ref{lem:ECH_CZ_change_of_triv} imply that the claim is independent of trivialization. 

\vspace{3pt}

Next choose the trivialization $\tau$ so that $\CZ_\tau(\eta) = \frac{1}{2}$. Then the ECH Conley-Zehnder term is
\[\CZ_\tau^{ECH}((\eta,k)) = 0 \qquad\text{for any integer }k\]
Moreover, by Lemma \ref{lemma:writhe_bound_chord}, the linking numbers of $\zeta^+_i$ and $\xi^+_j$ satisfies
\[l_\tau(\zeta^+_i,\xi^+_j) \le 0 \qquad\text{for all }i,j\]
Note that the linking number is the winding number, by convention, if $C_i$ or $D_j$ is a trivial cylinder over $\eta$. These two inequalities imply (\ref{eqn:linking_bound_currents_p}). Finally, choose the trivialization $\sigma$ so that $\CZ_\tau(\eta) = -\frac{1}{2}$. Then
\[\CZ_\tau^{ECH}((\eta,k)) = -\frac{k(k-1)}{2} \qquad\text{for any integer }k\]
In particular, the ECH Conley-Zehnder term in (\ref{eqn:linking_bound_currents_n}) satisfies
\[\CZ_\tau^{ECH}((\eta,m_- + n_-)) - \CZ_\tau^{ECH}((\eta,m_-)) - \CZ_\tau^{ECH}((\eta,n_-)) < 0\]
Moreover, by Lemma \ref{lemma:writhe_bound_chord}, the linking numbers of $\zeta^-_i$ and $\xi^-_j$ satisfies
\[l_\tau(\zeta^-_i,\xi^-_j) \ge 0 \qquad\text{for all }i,j\]
This proves (\ref{eqn:linking_bound_currents_n}) and completes the proof. \end{proof}

\subsection{ECH Index And Basic Properties} We are now ready to introduce the Legendrian ECH index in detail, and prove its basic properties.

\begin{definition}
\label{def:ech_index} The \emph{Legendrian ECH index} $I(A)$ of a surface class $A \in S(\Theta,\Xi)$ between orbit-chord sets $\Theta$ and $\Xi$ is given by the following formula.
\begin{equation} \label{eqn:main_lECH_index}
    I(A) := Q_\tau(A) + \frac{1}{2} \mu_\tau(A) +\CZ^{ECH}_\tau(\Theta) - \CZ^{ECH}_\tau(\Xi)
\end{equation}
The terms in the formula above are as follows.
\begin{itemize}
    \item $\tau$ is any trivialization of the bundle pair $(\xi,T\Lambda)$ over $\Theta$ and $\Xi$ (see Definition \ref{def:trivializations_over_orbit_chord_set})
    \item $\mu_\tau$ is the (relative) Maslov number (see Definition \ref{def:Maslov_number_of_surface_class})
    \item $Q_\tau$ is the relative self intersection number with respect to $\tau$ (see Definition \ref{def:relative_intersection})
    \item $\CZ_\tau^{ECH}$ is the\ ECH Conley-Zehnder term (see Definition \ref{def:LECH_CZ_term}). 
\end{itemize}
The Legendrian ECH index of a $J$-holomorphic current $\mathcal{C} \in \mathcal{M}(\Theta,\Xi)$ is simply the index of its corresponding surface class $A \in S(\Theta,\Xi)$.
\[I(\mathcal{C}) := I(A)\]\end{definition}

\begin{lemma} The Legendrian ECH index $I(A)$ of a surface class $A:\Theta \to \Xi$ is an integer that is independent of the trivialization in (\ref{eqn:main_lECH_index}). 
\end{lemma}

\begin{proof} Since $I(A)$ is a sum of half-integer valued terms, it is a half integer. To show that $I(A)$ is an integer, fix a well-immersed surface $S$ representing $A$. Let $m(S)$ denote the number of open boundary components of $\partial_+ S \cup \partial_- S$, and note that
\[
\bar{\chi}(S) = \frac{1}{2} \cdot m(S) \mod 1
\]
By topological adjunction (Theorem \ref{thm:topological_adjunction}), we have
\[
\frac{1}{2}\mu_\tau(S) + Q_\tau(S) = \bar{\chi}(S) = \frac{1}{2} \cdot m(S)\quad \mod 1
\]
On the other hand, the ECH Conley-Zehnder term is a sum of the (integer) $\CZ$-indices of closed orbits and the $\CZ$-indices of chords. The Conley-Zehnder index of a chord in dimension three is automatically a strict half-integer. Since the total number of chords in $\Theta$ and $\Xi$ (with multiplicity) is $m(S)$, we find that
\[\CZ^{ECH}_\tau(\Xi) - \CZ^{ECH}_\tau(\Theta) = \frac{1}{2} \cdot m(S) \mod 1\]
Thus we find that
\[
I(A) = \frac{1}{2}\mu_\tau(S) + Q_\tau(S) + \CZ^{ECH}_\tau(\Xi) - \CZ^{ECH}_\tau(\Theta) = m(S) = 0 \mod 1
\]

\vspace{3pt}

To see that $I(A)$ is independent of trivialization, fix two trivializations $\sigma$ and $\tau$ of $\xi$ along $\Xi_\pm$. We write the orbit sets $\Theta$ and $\Xi$ as follows.
\[
\Xi = \{(\gamma_i,m_i)\} \qquad\text{and}\qquad \Theta = \{(\eta_j,n_j) \}
\]
Fix an orbit or chord $\eta \subset \Xi$ with multiplicity $m$. Let $\sigma$ and $\tau$ be two trivializations of $\xi$ along $\Xi$ and $\Theta$ that agree everywhere except along $\eta$. We compute that
\[Q_\tau(A) - Q_\sigma(A) = m^2 \cdot (\tau - \sigma) \qquad\text{by Lemma \ref{prop:properties_of_relative_intersection}}\]
\[\mu_\tau(A) - \mu_\sigma(B) = \sigma - \tau \qquad\text{by Proposition \ref{prop:properties_of_Maslov_number}}\]

\vspace{-10pt}

\[\CZ^{ECH}_\tau(\Xi) - \CZ^{ECH}_\sigma(\Xi) = m(m-1) \cdot (\sigma - \tau) \qquad\text{by Lemma \ref{lem:ECH_CZ_change_of_triv}}\]
The sum of these terms is zero, thus yielding the desired invariance property. An identical proof works if $\tau$ and $\sigma$ only differ along an orbit or chord of $\Theta$. Any two trivializations are related by a sequence of changes supported on one orbit or chord, so this proves the result.
\end{proof}

The Legendrian ECH index has the following basic properties, generalizing the corresponding properties of the ordinary ECH index (cf. \cite[Proposition 1.6]{Hutchings2002}).

\begin{proposition}[Basic Properties] \label{prop:basic_prop_ECH} The Legendrian ECH index satisfies the following axioms.
\begin{itemize}
\item (Index Ambiguity) If $A,B:\Xi_+ \to \Xi_-$ are two classes with difference $B - A \in H_2(Y,\Lambda)$, and let $\Gamma \in H_1(Y,\Lambda)$ denote the homology class of $\Xi_\pm$. Then
\[I(B) - I(A) = \frac{1}{2} \langle \mu(\xi,\Lambda), B - A\rangle + 2 \cdot Q(\Gamma,B-A)\]
where $\mu(\xi,\Lambda)$ is the Maslov class of $(\xi,\Lambda)$ and $Q$ is the intersection pairing of $(Y,\Lambda)$. 
\item (Composition) If $A:\Xi_0 \to \Xi_1$ and $B:\Xi_1 \to \Xi_2$ are composible surface classe then
\[
I(A \circ B) = I(A) + I(B)
\]
 
\end{itemize}
\end{proposition}

\begin{proof} We prove each of these properties separately.

\vspace{3pt}

{\bf Index Ambiguity.} Let $A,B:\Theta \to \Xi$ be two surface classes with the same ends and fix a trivialization $\tau$ over $\Theta$ and $\Xi$. We compute the difference of each term in $I(A)$ and $I(B)$. Starting with the Maslov number, we have
\[\mu_\tau(B) - \mu_\tau(A) = \mu(\xi,\Lambda) \cdot (B - A) \qquad\text{by Proposition \ref{prop:properties_of_Maslov_number}}\]
To compute the difference between the self-intersection numbers, we note that
\[
Q_\tau(B) - Q_\tau(A) = Q_\tau(B,B) - Q_\tau(A,B) + Q_\tau(B,A) - Q_\tau(A,A) = q_B(B - A) + q_A(B - A)\]
Here $q_A$ and $q_B$ are the homomorphisms $H_2(Y,\Lambda) \to \frac{1}{2}\Z$ in Definition \ref{def:intersection_map}. By Lemma \ref{lem:intersection_3fld_vs_cobordism}
\[q_A(B - A) + q_B(B - A) = 2 \cdot Q(\Gamma,B - A)\]
Finally, since $A$ and $B$ have the same ends, the ECH Conley-Zehnder terms coincide. This proves the desired formula.

\vspace{3pt}

{\bf Composition.} This follows from the corresponding composition property for $Q_\tau$ and $\mu_\tau$. \end{proof}

\subsection{Index Inequality And Union Property} The Legendrian ECH index has a number of properties that hold for currents with boundary due to the writhe and linking inequalities. We are now ready to demonstrate these properties.

\vspace{3pt}

We begin with the the Legendrian ECH index inequality as stated in the introduction.

\begin{theorem}[Theorem \ref{thm:legendrian_index_inequality}] \label{thm:ECH_index_inequality}
Let $C$ be a somewhere injective, $J$-holomorphic curve with boundary in $(\R \times Y,\R \times \Lambda)$ for tailored $J$. Then
\begin{equation}
    \ind(C)\leq I(C)-2\delta(C)-\epsilon(C).
\end{equation}
\end{theorem}
\begin{proof} Let $C$ be a somewhere injective $J$-holomorphic curve asymptotic to Reeb chords and orbits $\Gamma_\pm$ at $\pm \infty$. Then the difference between ECH index and Fredholm index is
\[I(C) - \ind(C) = Q_\tau(C) - \frac{1}{2} \mu_\tau(C) + \CZ^{ECH}_\tau(\Gamma_+) - \CZ^{ECH}_\tau(\Gamma_-) +\bar{\chi}(C) - \CZ_\tau(\Gamma_+) + \CZ_\tau(\Gamma_-)\]
Then by Legendrian adjunction, Theorem \ref{cor:legendrian_adjunction_body}, the above formula becomes
\[Q_\tau(C) +  \CZ^{ECH}_\tau(\Gamma_+) - \CZ^{ECH}_\tau(\Gamma_-) - \CZ_\tau(\Gamma_+) + \CZ_\tau(\Gamma_-) -Q_\tau(C)-w_\tau(C) +2\delta(C) +\epsilon(C)\]
Here the writhe $w_\tau(C)$ is the sum of the writhes of the positive asymptotic braid $\zeta_+$ and negative asymptotic braid $\zeta_-$ of $C$. Thus we may write the above formula as
\[= 2\delta(C) + \epsilon(C)  +(\CZ^{ECH}_\tau(\zeta_+) - \CZ_\tau(\Gamma_+) - w_\tau(\zeta_+)) - (\CZ^{ECH}_\tau(\zeta_-) - \CZ_\tau(\Gamma_-) - w_\tau(\zeta_-))\]
The writhe bound, Proposition \ref{prop:LECH_writhe_inequality}, implies that the middle and right terms are non-negative. Thus we have proven that
\[I(C) - \ind(C) \ge 2\delta(C) + \epsilon(C) \qedhere\]\end{proof}

Next, we prove a fundamental sub-additivity property of the ECH index under union. This generalizes \cite[Thm. 5.1]{Hutchings_index_revisited} to currents with boundary.

\begin{theorem} [Union]
\label{thm:disjoint_union_index} Let $\mathcal{C}$ and $\mathcal{D}$ be $J$-holomorphic currents for a tailored $J$ on $(Y,\Lambda)$, with distinct components. Then the ECH index satisfies
\[I(\mathcal{C} \cup \mathcal{D}) \ge I(\mathcal{C}) + I(\mathcal{D}) + 2 (\mathcal{C} \cdot \mathcal{D})\]
\end{theorem}

\begin{proof} We compute that the difference of ECH indices is given by
\[
I(\mathcal{C} \cup \mathcal{D}) - I(\mathcal{C}) - I(\mathcal{D}) = 2 \cdot Q_\tau(\mathcal{C},\mathcal{D}) + \CZ^{ECH}_\tau(\mathcal{C} \cup \mathcal{D}) - \CZ^{ECH}_\tau(\mathcal{C}) - \CZ^{ECH}_\tau(\mathcal{D})
\]
By Lemma \ref{lem:intersection_vs_lk_currents}, this can be rewritten using the geometric intersection number as
\[
2(\mathcal{C} \cdot \mathcal{D}) + \big(\CZ^{ECH}_\tau(\mathcal{C} \cup \mathcal{D})  - \CZ^{ECH}_\tau(\mathcal{C}) - \CZ^{ECH}_\tau(\mathcal{D}) - 2 \cdot l_\tau(\mathcal{C},\mathcal{D})\big)
\]
The second term above is non-negative by the linking bound for currents (Proposition \ref{prop:linking_bound_currents}). \end{proof}

\section{Legendrian ECH}\label{sec:lech} We are now ready to construct Legendrian embedded contact homology in detail. Specifically, we prove the two key technical results Theorem \ref{thm:regular_LECH_index} and Theorem \ref{thm:intro_Legendrian_ECH_compactness} from the introduction. 

\begin{setup} To proceed with our construction, we fix the following setup for the rest of this section. 
\begin{itemize}
    \item[(a)] $(Y,\xi)$ is a contact $3$-manifold with convex sutured boundary $\partial Y$.
    \item[(b)] $\Lambda = \Lambda_+ \cup \Lambda_-$ is a union of exact Legendrians $\Lambda_\pm \subset \partial_\pm Y$.
    \item[(c)] $\alpha$ is a non-degenerate, adapted contact form on $Y$.
    \item[(d)] $J$ is a tailored complex structure on $\xi$
\end{itemize}\end{setup}

\noindent We start by recalling the definition of an ECH generator.

\begin{definition} An \emph{ECH generator} of $(Y,\Lambda)$ is an orbit-chord set $\Theta = \{(\gamma_i,m_i)\} \cup \{(c_i,n_i)\}$ where
\begin{itemize}
    \item Every hyperbolic orbit $\gamma_i$ has multiplicity $1$.
    \item Every chord $c_i$ is multiplicity $1$.
    \item There is at most one Reeb chord incident to $L$ in $\Theta$ for each connected component $L$ of $\Lambda$
\end{itemize}
\end{definition}

\subsection{Classification Of Low  ECH Index Currents} \label{subsec:classification_of_low_ECH} We begin by classifying low ECH index currents with boundary in Lemmas \ref{lem:ECH_index_0_currents} and Proposition \ref{prop:ECH_index_1_currents} and \ref{prop:ECH_index_2_currents} below. These results, together, yield Theorem \ref{thm:regular_LECH_index} in the introduction.

\begin{lemma}\label{lem:ECH_index_0_currents} Let $\mathcal{C}$ be an $J$-holomorphic current for a regular, tailored almost complex structure $J$ on $(\R \times Y,\R \times \Lambda)$. Then the following are equivalent.
\begin{itemize}
\item $\mathcal{C}$ is a trivial current $\R \times \Xi$ over an orbit-chord set $\Xi$.
\item $\mathcal{C}$ has ECH index less than or equal to zero.
\end{itemize}
\end{lemma}

\begin{proof} By the Legendrian ECH index inequality (Theorem \ref{thm:ECH_index_inequality}) and the sub-additivity of the ECH index
\[\ind(C_i) \le I(C_i) - 2 \delta(C_i) - \epsilon(C_i) \le I(\mathcal{C}) \le 0 \qquad\text{where}\qquad \mathcal{C} = \{(C_i,m_i)\}\]
For any regular almost complex structure $J$, the Fredholm index of any somewhere injective curve is non-negative, and zero if and only if the curve is a trivial cylinder or strip. Therefore
\[C_i = \R \times \gamma_i \qquad\text{for an simple orbit or chord }\gamma_i \subset Y\]
It follows that $\mathcal{C}$ is the trivial current $\R \times \Xi$ over the orbit-chord set $\Xi = \{(\gamma_i,m_i)\}$. \end{proof}

\begin{proposition} \label{prop:ECH_index_1_currents} Let $\mathcal{C}$ be a $J$-holomorphic current for a regular, tailored almost complex structure on $(\R \times Y,\R \times \Lambda)$ of ECH index one. Then
\[\mathcal{C} = C \sqcup \mathcal{T}\]
where $\mathcal{T}$ is trivial and $C$ is an embedded, connected curve with $\ind(C) = I(C) = 1$.
\end{proposition}

\begin{proof} We may write $\mathcal{C} = \mathcal{S} \cup \mathcal{T}$ where $\mathcal{T}$ is a trivial current $\R \times \Xi$ and $\mathcal{S}$ is a current with no trivial components. By Lemma \ref{lem:ECH_index_0_currents} and the sub-additivity of the ECH index (Theorem \ref{thm:disjoint_union_index}), we have
\[0 < I(\mathcal{S}) \qquad\text{and}\qquad I(\mathcal{S}) \le I(\mathcal{S}) + I(\mathcal{T}) + 2(\mathcal{S} \cdot \mathcal{T}) \le I(\mathcal{C}) \le 1\]
Thus, $\mathcal{S}$ is non-empty and $\mathcal{T}$ is disjoint from $\mathcal{S}$. We can then write
\[\mathcal{S} = \{(C_i,m_i)\} \qquad\text{for simple, non-trivial curves }C_i\text{ and multiplicities }m_i\]

\vspace{3pt}

If $(C_i,m_i)$ has multiplicity $m_i >1$, we may translate $m_i$ of copies of $C_i$ along the symplectization direction to obtain $m_i$ somewhere injective curves whose union forms a new $J$-holomorphic curve $C'$.

We now apply the ECH index inequality (Theorem \ref{thm:ECH_index_inequality}) to $C'$ and obtain:
\begin{equation}\label{equation:low index currents}
\sum_i m_i\ind(C_i) \leq \sum_i I(\mathcal{C}) - 2\delta(C') - \epsilon(C') 
\end{equation}
since the Fredholm index is additive under taking unions and ECH index is only dependent on the relative homology class. Since $J$ is regular, $\ind(C_i)\geq 0$ and in fact $\ind(C_i)>0$ by our assumption of $\mathcal{S}$ and Lemma \ref{lem:ECH_index_0_currents}. This contradicts (\ref{equation:low index currents}). Furthermore, 
\[\mathcal{S} = \{(C_1,1)\} \qquad \text{for $C_1$ an embedded, non-trivial curve}\]
since $\delta(C_1) = \epsilon(C_1) = 0$ by (\ref{equation:low index currents}) again. Therefore, $\ind(C_1) = I(C_1) = 1$.

\end{proof}

\begin{proposition}\label{prop:ECH_index_2_currents}  Let $\mathcal{C}$ be a $J$-holomorphic current for a regular, compatible almost complex structure on $(\R \times Y,\R \times \Lambda)$ of ECH index two. Also assume that the ends $\Xi_+$ and $\Xi_-$ of $\mathcal{C}$ are ECH generators. Then
\[\mathcal{C} = \mathcal{S} \sqcup \mathcal{T}\]
where $\mathcal{T}$ is trivial and $\mathcal{S}$ is a current of one of the following types.
\begin{itemize}
\item A pair of disjoint, embedded curves $C_1 \sqcup C_2$ with $I(C_i) = \op{ind}(C_i) = 1$. Note here that $C_1$ and $C_2$ are distinct in the sense they are not $\mathbb{R}$ translates of each other.
\item A single embedded curve $C$ of multiplicity $1$ with $I(C) = \op{ind}(C) = 2$.

\end{itemize}
\end{proposition}

To prove Proposition \ref{prop:ECH_index_2_currents}, the following lemma will be helpful. It severely limits the type of boundary singularities that can occur in currents connecting ECH generator.

\begin{lemma}[Singularity Censorship] \label{lem:intersection_number_interior_CD} Let $\mathcal{C}$ be a (finite energy, proper) $J$-holomorphic current with boundary between ECH generators $\Xi_+$ and $\Xi_-$. Then
\begin{itemize}
\item $\mathcal{S} \cdot \mathcal{T}$ is an integer for any pair of currents $\mathcal{S}$ and $\mathcal{T}$ with disjoint components such that $\mathcal{S} \cup \mathcal{T} \subset \mathcal{C}$.
\item Every component curve $C$ of $\mathcal{C}$ is non-singular near the boundary. That is $\epsilon(C) = 0$.
\end{itemize}\end{lemma}

\begin{proof} For the first claim, it suffices to show that any pair of component curves $C$ and $D$ satisfy
\[C \cap D \cap (\R \times \Lambda) = \emptyset\]
Thus suppose otherwise. Then there are components $\Gamma$ and $\Gamma'$ of the boundary of $C$ and $D$, respectively, such that $\Gamma \cap \Gamma'$ is non-empty and lie on $\R \times L$ where $L \subset \Lambda$ is a component. 

\vspace{3pt}

By Lemma \ref{lem:boundary_immersion}, both components $\Gamma$ and $\Gamma'$ are diffeomorphic to $\R$. Thus $C$ and $D$ both have punctures asymptotic at $+\infty$ to a Reeb chord of $L$. Thus $\Xi_+$ has either a Reeb chord of multiplicity two or two Reeb chords incident to a single component $L$ of $\Lambda$. This is a contradiction, since $\Xi_\pm$ is an ECH generator. The second claim follows immediately from Lemma \ref{lem:boundary_immersion}. \end{proof}

\begin{proof}[proof of Proposition \ref{prop:ECH_index_2_currents}] We again write $\mathcal{C} = \mathcal{S} \cup \mathcal{T}$ where $\mathcal{T}$ is a trivial current $\R \times \Xi$ and $\mathcal{S}$ is a current with no trivial components. Applying Lemma \ref{lem:ECH_index_0_currents} and Theorem \ref{thm:disjoint_union_index} now yields
\[0 < I(\mathcal{S}) \qquad\text{and}\qquad I(\mathcal{S}) \le I(\mathcal{S}) + I(\mathcal{T}) + 2(\mathcal{S} \cdot \mathcal{T}) \le I(\mathcal{C}) \le 2\]
Since $\mathcal{C}$ has ends on ECH generators, $\mathcal{S}$ and $\mathcal{T}$ are disjoint on $\R \times \Lambda$. By Lemma \ref{lem:intersection_number_interior_CD}, we must have $\mathcal{S} \cdot \mathcal{T} = 0$ or $1$. In the later case, $I(\mathcal{S}) = 0$. Thus we must have
\[
I(\mathcal{S}) = 2 \quad\text{and}\quad \mathcal{S} \cdot \mathcal{T} = \emptyset
\]
As in Proposition \ref{prop:ECH_index_1_currents}, we can assume that $\mathcal{S}$ consists of connected, somewhere injective components $C_i$ of multiplicity $1$. We apply the ECH index inequality (Theorem \ref{thm:ECH_index_inequality}) and the sub-additivity of the ECH index (Theorem \ref{thm:disjoint_union_index}) to find that
\[
\sum_i \op{ind}(C_i) \le \sum_i I(C_i) - 2\delta(C_i) - \epsilon(C_i) \qquad\text{and}\qquad \sum_i I(C_i) + 2 \sum_{i < j} C_i \cdot C_j \le I(\mathcal{S}) = 2
\]
Since each $C_i$ is non-trivial, we must have $\op{ind}(C_i) > 0$ for each $i$. Thus we infer from Lemma \ref{lem:intersection_number_interior_CD} that $C_i \cdot C_j = 0$ for each $i$ and $j$ and $\delta(C_i) = \epsilon(C_j) = 0$.

Finally we rule out the case of $\mathcal{C}$ containing a nontrivial curve $C$ with multiplicity two. By assumption on ECH generators, $C$ cannot contain chords or be asymptotic to hyperbolic Reeb orbits. Then the same proof as in regular ECH shows $I(C) - \ind(C)$ is even. Hence we cannot have nontrivial curves of multiplicity $\geq 2$.
\end{proof}

\subsection{Compactness Of Low ECH Index Moduli Spaces} Next, we prove the requisite compactness properties for the low ECH index moduli spaces required in ECH. 

\vspace{3pt}

We begin by proving a crude topological lower bound for the (corrrected) Euler characteristic of $J$-holomorphic curves between ECH generators. 

\begin{lemma}[Topological Index Bound] \label{lem:topological_upper_bound} Let $C$ be a (proper, finite energy) $J$-holomorphic curve in $(\R \times Y,\R \times \Lambda)$ between ECH generators $\Xi_+$ and $\Xi_-$ in the surface class $A$. Then
\[
-\bar{\chi}(C) \le -\frac{1}{2} \cdot \mu_\tau(A) + Q_\tau(A) + \CZ^{ECH}_\tau(\check{\Xi}_+) - \CZ^{ECH}_\tau(\check{\Xi}_-)
\]
Here $\check{\Xi}_+$ and $\check{\Xi}_-$ are the orbit-chord sets made from $\Xi_+$ and $\Xi_-$ by reducing the multiplicity of each orbit and chord by one (and removing everything of multiplicity one).\end{lemma}

\begin{proof} We apply Legendrian adjunction (Corollary \ref{cor:legendrian_adjunction_body}) to see that 
\[
-\bar{\chi}(C) = -\frac{1}{2} \cdot \mu_\tau(C) + Q_\tau(C) + w_\tau(C) - 2\delta(C) - \epsilon(C) \le -\frac{1}{2} \cdot \mu_\tau(C) + Q_\tau(C) + w_\tau(C)
\]
We may decompose the writhe of $C$ into a sum of writhes
\[
w_\tau(C) = \sum_i w_\tau(\zeta_i^+) - \sum_j w_\tau(\zeta^-_j)
\]
where the sum is over all braids over each orbit and chord appearing as a positive and negative end of $C$. The writhe is automatically $0$ for any chord and orbit of multiplicity $1$. Since $\Xi_+$ and $\Xi_-$ are ECH generators, this includes all chords and hyperbolic orbits. This reduces the desired result to the inequality
\begin{equation} \label{eq:topological_upper_bound_1} w_\tau(\zeta^+) \le \CZ^{ECH}_\tau((\eta,m-1))\end{equation}
for any simple elliptic orbit of multiplicity $m > 1$ in $\Xi_+$ and the corresponding inequality
\begin{equation} \label{eq:topological_upper_bound_2} w_\tau(\zeta^-) \ge \CZ^{ECH}_\tau((\gamma,n-1))\end{equation}
for any simple elliptic orbit of multiplicity $n > 1$ in $\Xi_-$. This is proven in \cite[Prop. 6.9]{Hutchings_index_revisited}. In particular, \cite[Eq. 6.2]{Hutchings_index_revisited} states that
\[\CZ^{ECH}((\eta,m-1)) - w_\tau(\zeta^+) \ge n - 1\]
where $\eta$ is elliptic and $n$ denotes the number of components of the braid $\zeta^+$. This implies (\ref{eq:topological_upper_bound_1}), and (\ref{eq:topological_upper_bound_2}) can be argued analogously. \end{proof}

\begin{remark} In \cite{Hutchings_index_revisited}, Hutchings also develops a general theory of the $J_0$-index, a replacement for the ECH index that provides a topological filtration on any type of contact homology. Proposition \ref{lem:topological_upper_bound} is essentially using this theory in the relative case. We will not develop a relative $J_0$-invariant in this paper.
\end{remark}

Next, we demonstrate the compactness of the space of $J$-holomorphic currents of ECH index $1$. This is the first part of Theorem \ref{thm:intro_Legendrian_ECH_compactness} in the introduction. 

\begin{lemma} \label{lem:compactness_ECH_1}
The moduli space of $J$-holomorphic currents $\mathcal{M}_1 (\Xi_+,\Xi_-) / \mathbb{R}$ of ECH index $1$ is finite.
\end{lemma}

\noindent This is a virtual repeat of the same argument in ECH using Gromov compactness and bounds on topological complexity of $J$ holomorphic curves. We sketch it here briefly for completeness. 

\begin{proof} Fix a sequence of distinct currents with boundary of index $1$, denoted by
\[\mathcal{C}^\nu \in \mathcal{M}(\Xi_+,\Xi_-)/\R\]
By Proposition \ref{prop:ECH_index_1_currents}, we have $\mathcal{C}^\nu = C^\nu \sqcup \mathcal{T}^\nu$ where $C^\nu$ is a connected, embedded $J$-holomorphic curve with $I(C^\nu) = 1$ and $\mathcal{T}^\nu$ is a trivial current over an orbit-chord subset $\Theta^\nu$ of $\Xi_+$ and $\Xi_-$. There are only finitely many orbit-subsets of $\Xi_+$, so after passing to a subsequence we may assume that
\[\mathcal{T}^\nu = \mathcal{T} = \R \times \Theta \qquad\text{for all }\nu\]
Moreover, let $\Theta_+$ and $\Theta_-$ be the unique orbit sets such that $\Xi_\pm = \Theta_\pm \cup \Theta$. Then
\[C^\nu \in \mathcal{M}(\Theta_+,\Theta_-)/\R\]
By Gromov compactness for currents (Theorem \ref{thm:Gromov_compactness}), we can choose a subsequence of $C^\nu$ that lie in the same surface class $A$ in $S(\Theta_+,\Theta_-)$. By Lemma \ref{lem:topological_upper_bound}, we have
\[
-\bar{\chi}(C) \le -\frac{1}{2} \cdot \mu_\tau(A) + Q_\tau(A) + \CZ_\tau^{ECH}(\check{\Theta}_+) -\CZ_\tau^{ECH}(\check{\Theta}_-)
\]
Thus the genus and number of boundary components of $C$ is uniformly bounded. We can thus apply SFT compactness for curves with boundary (cf. Abbas \cite[Thm. 3.6]{abbasSFTcompactness}) to acquire a $J$-holomorphic building
\[
{\bf C} = (C_1,C_2,\dots,C_m)
\]

\begin{claim} \label{claim:no_levels} The building ${\bf C}$ consists of a single level $C$ of ECH index $1$.
\end{claim}

\begin{proof}  Since the ECH index is additive under composition and the surface class of ${\bf C}$ agrees with that of $C^\nu$ for large $\nu$, we know that
\[
\sum_j I(C_j) = I(C^\nu) = 1
\]
Moreover, $I(C_j) \ge 0$ for each $j$ since $J$ is regular. Therefore, for some fixed $i$ we must have
\[I(C_i) = \op{ind}(C_i) = 1\]
and all other levels $C_j$ are ECH index $0$. By Lemma \ref{lem:ECH_index_0_currents}, we have the following equality of currents
\[
C_j = \R \times \Xi_+ \quad\text{for $j < i$}\qquad\text{and}\qquad C_j = \R \times \Xi_-\quad\text{for $j > i$}
\]
and each $C_j$ is (as a curve) an explicit branched cover of the trivial cylinders over the simple orbits and chords of $\Xi_\pm$.  

\vspace{3pt}

Now fix an orbit or chord $\gamma$ of multiplicity $m$ in $\Xi_+$. Let $Z_j$ be the sub-curve $C_j$ that is a cover of $\R \times \gamma$ for $1 \le j < i$ and consider the building ${\bf Z} = (Z_1,\dots,Z_{i-1},Z_i)$. Then Lemma \ref{lem:trivial_cylinder_covers} implies that ${\bf Z}$ consists of unbranched covers if $\gamma$ is an orbit, and Lemma \ref{lem:trivial_strip_covers} implies the same if $\gamma$ is a chord. Such components cannot appear in a building, so we must have $i = 1$ and $C_1 = C_i$ is non-trivial.

\vspace{3pt}

The same argument shows that there are no trivial levels $C_j$ for $j > i$. \end{proof} 

Returning to the proof of Lemma \ref{lem:compactness_ECH_1}, we now find that the embedded curves $C^\nu$ converge to the curve $C$ in the moduli space of somewhere injective Fredholm index $1$ curves. Since $J$ is regular, this moduli space is discrete. Thus $C^\nu = C$ for sufficiently large $\nu$, and
\[\mathcal{C}^\nu = C^\nu \sqcup \mathcal{T} = C \sqcup \mathcal{T} \qquad\text{for large }\nu\]
This contradicts the assumption that the currents $\mathcal{C}^\nu$ are distinct. This concludes the proof. \end{proof}

Next, we demonstrate the compactness of the space of $J$-holomorphic currents of ECH index $2$, under the assumption that the positive and negative ends are on ECH generators.

\begin{lemma} \label{lem:compactness_ECH_2} Let $\Xi_1,\Xi_2$ be ECH generators. Fix a regular, tailored $J$ on $(\R \times Y,\R \times \Lambda)$. Then the moduli space of ECH index two holomorphic currents
\[\mathcal{M}_2(\Xi_1,\Xi_2)/\R\]
is a topological $1$-manifold that admits a (possibly singular) compactification
\[
\bar{\mathcal{M}}_2(\Xi_1,\Xi_2)/\R := \mathcal{M}_2(\Xi_1,\Xi_2)/\R \cup \Big( \bigcup_\Theta  \mathcal{M}_1(\Xi_1,\Theta)/\R \times \mathcal{M}_1(\Theta,\Xi_2)/\R\Big)\]
\end{lemma}

\noindent Note that the above compactification is topologized with the Gromov topology on currents. It is \emph{not} necessarily a manifold itself. There could be many connected components of $\mathcal{M}_2(\Xi_1,\Xi_2)/\mathbb{R}$ that have the same boundary point in the compactification. The number of such components is given by a obstruction bundle count that we describe in more detail below (Section \ref{subsec:obg}).
 
\begin{proof} By Gromov compactness for currents, we know that $C^\nu$ converges in the Gromov topology to a broken current
\[
\bar{\mathcal{C}} = (\mathcal{C}_1,\dots,\mathcal{C}_m)
\]
The ECH index is additive (Proposition \ref{prop:basic_prop_ECH}) and non-negative (by Lemma \ref{lem:ECH_index_0_currents}), so the broken current must have $m \le 2$ component currents. It suffices to show that if $m = 1$ and $\bar{\mathcal{C}} = \mathcal{C}$, then
\[\mathcal{C} = C\]
where $C$ is a connected embedded curve of ECH index two. This is the content of Proposition \ref{prop:ECH_index_2_currents}.\end{proof}

\subsection{Moduli Spaces Truncation} \label{subsec:truncated_moduli_space} We now formally describe the truncated moduli space of ECH index two currents in Theorem \ref{thm:intro_Legendrian_ECH_compactness}. 

\vspace{3pt}

In the case of standard ECH, the construction of this moduli space and the proof of its key properties is spread over several papers \cite{Hutchings_lectures_on_ECH,obs1,obs2}. Thus, although the construction is exactly analogous to that case, we collect several details here for the reader.

\vspace{3pt}

We will require a version of gluing pairs in the sense of \cite[Definition 1.9]{obs1}. We adopt the following definition.

\begin{definition} \label{def:ECH_gluing_pair} An \emph{ECH gluing pair} in $(\R \times Y,\R \times \Lambda)$ is a pair of $J$-holomorphic curves
\[
(u_+,u_-)
\]
satisfying the following properties.

\begin{itemize}
\item $u_+$ and $u_-$ have ECH index one.
\item $u_+$ and $u_-$ are embedded, except for unbranched covers of cylinders of orbits and chords.
\item The orbit-chord sets at the negative end of $u_+$ and the positive end of $u_-$ are the same.
\end{itemize}
Moreover, let $\gamma$ be a simple elliptic orbit of total multiplicity $m$ at the negative end of $u_-$ (or equivalently, at the positive end of $u_+$).
\begin{itemize}
\item The partition $p_-(u_+;\gamma)$ of $m$ determined by the negative end of $u_+$ is the negative ECH partition $p_-(m)$.
\item The partition $p_+(u_-;\gamma)$ of $m$ determined by the positive end of $u_-$ is the positive ECH partition $p_+(m)$
\end{itemize}
There is a space of ECH index two curves that are \emph{nearly broken} at $(u_+,u_-)$, denoted by
\[
\mathcal{G}_\delta(u_+,u_-) \subset \mathcal{M}_2(\Xi_+,\Xi_-)/\R
\]
consisting of all connected, embedded curves $C$ in $\mathcal{M}_2(\Xi_1,\Xi_2)$ that admit constants $R_+,R_-$ and a decomposition
\[C = C'_+ \cup C_0 \cup C'_-\]
that satisfy the following conditions.
\begin{itemize}
\item $C_0$ is contained in the radius $\delta$ neighborhood of $\R \times \Theta$, where $\Theta$ is the orbit-chord set at the negative end of $u_+$.
\item $R_+ - R_- > \frac{2}{\delta}$.
\item There is a section $\psi_+$ of the normal bundle of $u_+$ such that
\[|\psi_+| < \delta \quad\text{and}\quad C_+'  = \big(\op{exp}_{C_+}(\psi_+) + R_+ \big) \cap [-1/\delta,\infty) \times Y\]
\item There is a section $\psi_-$ of the normal bundle of $u_-$ such that
\[|\psi_-| < \delta \quad\text{and}\quad 
C_-' = \big(\op{exp}_{C_-}(\psi_-) + R_- \big) \cap (-\infty,1/\delta] \times Y
\]
\end{itemize}
\end{definition}

\begin{lemma} \label{lem:convergence_of_Gd} Fix an ECH gluing pair of the form 
\[
(u_+,u_-)
\]
Then there is an $\epsilon > 0$ such that any sequence of curves $C_i \in \mathcal{G}_\epsilon(\mathcal{C}_+,\mathcal{C}_-)/\R$ has a subsequence that converges in the SFT topology (see \cite[Thm 3.6]{abbasSFTcompactness}) to one of the following types of $J$-holomorphic buildings.
\begin{itemize}
\item  A connected, embedded curve $C$ of ECH index two in $\mathcal{G}_\delta(\mathcal{C}_+,\mathcal{C}_-)$
\item A building $(u_+,v_1,\dots,v_k,u_-)$ where $v_i$ is a union of branched covers of trivial cylinders representing $\R \times \Theta$.
\end{itemize}
\end{lemma}

\begin{proof} Let $C_i \in \mathcal{G}_\epsilon(u_+,u_-)/\R$ be a sequence. By SFT compactness in the relative case \cite[Thm 3.6]{abbasSFTcompactness} and the topological bounds in Lemma \ref{lem:topological_upper_bound}, we know that there is a limit SFT building of the form
\[{\bf C} = (w_1,\dots,w_m)\]
The sum of the ECH indices of the levels must be $2$ by addivity of the ECH index. Therefore, there are two cases.

\vspace{3pt}

{\bf Case 1.} In the first case, there is a single level of ECH index two and every other level is ECH index zero. In this case, we can argue that there is only one level $C = {\bf C}$, by an identical as in Claim \ref{claim:no_levels}. It follows that $C_i \to C$ in the SFT topology. 

\vspace{3pt}

{\bf Case 2.} In the second case, there are two levels $w_a$ and $w_b$ of ${\bf C}$ of ECH index one, and the remaining levels are ECH index zero. The levels $w_a$ and $w_b$ must be equal to $u_+$ and $u_-$ due to the construction of $\mathcal{G}_\delta$ and the definition of SFT convergence. Moreover, since the positive and negative ends of the curves $C_i$ satisfy the ECH partition conditions, by \cite[Lemma 1.7]{obs1} and \cite[Exercise 3.14]{Hutchings_lectures_on_ECH}, the levels $w_j$ must be trivial for $j < a$ and $j > b$. 
\end{proof}

\begin{definition} \label{def:count_of_gluings} The \emph{count of gluings} $\# G(u_+,u_-)$ of an ECH gluing pair $(u_+,u_-)$ is defined as follows. Let $\epsilon$ be as in Lemma \ref{lem:convergence_of_Gd} and choose an open subset $U \subset \mathcal{G}_\epsilon(u_+,u_-)/\R$ such that
\begin{itemize}
\item $\bar{U}$ has finitely many boundary points.
\item $U$ contains $\mathcal{G}_\delta(u_+,u_-)/\R$ for some $\delta < \epsilon$.
\end{itemize}
Then we define $\# G(u_+,u_-)$ to be the number of boundary points in the closure $\bar{U}$ of $U$.
\end{definition}

We are now ready to construct the truncated moduli space of ECH index two curves. For each ECH gluing pair $(u_+,u_-)$ between ECH generators $\Xi_+$ and $\Xi_-$, fix for the remainder of the section an open set
\[
U(u_+,u_-) \subset \mathcal{G}_\epsilon(u_+,u_-)
\]
as in Definition \ref{def:count_of_gluings}. For any ECH generators $\Theta_+$ and $\Theta_-$, we adopt the notation
\[
W(\Theta_+,\Theta_-) := \big\{C \sqcup \mathcal{T} \; : \; C \in U(u_+,u_-) \text{ for some ECH gluing pair $(u_+,u_-)$ and $\mathcal{T}$ is trivial}\big\}
\]

\begin{definition}[Truncated Moduli Space] Let $\Theta_+$ and $\Theta_-$ be ECH generators and fix a regular, tailored $J$ on $(\R \times Y,\R \times \Lambda)$. The \emph{truncated moduli space}
\[\mathcal{M}_2'(\Theta_+,\Theta_-)/\R \subset \mathcal{M}_2(\Theta_+,\Theta_-)/
\R\]
is defined as follows. By Proposition \ref{prop:ECH_index_2_currents}, we may divide $\mathcal{M}_2(\Theta_+,\Theta_-)$ into disjoint pieces
\[
\mathcal{A}_2(\Theta_+,\Theta_-) = \big\{C \sqcup \mathcal{T} \; : \; \text{$C$ is connected with $I(C) = 2$ and $\mathcal{T}$ trivial.}\big\}
\]
\[
\mathcal{B}_2(\Theta_+,\Theta_-) = \big\{C \sqcup D \sqcup \mathcal{T} \; : \; \text{$C,D$ are connected with $I(C) = I(D) = 1$ and $\mathcal{T}$ trivial.}\big\}
\]
We define $\mathcal{M}_2'(\Theta_+,\Theta_-)/\R$ as a union of pieces
\[\mathcal{A}'_2(\Theta_+,\Theta_-) \subset \mathcal{A}_2(\Theta_+,\Theta_-) \qquad\text{and}\qquad \mathcal{B}'_2(\Theta_+,\Theta_-) \subset \mathcal{B}_2(\Theta_+,\Theta_-)\]

{\bf Truncation Of $\mathcal{A}$.} To truncate $\mathcal{A}_2(\Theta_+,\Theta_-)$, we note that by Gromov compactness for currents (Theorem \ref{thm:Gromov_compactness}), any sequence of currents $\mathcal{C}_i \in \mathcal{A}_2(\Theta_+,\Theta_-)/
\R$ has a subsequence of the form
\[
\mathcal{C}_i = C_i \sqcup \mathcal{T}
\]
where $C_i$ is a connected curve with $I(C_i) = 2$ from $\Theta_+$ to $\Theta_-$ converging to
\begin{itemize}
\item a connected ECH index two curve $C \in \mathcal{M}_2(\Theta_+,\Theta_-)/
\R$ or
\item a building $(u_+,v_1,\dots,v_k,u_-)$ where $(u_+,u_-)$ are an ECH gluing pair and $v_i$ are covers of trivial cylinders and chords.
\end{itemize}
In the latter case, it follows that $C_i \notin W(\Theta_+,\Theta_-)$ for sufficiently large $i$. We thus let
\[
\mathcal{A}'_2(\Theta_+,\Theta_-) = \mathcal{A}_2(\Theta_+,\Theta_-) \setminus W(\Theta_+,\Theta_-)
\]
There is a natural projection map
\[\Pi:\partial \mathcal{A}_2'(\Theta_+,\Theta_-)/\R \to \bigsqcup_{\Theta'} \; \mathcal{M}_1(\Theta_+,\Xi)/\R \times \mathcal{M}_1(\Xi,\Theta_-)/\R\]
defined by mapping $C \sqcup \mathcal{T} \in \partial \mathcal{A}_2'(\Theta_+,\Theta_-)/\R$ to the broken $J$-holomorphic current
\[(\mathcal{C}_+,\mathcal{C}_-) = (u_+ \sqcup \mathcal{T},u_- \sqcup \mathcal{T})\]
where $(u_+,u_-)$ is the unique ECH gluing pair such that $C \in \mathcal{G}_\delta(\mathcal{C}_+,\mathcal{C}_-)$. Moreover, we have
\[
\#\Pi^{-1}(\mathcal{C}_+,\mathcal{C}_-) = \#\bar{U}(u_+,u_-) = \#G(u_+,u_-)
\]

\vspace{3pt}

{\bf Truncation Of $\mathcal{B}$.} 
To truncate $\mathcal{B}_2(\Theta_+,\Theta_-)$, note that a component $S \subset \mathcal{B}_2(\Theta_+,\Theta_-)/\R$ is $1$-dimensional, consisting of curves of the form
\[
\mathcal{C}_s = C \sqcup (D + s)\sqcup \mathcal{T} \qquad\text{for }s \in \R
\]
Here $C$ and $D$ are connected, embedded curves of ECH index $1$ and $\mathcal{T}$ is a trivial current. Let $\mathcal{S}_+$ and $\mathcal{S}_-$ be the trivial currents over the positive and negative ends of $C$, and let $\mathcal{T}_+$ and $\mathcal{T}_-$ be the analogous currents for $D$. Then
\[
\lim_{s \to -\infty} \mathcal{C}_s = (C \sqcup \mathcal{T}_+ \sqcup \mathcal{T},D \sqcup \mathcal{S}_- \sqcup \mathcal{T})\]
\[\lim_{s \to +\infty} \mathcal{C}_s = (D \sqcup \mathcal{S}_+ \sqcup \mathcal{T},C \sqcup \mathcal{T}_+ \sqcup \mathcal{T})
\]
where the limit is taken in the Gromov topology on broken currents. We now truncate $S$ by setting
\[
S' = \{\mathcal{C}_s \; : \; s \in [-1,1]\} \subset S
\]
and let $\mathcal{B}'_2$ be the union of these pieces over all components $S$. There is a natural projection map
\[\Pi:\partial \mathcal{B}_2'(\Theta_+,\Theta_-)/\R \to \bigsqcup_{\Theta'} \; \mathcal{M}_1(\Theta_+,\Xi)/\R \times \mathcal{M}_1(\Xi,\Theta_-)/\R\]
sending $C_{-1}$ to $\lim_{s \to -\infty} \mathcal{C}_s$ and $C_1$ to the broken current $\lim_{s \to \infty} \mathcal{C}_s$. Note that
\[\#\Pi^{-1}(\mathcal{C}_+,\mathcal{C}_-) = 1 \qquad\text{for any}\qquad (\mathcal{C}_+,\mathcal{C}_-) \in \Pi(\partial\mathcal{B}'_2(\Theta_+,\Theta_-)/\R)\]
\end{definition}

\subsection{Gluing Counts} \label{subsec:obg} We conclude this section by explaining the applications of the obstruction bundle gluing results of Hutchings-Taubes \cite{obs1,obs2} to our setting. In particular, we prove the last item of Theorem \ref{thm:intro_Legendrian_ECH_compactness} as Corollary \ref{cor:truncation_boundary}. This completes the proofs of all the main results of the paper.

\vspace{3pt}

Let $\gamma$ be a simple Reeb orbit and fix two partitions $p$ and $q$ of an integer $m$. In \cite[\S 1.5, 1.6]{obs1}, Hutchings-Taubes define \emph{gluing coefficients}
\[
c_\gamma(p,q) \in \Z
\]
that count (roughly) the number of ways to glue two curves asymptotic to $\gamma$ with multiplicity $m$, and corresponding partitions $p$ and $q$. For chords, we adopt the following definition.

\begin{definition} \label{def:chord_gluing_coeffs} The \emph{gluing coefficient} of a Reeb chord $\eta$ of $(Y,\Lambda)$ and of multiplicity one is given by
\[
c_\gamma(\gamma,1) =1.
\]
\end{definition}

Given a $J$-holomorphic curve $v$, we let $m_\pm(v,\eta)$ be the multiplicity of the orbit or chord $\eta$ at the $\pm$-end of $v$, and $p_\pm(v,\gamma)$ be the partition of $m_\pm(v,\gamma)$ determined by the $\pm$-end of a finite energy curve $v$ at any orbit $\gamma$. Finally, if $\gamma$ is a simple orbit and $\eta$ is a chord, then
\[
c_\gamma(u_+,u_-) := c_\gamma(p_-(u_+,\gamma),p_+(u_-,\gamma)) \qquad\text{and}\qquad c_\eta(u_+,u_-) := 1
\]

\begin{proposition} \label{prop:gluing_coeffs} Fix a tailored, regular almost complex structure $J$ on $(\R \times Y,\R \times \Lambda)$ and let $(u_+,u_-)$ be an ECH gluing pair. Then
\[
\#G(u_+,u_-) = \prod_\gamma c_\gamma(u_+,u_-)
\]
The product is over all simple closed orbits and chords appearing as a negative end of $u_+$. \end{proposition}

\begin{proof} The essential ingredients to this formula appear already in Hutchings-Taubes \cite{obs1,obs2}. Here we briefly sketch the argument.

\vspace{3pt}

Let $u_1$ and $u_2$ be a ECH gluing pair. Throughout we assume we have chosen a generic $J$. For simplicity of exposition we first assume none of the $u_1$ and $u_2$ have boundary, and that all negative ends of $u_1$ are asymptotic to one (simple) Reeb orbit with multiplicities (therefore the same is true for all positive punctures of $u_2$). We illustrate the general methodology in this case and explain the more general case later.

Then recall from Section 5.2 \cite{obs2} in order to glue $u_1$ and $u_2$ together, we must first preglue a branched cover of trivial cylinder $v$ of Fredholm index zero between them, so that the partition conditions at positive (resp. negative ends) of $v$ match the partition conditions of negative ends of $u_1$ (resp. positive ends of $u_2$). Then as in Section 5.4 that there is a gluing is translated into three equations as we describe below.

\vspace{3pt}

To borrow the notation of Section 5.3 in \cite{obs2}, let $N_*$ denote the normal bundle of $u_i, v$. Let $\mathcal{H}_1(N_*)$ denote a suitable completion of sections of $N_i$. Let $\psi_i \in \mathcal{H}_1( N_i)$ and $\phi \in \mathcal{H}_2 
(N_v)$ denote sections of the normal bundle. We consider deforming the preglued curve consisting of $u_1, u_2$ and $v$ using the sections $\psi_i$ and $\phi$ (patched together using cutoff functions, as in Section 5.3 in \cite{obs2}). The condition that gluing exists is translated into a system of three equations
\[
\Theta_1(\psi_1,\phi) =0, \quad \Theta_v(\psi_1,\psi_2,\phi) =0, \quad \Theta_2(\psi_2,\phi)=0
\]
so that given fixed $\phi$, the equations $\Theta_1,\Theta_2$ can always be solved essentially uniquely (with $\psi_i$ expressed as functions of $\phi$). See Section 5.6 of \cite{obs2}.

\vspace{3pt}

The number of gluings is equivalent to the number of solutions of $\Theta_v$ with independent variable $\phi$; and $\psi_1$ and $\psi_2$ are functions of $\phi$. This is the content of Theorem 7.3 (b) in \cite{obs2}. 

\vspace{3pt}

The number of solutions of $\Theta_v=0$ is counted by the number of zeroes of a section $\mathfrak{s}$ of an obstruction bundle $\mathcal{O}(\Sigma) \rightarrow \mathcal{M}_\Sigma$. The base $\mathcal{M}_\Sigma$ is the space of branched covers of the trivial cylinder satisfying the same multiplicity and parition conditions as $v$; and the fiber is the (dual of) cokernel of a linearized Cauchy Riemann operator $D_\Sigma$ acting on the normal bundle of the branched cover. This obstruction bundle is defined carefully in Definition 2.16 of \cite{obs1}. See Section 2.1 of \cite{obs1} for definitions and properties of the base $\mathcal{M}_\Sigma$ and Section 2.2 and 2.3 for the linearized Cauchy Riemann $D_\Sigma$. This obstruction section $\mathfrak{s}$ whose zero corresponds to 
the number of solutions of $\Theta_v=0$ is described carefully in Equation 5.43 of \cite{obs2}. The count of zeroes can be viewed as computing some version of a relative Euler class of this obstruction bundle. The zeros of $\mathfrak{s}$ are transverse (proved in Section 10 of \cite{obs2}), and the number of zeroes counted with sign is equal to $c_\gamma$. The definition of $c_\gamma$ is given in equation 1.7 of \cite{obs1}, and described more carefully in Sections 1.5 and 1.6 of \cite{obs1}. That the number of zeros is equal to $c_\gamma$ is the content of Theorem 1.13 in \cite{obs1}.

\vspace{3pt}

From this way of gluing, we see readily that we can count the number of gluings as a local count of zeroes of obstructions sections over spaces of branched covers of trivial cylinders. If $u_1$ and $u_2$ meet along multiple Reeb orbits, the count of total number of gluings is the product of the count of zeroes of such obstruction sections over all the Reeb orbits. This is the content of Theorem 1.13 in \cite{obs1}.

\vspace{3pt}

In the case of ECH gluing pairs for Legendrian ECH, we can again reduce the count of gluing to count of zeroes of obstruction sections over branched covers of trivial cylinders and trivial strips. We here use the observation that by partition conditions all chords between $u_+$ and $u_-$ must have mulitplicity one. Over branched covers of trivial cylinders such counts over closed Reeb orbits is given by the obstruction bundle counts given in \cite[Theorem 1.13]{obs1}, and the gluing over Reeb chords is just standard gluing.
\end{proof}

\begin{corollary}[Theorem \ref{thm:intro_Legendrian_ECH_compactness}, Truncation] \label{cor:truncation_boundary} The inverse image $\Pi^{-1}(\mathcal{C}_+,\mathcal{C}_-)$ of a pair of currents in $\mathcal{M}_1(\Theta,\Theta')/\R \times \mathcal{M}_1(\Theta',\Xi)/\R$ under the map 
\[\Pi:\mathcal{M}'_2(\Theta_+,\Theta_-) \to \bigsqcup_{\Theta'} \; \mathcal{M}_1(\Theta_+,\Xi)/\R \times \mathcal{M}_1(\Xi,\Theta_-)/\R\]
has an odd number of points if and only if the orbit set $\Theta'$ is an ECH generator.
\end{corollary}

\begin{proof} For any ECH gluing pair, the gluing coefficients $c_\gamma(u_+,u_-)$ are always odd when $\gamma$ is is an elliptic orbit \cite[Prop. 7.26]{obs1}. They are odd if and only if the multipliciy of $\gamma$ in $\Xi_+$ or $\Xi_-$ is one when $\gamma$ is a hyperbolic orbit \cite[Def. 1.14]{obs1} or chord (see Definition \ref{def:chord_gluing_coeffs}). Thus this follows from Proposition \ref{prop:gluing_coeffs}. \end{proof}

\appendix
\section{Decay Estimates} \label{sec:decay_estimates}
In this appendix, we prove asymptotic formulas for ends of $J$ holomorphic curves converging to a Reeb chord in the same style of Siefring \cite{siefring2008relative}. The asymptotic formula for a single end is already established in \cite{abbas1999finite}. Our goal is to explain the general case, where multiple ends approach the same chord. This will be essential in our proof of the general writhe bound. 

\begin{remark} The general writhe bound is not strictly necessary for our formulation of Legendrian ECH in Section \ref{subsubsec:legendrian_ECH_complex}. This is due to our combination of maximum principles and the definition of ECH generators prevents breaking along multiply covered chords in the ECH setting. 

\vspace{3pt}

However, the general asymptotic formula and writhe bound may play a role in other theories that count curves asymptotic to chords with higher multiplicities.
\end{remark}

The presented proof is an amalgamation of techniques and ideas found in \cite{hofer1996properties}, \cite{siefring2008relative}, and \cite{abbas1999finite}. The proof has the following outline.
\begin{enumerate}
    \item First, we describe the local geometric setup as in \cite{abbas1999finite}. In particular, we follow the choices of metrics and connections in \cite{abbas1999finite}.
    \item Suppose we have two ends approaching the same chord. Following \cite{siefring2008relative}, we show the difference between the two ends satisfy a particular PDE (see equation \ref{eq:projection_to_contact_plane}).
    \item We show that the bounded solutions to Equation \ref{eq:projection_to_contact_plane} have an asymptotic representation formula. We do this using the setup of \cite{abbas1999finite}, and essentially the same techniques there that proved asymptotic formula holds for a single end.
    \item Finally we translate the asymptotic formula of the difference between two ends into behaviour of two ends of the holomorphic curve, following \cite{siefring2008relative}.
\end{enumerate}

\subsection{Local differential geometry}
We start by recalling the geometric setup of \cite{abbas1999finite}. By rescaling the contact form, we may assume that the Reeb chord has action 1. We choose a neighborhood of the Reeb chord and local coordinates $x, y, z$ such that the following holds.
\begin{itemize}
\item The Reeb chord $\eta$ is the line segment $\{(0,0,t):t\in [0,1]\}$
\item The two Legendrians consist of the lines
\[L_1 := \{(t,0,0), t\in \mathbb{R}\} \qquad\text{and}\qquad L_2:=\{(0,t,1), t\in \mathbb{R}\}\]
\item The contact form $\lambda$ is given by
\[\lambda = f dz + c_1 dx+c_2dy\]
where $f = 1$ along $\eta$, and $c_1$ and $c_2$ are smooth functions satisfying $c_1=c_2 =0$ on the $z$-axis.
\item The contact form $\lambda$ satisfies \[d\lambda(0,0,z) = a dx \wedge dy\]
where $a(x,y,z)$ is positive in a neighborhood of the Reeb chord.
\item The Reeb vector field is given by
\[
X(x,y,z) = f(x,y,z) \cdot (0,0,1)
\]
where $f(0,0,z) =1$.

 \end{itemize}
Given any $\lambda$-adapted almost complex structure $J$ on the symplectization $\R\times \R^3$, we write $M(t)$ to denote the 2 by 2 matrix that is the restriction of $J$ to the contact plane $\ker \lambda _{0,0,z}$. It satisfies
\[M^TJ_0M=J_0 \qquad\text{and}\qquad -J_0 M >0\]
where $J_0$ is the standard 2 by 2 almost complex structure.

\vspace{3pt}

We now consider two $J$-holomorphic curves in the symplectization $\R\times \R^3$ that have boundary punctures that asymptotic to the Reeb chord $\eta$.
\[U,V:[0,\infty) \times [0,1] \to \R \times \R^3 \qquad \text{with}\qquad \lim_{s \to \infty} U(s,-) =  \lim_{s \to \infty} V(s,-) = \eta\]

In particular we consider how boundary punctures approach Reeb chords. We consider the case where there are two strip-like ends approaching the same chord. Assume that in the local conformal coordinates $(s,t)$ of the domain, the first end is parametrized as $$U(s,t)=(b(s,t),x(s,t),y(s,t),z(s,t)).$$ Using this parametrization, we define the local embedding
$E:[R,\infty)\times D_\epsilon\times [0,1]\to \mathbb{R}\times \mathbb{R}^3$ to be
\begin{equation}
    E(s,h_1,h_2,t)=(b(s,t),x(x,t)+h_1,y(s,t)+h_2,z(s,t)-c_1fh_1-c_2fh_2)
\end{equation}
where $R>0$ is sufficiently large, $\epsilon>0$ is sufficiently small, and $D_\epsilon$ is the $\epsilon$-neighborhood of the origin. With the embedding $E$ understood, any nearby end could be viewed as a graph $E(s,\eta_1(s,t), \eta_2(s,t),t)$. In particular, we now parametrize the second end as
\begin{equation}
\begin{split}
V(s,t)&=E(s,\eta_1(s,t), \eta_2(s,t),t)
\\
&=(b(s,t),x(s,t)+\eta_1(s,t),y(s,t)+\eta_2(s,t),z(s,t)-fc_1\eta_1(s,t)-fc_2\eta_2(s,t)).
\end{split}
\end{equation}
Note that by abuse of notation we have used $(s,t)$ to denote the coordinates for the domains of both ends. We also note
here that $u(s,t)=(x(s,t),y(s,t),z(s,t))$ is the projection of $U$ onto the last three coordinates, and we shorthand the functions $f(u(s,t))$, $c_1(u(s,t))$ and $c_2(u(s,t))$ as $f$, $c_1$ and $c_2$ respectively. Under these parametrizations, the map $(\eta_1(s,t), \eta_2(s,t))$ satisfies the following PDE:
\begin{equation}\label{eq:projection_to_contact_plane}
    \partial_s\begin{pmatrix}\eta_1\\\eta_2\end{pmatrix}+M(u(s,t))\partial_t\begin{pmatrix}\eta_1\\\eta_2\end{pmatrix}+\Delta(s,t)\begin{pmatrix}\eta_1\\\eta_2\end{pmatrix}=0
\end{equation}
where the matrix $\Delta(s,t)$ decays exponentially, i.e. there are positive constants $d$ and $M_\beta$ such that for any multi-index $\beta$, 
$
|\partial^\beta\Delta|\le M_\beta e^{-ds}.
$

\subsection{The asymptotic formula}
We now use the estimates from \cite{abbas1999finite} to derive an asymptotic formula for $(\eta_1,\eta_2)$. The main difference in Abbas' approach from that of Hofer or Siefring is choosing a 1-parameter family of norms dependent on $s$ which are equivalent to the ordinary $L^2$ norm on $[0,1]$. We first review this part of his construction, then we will state and prove the analogous lemmas in \cite{abbas1999finite} that will culminate in the asymptotic formula.

We first recall properties of the matrix $M(u(s,t))$, which we will often abbreviate $M(s,t)$ for convenience. It converges to a matrix $M_\infty(t)$ as $s\rightarrow \infty$ at an exponential rate. Further, since $M(s,t)$ comes from restricting the ambient $\lambda$-adapted almost complex structure on the contact distribution, it satisfies the following algebraic properties:
\begin{enumerate}
    \item $M^2 =-1$
    \item $M^TJ_0M=J_0$
    \item $-J_0 M >0$.
\end{enumerate}

We consider the following family of inner products on $L^2([0,1],\mathbb{R}^2)$, let $v_1, v_2 \in L^2([0,1],\mathbb{R}^2)$:
\begin{equation}
    (v_1,v_2)_s := \int_0^1 \langle v_1, -J_0 M(s,t) v_2 \rangle dt
\end{equation}
and it follows from the fact that $M(s,t)$ exponentially decays to $M_\infty(t)$ that for large values of $s$ this norm is uniformly equivalent to the ordinary $L^2$ norm.

We define 
\begin{equation}
W_{\Gamma}^{1,2}([0,1],\R^2) := \{ v\in W^{1,2}([0,1],\R^2) | v(s,0) \in \R \cdot (1,0), v(s,1)\in \R \cdot (0,1)\}
\end{equation}
hence we can further define an inner product
\begin{equation}
(v_1,v_2)_{s,1,2} : = (v_1,v_2)_s + (v_1',v_2')_s
\end{equation}
where $'$ denotes the derivative with respect to $s$.

We consider the trivial vector bundle $E:= [s_0,\infty) \times [0,1] \times \R^2$ with the fiber $\R^2$. We view maps $\zeta(s,t):[s_0,\infty)\times [0,1] \rightarrow \R^2$ as sections of $E$. With the above choice of metric, we choose a connection on $E$ by defining the following covariant derivatives:
\[
\nabla_s \zeta := \partial_s \zeta -\frac{1}{2}M(s,t) \partial_s M(s,t) \cdot \zeta(s,t)
\]
\[
\nabla_t \zeta := \partial_t \zeta -\frac{1}{2}M(s,t) \partial_t M(s,t) \cdot \zeta(s,t).
\]
We abbreviate:
\[
\Gamma_1 := -\frac{1}{2}M(s,t) \partial_s M(s,t)
\]
\[
\Gamma_2 := -\frac{1}{2}M(s,t) \partial_t M(s,t).
\]
The following properties of this connection are proved in \cite{abbas1999finite}:
\begin{proposition}
Let $X = a_1\partial_s + a_2 \partial_t $, then we define $\nabla_X := a_1 \nabla_s + a_2 \nabla _t$.
\begin{enumerate}
    \item For $u_1,u_2\in W^{1,2}([s_0,\infty)\times [0,1], \R^2)$, we have:
    \begin{equation}
        \frac{d}{ds}(u_1,u_2)_s= (\nabla_s u_1,u_2)_s + (u_1,\nabla_s u_2)_s.
    \end{equation}
    \item If $B$ is a section in the endomorphism bundle of $E$, i.e. $B\in \Gamma(End(E)) = \Gamma ([s_0,\infty) \times [0,1] \times End(\R^2))$, then we define:
    \begin{equation}
        \nabla_s B:= \partial_s B  + \frac{1}{2}[B \cdot M(s,t) \partial_s M(s,t)- M(s,t) \partial_s M(s,t) \cdot B].
    \end{equation}
    This satisfies:
    \begin{equation}
        \nabla_s(B\cdot \zeta) = \nabla_sB \cdot \zeta + B \cdot \nabla_s \zeta.
    \end{equation}
    In particular we have
    \begin{equation}
        \nabla_s M(s,t)=0
    \end{equation}
    and 
    \begin{equation}
        \partial_t \nabla_s \zeta - \nabla_s \partial_t \zeta = \partial_t \Gamma_1 \cdot \zeta.
    \end{equation}
    
\end{enumerate}
\end{proposition}

We also define the following family of unbounded self-adjoint operators:
\begin{equation}
    A(s): W_{\Gamma}^{1,2}([0,1],\R^2)  \subset L^2([0,1],\R^2) \longrightarrow L^2([0,1],\R^2)
\end{equation}
as
\begin{equation}
    A(s) : = -M(s,t) \frac{d}{dt}.
\end{equation}
We also use $A_\infty$ to denote the operator $-M_\infty(t) \frac{d}{dt}$.
The following properties are established in Proposition 3.4 in \cite{abbas1999finite}.

\begin{proposition}\label{proposition:some properties of A}
\begin{enumerate}
    \item $A(s)$ is self-adjoint on $(L^2([0,1],\R^2), (\cdot,\cdot)_s)$.
    \item $\textup{Ker}\ A(s) = \{0\}$.
    \item There exists $\delta>0$, independent of $s$, so that for all $s \in [s_0,\infty)$ ($s_0$ some fixed number), and for all $\gamma \in W_{\Gamma}^{1,2}([0,1],\R^2) $ we have
    \[
    ||A(s)\gamma||_s \geq \delta ||\gamma||_s
    \]
    where here and in what follows, we use $||\cdot||_s$ to denote the norm defined by $(\cdot,\cdot)_s$.
\end{enumerate}

\end{proposition}

We also need the following proposition on Hilbert spaces, also stated and proved in \cite[Theorem 3.7]{abbas1999finite}.
\begin{proposition}
Let $T : D(T) \in H \rightarrow H$ be a self-adjoint operator in a Hilbert space $H$. Let $A_0:H \rightarrow H$ be a linear symmetric, bounded operator.  Let $\sigma(\cdot)$ denote the spectrum of an operator. Then
\begin{align}
    \op{dist}(\sigma(T),\sigma(T+A_0)) &:= max\{ \sup_{\lambda \in \sigma(T)}\op{dist}(\lambda, \sigma(T+A_0)), \sup_{\lambda \in \sigma(T+A_0)} \op{dist}(\lambda,\sigma(T))\}\\
    &\leq ||A_0||_{\mathcal{L}(H)}
\end{align}
Assume further that the resolvent $(T-\lambda_0)^{-1}$ of $T$ exists and is compact for some $\lambda_0 \not\in \sigma(T)$, then $(T-\lambda)^{-1}$ exists and is compact for every $\lambda \not \in \sigma (T)$, and $\sigma(T)$ consists of isolated eigenvalues $\{\mu_k\}_{k\in \Z}$ with finite multiplicities $\{m_k\}_{k\in \Z}$. 
If we assume $\sup_{k\in \Z}m_k \leq M < \infty$, and that for each $L>0$ there is a number $m_T(L)\in \N$ so that every interval $I \subset \R$ of length $L$ contains at most $m_T(L)$ points of $\sigma(T)$ (counted with multiplicity), then for each $L>0$ there is also a number $m_{T+A_0}(L)$ so that every interval $I\in \R$ of length $L$  contains at most $m_{T+A_0}(L)$ points of $\sigma(T+A_0)$.
\end{proposition}

To lighten the exposition we also isolate the following lemma from \cite{abbas1999finite}, whose proof depends on the above proposition.
\begin{lemma}\label{lem:spectrum of A}
Consider the operators $A_\infty: W^{1,2}_\Gamma([0,1],\R^2) \rightarrow L^2([0,1],\R^2)$ and 

$A(s):$ $ W^{1,2}_\Gamma([0,1],\R^2) \rightarrow L^2([0,1] ,\R^2)$. We have:
\begin{enumerate}
    \item Given any $L>0$, there is some positive integer $m$ so that any interval in $\R$ of length $L$ contains at most $m$ eigenvalues of $A_\infty$ (up to multiplicity).
    \item $\op{dist}(\sigma(A(s)),\sigma (A_\infty)) \rightarrow 0$ as $s\rightarrow \infty$.
    \item Let $I_n :=[-(n+1)L,-nL]$, then each $I_n$ contains at most $m$ points of $\sigma(A_\infty)$. There is a closed interval $J_n \subset I_n$ of length $2d>0$ so that for $s>s_0$, the interval $J_n$ does not contain any element of $\sigma(A(s))$. In the following, we will write $J_n$ to be of the form $[r_n-d,r_n+d]$.
\end{enumerate}
\end{lemma}
Let $\zeta(s,t) :[0,1]\times \R \rightarrow \R^2$ be a solution to equation \ref{eq:projection_to_contact_plane} that exponentially decays to zero as $s\rightarrow \infty$. Then we have (see \cite{abbas1999finite} Theorem 3.6):

\begin{proposition}\label{proposition:decay normal form}
If $\zeta$ does not vanish identically, then we have the following asymptotic formula:
\[
\zeta(s,t) = e^{\int_{s_0}^s \alpha(\tau) d\tau} (e(s,t) + r(s,t)),
\]
where
\begin{enumerate}
    \item $e(s,t)$ is an eigenvector of $A_\infty$ in $W_{\Gamma}^{1,2}([0,1],\R^2)$ with eigenvalue $\lambda <0$.
    \item $\alpha(s):[s_0,\infty) \times \R$ is a smooth function satisfying $\alpha(s)\rightarrow \lambda$ as $s\rightarrow \infty$.
    \item $r(s,t): [s_0,\infty) \times [0,1] \rightarrow \R$ is a smooth function with
    \[
    |\partial^\beta r(s,t)| \rightarrow 0
    \]
    as $s\rightarrow \infty$. Here $\beta \in \N^2$ is some multi-index.
\end{enumerate}
\end{proposition}

We break down the proof into a series of lemmas, as was done analogously in \cite{abbas1999finite}.
\begin{lemma}[Lemma 3.9 in \cite{abbas1999finite}]
If the assumptions of the above proposition are satisfied, then $\alpha(s)\rightarrow \lambda$, where $\lambda$ is a negative eigenvalue of $A_\infty$. 
\end{lemma}

\begin{proof}
We first assume $\zeta(s,t)\neq0$ for all $s\geq s_0$. The alternative case is handled later. 
We start with some preliminary manipulations. \\
\textbf{Step 0}
Define
\begin{equation}
    \alpha(s):= \frac{d/ds ||\zeta(s,t)||_s^2}{2||\zeta(s,t)||^2_s}
\end{equation}
Then trivially we can write
\[
\frac{d}{ds}||\zeta(s,t)||_s^2 = 2\alpha ||\zeta(s,t)||_s^2 
\]
from which we deduce
\[
||\zeta||_s^2 = e^{2\int_{s_0}^s\alpha ds'} ||\zeta(s_0)||_s^2.
\]
Hence we need to establish the asymptotics of $\alpha$.
We define
\[
\xi(s,t) : =\frac{\zeta(s,t)}{||\zeta(s,t)||_s}.
\]
Consequently
\[
\partial_t \xi = \frac{\partial_t \zeta}{||\zeta||_s}
\]
and 
\[
\partial_s \xi = \frac{\partial_s \zeta}{||\zeta||_s} - \frac{\zeta}{2||\zeta||_s^3}\frac{d}{ds}||\zeta||_s^2.
\]
Recall that $\zeta$ satisfies the equation \ref{eq:projection_to_contact_plane}, we have
\[
\partial_s \xi +\alpha \xi + M(s,t)\partial_t \xi + \Delta(s,t) \xi =0
\]
which we can further rewrite as
\[
\nabla_s \xi -\Gamma_1 \xi +\alpha \xi +M(s,t)\partial_t \xi + \Delta(s,t)\xi =0.
\]
Taking the $(\cdot,\cdot)_s$ innter product of the above with $\xi$, we obtain
\[
\alpha(s) = (\Gamma_1\xi,\xi)_s - (M\partial_t \xi,\xi)_s - (\Delta \xi,\xi)_s
\]
where we used $(\nabla_s \xi,\xi)_s=0$. Taking the $s$ derivative of both sides, we obtain
\begin{align*}
    \alpha'(s) =& -(\nabla_s(M(s,t)\partial_t\xi),\xi)_s - (M\partial_t\xi,\nabla_s \xi)_s + (\nabla_s(\Gamma_1 \xi),\xi)_s + (\Gamma_1\xi,\nabla_s \xi) \\
    &- (\nabla_s(\Delta(s,t)\xi, \xi)_s - (\Delta(s,t)\xi,\nabla_s\xi)_s\\
    =&-(M\partial_t \nabla_s \xi,\xi)_s + (M \partial_t \Gamma_1\xi,\xi)_s  +(A\xi,\nabla_s\xi)_s+ ((\nabla_s\Gamma_1)\xi,\xi)_s  + 2(\Gamma_1\xi,\nabla_s\xi)_s\\
    &-((\nabla_s \Delta(s,t) \cdot \xi,\xi)_s + (\Delta(s,t) \cdot \nabla_s\xi,\xi)_s + (\Delta(s,t) \xi,\nabla_s \xi)_s)\\
    =&2(\nabla_s\xi,A\xi) +2 (\Gamma_1\xi,\nabla_s\xi)_s + (M\partial_t\Gamma_1\xi,\xi)_s + ((\nabla_s\Gamma_1)\xi,\xi)_s\\
    & -((\nabla_s \Delta(s,t) \cdot \xi,\xi)_s + (\Delta(s,t) \cdot \nabla_s\xi,\xi)_s + (\Delta(s,t) \xi,\nabla_s \xi)_s)
\end{align*}
where in the passage from first line to second line we used properties of $\nabla$, as well as the fact that
\[
(\Gamma_1\nabla_s \xi,\xi)_s = (\nabla_s \xi,\Gamma_1 \xi)_s 
\]
which follows from the definition of $(\cdot,\cdot)_s$ as follows:
\begin{align*}
    (\Gamma_1 \nabla_s\xi,\xi)_s &= \int_0^1\la -\frac{1}{2}M\partial_s M \nabla_s \xi,-J_0M \xi\ra dt\\
    &= -\frac{1}{2} \int_0^1 \la \partial_sM \nabla_s \xi,-J_0\xi\ra dt \\
    & = \frac{1}{2} \int_0^1 \la \nabla_s \xi, \partial_s M ^T J_0 \xi\ra dt\\
    & = \frac{1}{2} \int_0^1 \la \nabla_s \xi, -\partial_s(M^TJ_0M M)\xi\ra dt\\
    &=-\frac{1}{2} \int_0^1 \la \nabla_s \xi, J_0 \partial_sM\xi\ra dt\\
    &=\int_0^1\la \nabla_s\xi, -J_0M (-\frac{1}{2}M \partial_sM \xi)\ra dt.
\end{align*}
Now inserting 
\[
A\xi = \nabla_s \xi -\Gamma_1\xi +\alpha\xi+\Delta(s,t)\xi
\]
into the expression for $\alpha'$, we obtain:
\begin{align*}
    \alpha'(s) =& 2(\nabla_s\xi,\nabla_s\xi)_s - 2(\nabla_s\xi,\Gamma_1\xi)+2(\nabla_s\xi,\Delta(s,t)\xi) + 2(\nabla_s\xi,\Gamma_1\xi)_s\\
    &+(M\partial_t\Gamma_1\xi,\xi)_s + ((\nabla_s\Gamma_1)\xi,\xi)_s\\
    & -((\nabla_s \Delta(s,t) \cdot \xi,\xi)_s + (\Delta(s,t) \cdot \nabla_s\xi,\xi)_s + (\Delta(s,t) \xi,\nabla_s \xi)_s)\\
    =&2(\nabla_s\xi,\nabla_s\xi)_s+(\nabla_s\xi,\Delta(s,t)\xi)_s+(M\partial_t\Gamma_1\xi,\xi)_s + ((\nabla_s\Gamma_1)\xi,\xi)_s\\
    & -((\nabla_s \Delta(s,t) \cdot \xi,\xi)_s + (\Delta(s,t) \cdot \nabla_s\xi,\xi)_s ).
\end{align*}
We observe the following bounds
\[
(\nabla_s\xi,\Delta(s,t)\xi)_s \leq \epsilon(s) \{(\nabla_s\xi,\nabla_s\xi)_s +(\xi,\xi)_s\}.
\]
where $\epsilon(s)$ is a function that decays exponentially to zero. 
Likewise
\[
|(\Delta(s,t) \cdot \nabla_s\xi,\xi)_s|\leq \epsilon(s) \{(\nabla_s\xi,\nabla_s\xi)_s +(\xi,\xi)_s\}.
\]
Similarly the terms $(M\partial_t\Gamma_1\xi,\xi)_s$, $((\nabla_s\Gamma_1)\xi,\xi)_s$, and $((\nabla_s \Delta(s,t) \cdot \xi,\xi)_s$ have absolute value uniformly upper bounded by $\epsilon(s)$ as well. Hence we arrive at the following inequality for $\alpha'(s)$:
\begin{equation}
    \alpha'(s) \geq (2-\epsilon(s)) (\nabla_s\xi,\nabla_s\xi)_s -\epsilon(s).
\end{equation}
We note our precise definition of $\epsilon(s)$ may change from line to line, but it will always denote a function that exponentially decays to zero.
With the above inequality established, from this point onward we can then repeat the arguments in \cite{abbas1999finite}.

\vspace{3pt}

\textbf{Step 1.} We first show that $\alpha(s)$
is bounded above. Suppose not, then we can find a sequence
\[\{s_k\}_k \rightarrow \infty \qquad\text{satisfying}\qquad \alpha(s_k) \rightarrow \infty\]
If we had $\alpha(s) \geq \eta$ for some $\eta >0$ for all $s$ large enough, then by definition of $\alpha$ this implies:
\[
\frac{d}{ds}||\zeta||_s^2 \geq 2\eta||\zeta||_s^2 
\]
which would imply (assuming $||\zeta||_s$ is nonzero for large enough $s$) $||\zeta||_s \rightarrow \infty$ as $s\rightarrow \infty$, a contradiction. 

The above argument gives us that for any $\eta >0$, we can find $s_k'\rightarrow \infty$ so that $\alpha(s_k')<\eta$. Now choose $\eta\in(0,\delta)$, where $\delta$ is given by Proposition \ref{proposition:some properties of A}. Let $\hat{s}_k$ be the smallest number satisfying $\hat{s}_k>s_k$ and $\alpha(\hat{s}_k) = \eta$. Observe this $\eta$ cannot be an eigenvalue of any $A(s)$, because for any $\gamma \in W^{1,2}_\Gamma([0,1],\R^2)$, we have
\[
||A(s)\gamma - \eta \gamma|| \geq ||A(s)\gamma|| - \eta ||\gamma|| \geq (\delta-\eta)||\gamma||>0.
\]
Thus
\begin{align*}
    ||\nabla_s \xi(\hat{s}_k)||_{\hat{s}_k} & \geq ||A(\hat{s}_k)\xi(\hat{s}_k) -\eta \xi(\hat{s}_k)||_{\hat{s}_k} - \epsilon(s) ||\xi(\hat{s}_k)||_{\hat{s}_k}\\
    & \geq (\delta-\eta) -\epsilon(s)\\
    &\geq \tau
\end{align*}
for some $\tau>0$. 
Consequently
\[
\alpha'(\hat{s}_k) \geq (2-\epsilon(\hat{s}_k))||\nabla_s \xi (\hat{s}_k)||_{\hat{s}_k}^2 -\epsilon(\hat{s}_k)\geq \tau^2
\]
for large enough $k$. This is a contradiction because we picked $\hat{s}_k$ to be the smallest number greater than $s_k$ such that $\alpha(\hat{s}_k) =\eta$, and that $\alpha(s_k)>\eta$. The fact $\alpha'(\hat{s}_k)>0$ implies there is some number in $[s_k,\hat{s}_k)$ such that $\alpha <\eta$, which is a contradiction.

\vspace{3pt}

\textbf{Step 2.} We next proceed to show $\alpha$ is also bounded from below. Assume not, then we can find a sequence $s_n\rightarrow \infty$ so that $\alpha(s_n)=r_n$ and $\alpha'(s_n) \leq 0$, where $r_n$ and $d$ are given in Lemma \ref{lem:spectrum of A}. Then we consider 
\[
A(s_n)\xi(s_n) = \nabla_s \xi (s_n) -\Gamma_1\xi (s_n) +\alpha\xi(s_n)+\Delta(s_n,t)\xi(s_n)
\]
where both sides are functions of $t$. We have
\begin{align*}
||\nabla_s \xi(s_n)||_{s_n} &= ||A\xi(s_n)-r_n\xi(s_n)+\Gamma_1\xi(s_n)-\Delta(s_n,t)\xi(s_n)||_{s_n}\\
&\geq d-\epsilon(s_n)\\
&\geq d/2
\end{align*}
for large enough $n$. Combining this with
\begin{equation}
    \alpha'(s) \geq (2-\epsilon(s)) (\nabla_s\xi,\nabla_s\xi)_s -\epsilon(s)
\end{equation}
we conclude $\alpha'(s_n)\geq d^2/4$, which is contrary to our assumptions. Hence $\alpha$ is bounded from below as well.

\vspace{3pt}

\textbf{Step 3.}
We observe we can always find a sequence $\{s_k\}$ so that $||\nabla_s\xi(s_k)||_{s_k} \rightarrow 0$. Suppose not, then we can find $\eta>0$ so that
\[
||\nabla_s\xi||_s\geq \eta
\]
for large enough $s$, which implies 
\[
\alpha'(s) \geq \eta^2
\]
contradicting the claim $\alpha(s)$ is bounded.

\vspace{3pt}

\textbf{Step 4.}
Because $\alpha$ is bounded, we can find a subsequence of $\{s_k\}$, which we also denote by $\{s_k\}$ so that
\[
\alpha (s_k) \rightarrow \lambda.
\]
As in \cite{abbas1999finite}, we first show $\lambda \in \sigma(A_\infty)$. Suppose not, let $\epsilon>0$ be defined by
\[
\epsilon := \op{dist}(\lambda, \sigma(A_\infty)).
\]
Then for also large enough $s$ we have:
\[
\op{dist}(\lambda,\sigma(A(s)))\geq \epsilon/2.
\]
Then as before we can compute:
\begin{align*}
    ||\nabla_s \xi(s_k)||_{s_k} & \geq ||A(s_k)\xi(s_k) -\alpha(s_k) \xi(s_k)||_{s_k} - \epsilon(s_k) ||\xi(_k,\cdot)||_{s_k}\\
    & \geq \epsilon/4 -\epsilon(s_k)\\
    &\geq \epsilon/8
\end{align*}
for large enough $k$. This contradicts $||\nabla_s \xi(s_k)||_s\rightarrow 0$, and hence $\lambda \in \sigma(A_\infty)$.

\vspace{3pt}

\textbf{Step 5.} Finally we show $\lim_{s\rightarrow \infty} \alpha(s) = \lambda$. Suppose not, then we can find $\{s_k'\}$ so that
\[\lim_{k\rightarrow \infty} \alpha(s_k')\rightarrow \mu\]
Without loss of generality, let us assume $\mu < \lambda$. Then we can find some $\nu \in (\mu,\lambda)$ and $d>0$ so that
\[
\op{dist}(\nu, \sigma(A(s))) \geq d
\]
for all $s$ sufficiently large. Then for any $\hat{s}$ so that $\alpha(\hat{s}) = \nu$, we have:
\[
||\nabla_s \xi(\hat{s})||_{\hat{s}} \geq d/2
\]
as before, but this implies $\alpha'(\hat{s})\geq d^2/4> 0$, which implies for large enough $s$ we have $\alpha(s)\geq \nu$, which is a contradiction. A very similar argument can be applied in case $\mu>\lambda$. We have therefore $\alpha(s) \rightarrow \lambda$. We claim $\lambda <0$ because the norm of $\zeta$ exponentially decays to zero as $s\rightarrow \infty$.

\vspace{3pt}

\textbf{Step 6} Finally we return to the case where for some $s^*$ we have $||\zeta(s^*,\cdot)||_{s^*}=0$, this would then imply $\zeta(s^*,t^*)=0$ for all $t^*$, and the Carleman similarity principle then implies $\xi$ is zero everywhere.
\end{proof}

\begin{lemma}
With the same notation as above. Let $\beta = (\beta_1,\beta_2) \in \N^2$, and $j \in \N$. We have
\[
sup_{(s,t) \in [s_0,\infty) \times [0,1]} |\partial^\beta \xi(s,t) | < \infty
\]
\[
sup_{s\in [s_0,\infty)} |\frac{d^j \alpha}{ds^j}(s)| < \infty.
\]
\end{lemma}
\begin{proof}
The proof is almost verbatim the proof of lemma 3.10 in \cite{abbas1999finite}. The only difference in our case is the appearance of the term $\Delta(s,t) \xi$, but in the notation of lemma 3.10 of Abbas this can be absorbed into $\hat{\alpha}\xi$, and using the fact $\Delta(s,t)$ and its derivatives decay exponentially, we can recover bounds of derivatives of $\xi$ and $\alpha$.
\end{proof}

\begin{proposition}
Let $E$ denote the eigenspace of the operator $A_\infty$ in $W^{1,2}_\Gamma([0,1],\R^2)$ corresponding to eigenvalue $\lambda$, then 
\[
\op{dist}(\xi(s),E)\rightarrow 0
\]
as we take $s\rightarrow \infty$. The notion of distance is defined with respect to $W^{1,2}_\Gamma([0,1],\R^2)$.
\end{proposition}
\begin{proof}
The same proof as in Lemma 3.6 in \cite{hofer1996properties}.
\end{proof}

\begin{lemma}
There exists $e\in E$ so that $\xi(s)\rightarrow e$ as $s\rightarrow \infty$.
\end{lemma}
\begin{proof}
The proof is the same as the proof of Lemma 3.12 in \cite{abbas1999finite}, as such we will only sketch the notable differences between our proof and theirs. As in their case for every sequence $\{s_n\}$ that converges to infinity, we can extract a subsequence $\{s_n'\}$ so that $\xi(s_n') \rightarrow e \in E$. Hence it remains to show if we have sequences $\{s_n\}$ and $\{\tau_n\}$ so that
\begin{itemize}
    \item $\xi(s_n)\rightarrow e \in E$,
    \item $\xi(\tau_n)\rightarrow e'\in E$
\end{itemize}
then we have $e=e'$. To this end, consider the inner product on $L^2([0,1],\R^2)$ given by:
\[
(u_1,u_2):=\int_0^1\la u_1(t),-J_0M_\infty u_2(t)\ra dt
\]
Let $P$ denote the orthogonal projection to $E$ with respect to the above inner product, and we define
\[
\hat{\xi}:=P\xi.
\]
Then it follows as in the proof in \cite{abbas1999finite} that:
\begin{itemize}
    \item $A_\infty P\xi =P A_\infty \xi$.
    \item If we define $\epsilon'(s) := A(s)-A_\infty$, we have the equation 
    \[
    \partial_s \hat{\xi} =(\lambda-\alpha(s))\hat{\xi} +P\epsilon'(s)\xi -P\Delta \xi(s)
    \]
    \item We have that both $||P\epsilon'(s)\xi||$ and $||P\Delta(s,t)\xi(s)||$ are bounded above by $Ce^{-\rho s}$.
    \item We have the norm bounds:
    \[
    ||\xi(s)||\rightarrow 1,
    \]
    \[
    ||\hat{\xi} -\xi(s)||  \rightarrow  0,
    \]
    which in particular implies for large enough $s$ we have
    \[
    ||\hat{\xi}|| \geq 1/2.
    \]
\end{itemize}

With the above we define
\[
\eta := \frac{\hat{\xi}(s,t)}{||\hat{\xi}(s,t)||}
\]
Then we have the equation
\begin{align*}
\partial_s \eta &= \frac{\partial_s \hat{\xi}}{||\hat{\xi}||} - \frac{(\hat{\xi},\partial_s \hat{\xi(s)})}{||\hat{\xi}||^3} \hat{\xi}\\
&= \frac{(\lambda -\alpha(s))\hat{\xi} +P\epsilon' \xi -P\Delta \xi}{||\hat{\xi}||} - \frac{(\hat{\xi},(\lambda -\alpha(s))\hat{\xi} +P\epsilon' \xi -P\Delta \xi)}{||\hat{\xi}||^3} \hat{\xi}\\
& = \frac{P\epsilon' \xi -P\Delta \xi}{||\hat{\xi}||} - \frac{(\hat{\xi}, P\epsilon' \xi -P\Delta \xi)}{||\hat{\xi}||^3} \hat{\xi}
\end{align*}

From the above norm estimates we obtain
\[
|\partial_s \eta | \leq e^{-\rho s}
\]
where the above is the genuine absolute value. Then we have
\begin{align*}
|\eta(s_n) - \eta(\tau_n)| &\leq \left|\int_{\tau_n}^{s_n} \partial_s \eta ds \right|\\
& \leq \left|\int_{\tau_n}^{s_n} e^{-\rho s} ds\right |\\
& \rightarrow 0
\end{align*}
as $n\rightarrow \infty$. This concludes the proof of our lemma.
\end{proof}

\begin{proof}[Proof of Proposition \ref{proposition:decay normal form}]
With the previous lemmas assembled, the proof of Proposition \ref{proposition:decay normal form} follows verbatim to that of Proposition 3.6 in \cite{abbas1999finite}. We use Sobolev embedding and induction to show that all derivatives of $r(s,t)$ converge uniformly to zero. We use Arzela-Ascoli to show that $\alpha(s)$ approaches $\lambda$ in the $C^\infty$ norm.
\end{proof}

We are now ready to apply the above estimates to analyze the difference of two ends asymptotic to the same Reeb orbit. Namely, we prove the following:
\begin{proposition}\label{prop:ends difference}
For two different ends asymptotic to the same Reeb orbit with local coordinates:
\[
    U(s,t)=(b(s,t),x(s,t),y(s,t),z(s,t))
\]
and
\[
    V(s,t)=(b(s,t),x(s,t)+\eta_1(s,t),y(s,t)+\eta_2(s,t),z(s,t)-fc_1\eta_1(s,t)-fc_2\eta_2(s,t)).
\]
There are reparametrizations $\phi_u,\phi_v:[R,\infty)\times [0,1]\to [R,\infty)\times [0,1]$ for sufficiently large $R$, such that
\[
    U(\phi_u(s,t))=(s,U_0(s,t),t),
\]
\[
    V(\phi_v(s,t))=(s,V_0(s,t),t),
\]
and
\[
    V_0(s,t)-U_0(s,t)=e^{\lambda s}(e(t)+r(s,t))
\]
where $e(t)$  is an  eigenvector of the asymptotic operator with eigenvalue $\lambda$, and $r(s,t)$ decays exponentially to zero as $s\to\infty$.

\end{proposition}

\begin{proof}[Proof of Proposition \ref{prop:ends difference}]
Fix a sufficiently large positive number $R$, and $\epsilon>0$ sufficiently small. Let $D_\epsilon$ be the $\epsilon$-disk around the origin in $\mathbb{R}^2$, and $E:[R,\infty)\times D_\epsilon\times [0,1]\to \mathbb{R}\times \mathbb{R}^3$ be the embedding:
\begin{equation}
    E(s,h_1,h_2,t)=(b(s,t),x(x,t)+h_1,y(s,t)+h_2,z(s,t)-c_1fh_1-c_2fh_2)
\end{equation}
where, as before, we shorthand $c_1(u(s,t))$, $c_2(u(s,t))$ and $f(u(s,t))$ as $c_1$, $c_2$ and $f$ respectively. Now the first end $U$ is parameterized as:
\[
    U(s,t)=E(s,0,0,t)
\]
and since any different end that is asymptotic to the same Reeb chord can be viewed as an embedded surface in the image of $E$, $V$ can be parametrized as
\[
    V(s,t)=E(s,\eta_1(s,t),\eta_2(s,t),t).
\]
Since both ends limit to the Reeb chord exponentially, there are reparametrizations $\phi_u,\phi_v:[R,\infty)\times [0,1]\to[R,\infty)\times S^1$ such that 
\[
    U(\phi_u(s,t))=(s,U_0(s,t),t),
\]
\[
    V(\phi_v(s,t))=(s,V_0(s,t),t).
\]
Let $\lambda$ be the negative eigenvalue that appears in the asymptotic expansion of $$\eta(s,t)=(\eta_1(s,t),\eta_2(s,t))$$ provided by Proposition \ref{proposition:decay normal form}, and $e(t)$ be the corresponding eigenvector. We have:
$$
    \phi_u(s,t)-\phi_v(s,t)=o_\infty(\lambda)
$$
Here we use the same notation as in \cite{siefring2008relative}. Namely, a function $f:[R,\infty)\times [0,1]\to \mathbb{R}^2=o_\infty(\lambda)$ if and only if there are constants $\delta>0$ and $M_\beta>0$ for every multi-index $\beta$, such that 
$$
    |\partial^\beta(e^{\lambda s} f)|\le M_\beta e^{-\delta s}
$$

The reason is that, by definition 
$$
\phi_u^{-1}(s,t)-\phi_v^{-1}(s,t)=(0,-c_1(u(s,t))f(u(s,t))\eta_1(s,t)-c_2(u(s,t))f(u(s,t))\eta_2(s,t)),
$$
therefore $\phi_u^{-1}(s,t)-\phi_v^{-1}(s,t)=o_\infty(\lambda)$, and hence $\phi_u(s,t)-\phi_v(s,t)=o_\infty(\lambda)$. Now let $\pi:\mathbb{R}^4\to \mathbb{R}^2$ be the projection to the second and third coordinates. We have: 
\begin{align*}  
    V_0(s,t)-U_0(s,t)&=\pi(V(\phi_v(s,t))-U(\phi_v(s,t)))\\  
    &=\pi(E(\phi_v(s,t),\eta(\phi_v(s,t)))-E(\phi_u(s,t),0,0))\\ 
    &=\pi(E(\phi_u(s,t),\eta(\phi_u(s,t)))-E(\phi_u(s,t),0,0))+o_\infty(\lambda)\\
    &=\eta(\phi_u(s,t))+o_\infty(\lambda)\\
    &=e^{\lambda s}e(t)+o_\infty(\lambda)
\end{align*}

\end{proof}

The above discussion finishes the proof of Theorem \ref{thm:relative_asymptotic_behaviour}. To see this fact, choose an Riemannian metric on $Y$ so that near the Reeb chord $\eta$ with the chosen local coordinates, the metric agrees with the standard Riemannian metric on $\R^3$. Extend this metric to an $\R$-invariant metric on $\R\times Y$, and we see by definition that the functions $U_0$ and $V_0$ are asymptotic representatives of the two ends (see Definition 2.1 of \cite{siefring2008relative}). The above Proposition shows that there is a section $r(s,t)$ of $\eta^*\xi$ that exponentially decays to zero as $s$ tends to $\infty$, so that under the chosen local coordinates the two asymptotic representatives satisfy
\[
U_0(s,t)-V_0(s,t)=e^{\lambda s}(e(t)+r(s,t)).
\]
Having established the above proposition, the exact same argument as in Section 4 in \cite{siefring2008relative} proves the following:
\begin{proposition}
(Compare Theorem 2.4 in \cite{siefring2008relative}) Let $\gamma$ be a Reeb chord connecting two Legendrians $\Gamma_0$ and $\Gamma_1$, with action $1$, and fix a trivialization $\tau$ of $\gamma^*\xi$. Let $\{u_i\}_{i=1}^n$ be a collection of pseudo-holomorphic curves with $\gamma$ as a positive (resp. negative) end. Then there exist a neighborhood $U$ of $\gamma$, a smooth embedding $\Phi: \R\times U\to \R\times \R^2\times [0,1]$ with the property
\[
\Phi(s,\gamma(t))=(s,0, 0, t),
\]
proper reparametrizations $\psi_i:[R,\infty)\times [0,1]\to \R\times[0,1]$ asymptotic to the identity, and positive integers $N_i$, such that near the asymptotic ends
\[
\Phi\circ u_i\circ\psi_i(s,t)=(s,t,\sum_{k=1}^{N_i}e^{\lambda_{i,j}s}e_{i,j}(t))
\]
where $\lambda_{i,j}$ are negative (resp. positive) eigenvalues of the asymptotic operator associated to $\gamma$ and $\tau$, and $e_{i,j}$ are eigenvectors with eigenvalue $\lambda_{i,j}$.
\end{proposition}

\printbibliography

\end{document}